\documentclass[11pt,reqno,a4paper]{amsart}

\usepackage[T1]{fontenc}
\usepackage[utf8]{inputenc}

\usepackage{libertinus}
\usepackage{microtype}
\usepackage{geometry}
\usepackage{amsmath,amssymb,amsthm,amsfonts}
\usepackage{mathtools}
\usepackage{mathrsfs}
\usepackage{bbm}

\numberwithin{equation}{section}

\usepackage{graphicx}
\usepackage{xcolor}
\usepackage{enumitem}
\setlist{itemsep=2pt,topsep=4pt}

\usepackage[
  colorlinks=true,
  linkcolor=blue,
  citecolor=blue,
  urlcolor=blue,
  pagebackref=true
]{hyperref}

\renewcommand*{\backrefalt}[4]{%
  \ifcase #1\relax
  \or
    {\footnotesize Cited on page #2.}%
  \else
    {\footnotesize Cited on pages #2.}%
  \fi
}

\usepackage[nameinlink,noabbrev]{cleveref}
\crefname{section}{Section}{Sections}
\Crefname{section}{Section}{Sections}
\crefname{subsection}{Subsection}{Subsections}
\Crefname{subsection}{Subsection}{Subsections}

\usepackage{titlesec}

\titleformat{\section}
  {\normalfont\large\bfseries\centering}
  {\thesection.}
  {0.75em}
  {}

\titlespacing*{\section}
  {0pt}
  {3.2ex plus 1ex minus .2ex}
  {1.6ex plus .3ex}

\titleformat{\subsection}
  {\normalfont\normalsize\bfseries}
  {\thesubsection.}
  {0.75em}
  {}

\titlespacing*{\subsection}
  {0pt}
  {2.4ex plus .8ex minus .2ex}
  {1.0ex plus .2ex}

\titleformat{\subsubsection}
  {\normalfont\normalsize\bfseries\itshape}
  {\thesubsubsection.}
  {0.75em}
  {}

\titlespacing*{\subsubsection}
  {0pt}
  {2.0ex plus .6ex minus .2ex}
  {0.7ex plus .2ex}

\usepackage{aliascnt}
\usepackage{etoolbox}

\newtheoremstyle{compact}
  {6pt}
  {1pt}
  {\itshape}
  {}
  {\bfseries}
  {.}
  {0.5em}
  {}

\theoremstyle{compact}
\newtheorem{theorem}{Theorem}[section]

\newaliascnt{proposition}{theorem}
\newtheorem{proposition}[proposition]{Proposition}
\aliascntresetthe{proposition}

\newaliascnt{lemma}{theorem}
\newtheorem{lemma}[lemma]{Lemma}
\aliascntresetthe{lemma}

\newaliascnt{corollary}{theorem}
\newtheorem{corollary}[corollary]{Corollary}
\aliascntresetthe{corollary}

\newtheorem{maintheorem}{Theorem}

\theoremstyle{definition}
\newaliascnt{definition}{theorem}
\newtheorem{definition}[definition]{Definition}
\aliascntresetthe{definition}

\newtheorem*{example}{Example}

\theoremstyle{definition}
\newtheorem{remark}[theorem]{Remark}

\AtBeginEnvironment{proof}{\vspace{6pt}}

\makeatletter
\AtBeginEnvironment{thebibliography}{%
  \def\@mklab#1{#1\hfil}%
}
\makeatother

\newcommand{\HU}{\mathrm{HU}}
\newcommand{\Var}{\operatorname{Var}}
\newcommand{\Cov}{\operatorname{Cov}}
\newcommand{\Vol}{\operatorname{Vol}}
\newcommand{\cM}{\mathcal M}
\newcommand{\Sh}{\operatorname{Sh}}
\newcommand{\sigmak}{\smash[t]{\sigma^{(k)}}}
\newcommand{\bsigmak}{\smash[t]{\boldsymbol{\sigma}^{(k)}}}

\begin{document}

\title{Higher-order hyperuniformity of random measures}

\author{Michael Bj\"orklund}
\address{Department of Mathematics, Chalmers University of Technology and University of Gothenburg, Gothenburg, Sweden}
\email{micbjo@chalmers.se}


\begin{abstract}
We study higher-order hyperuniformity through the large-scale variance of
local $k$-point patterns, and write $\HU_k$ for the resulting condition;
$\HU_1$ is ordinary hyperuniformity.

The conditions $\HU_k$ need not coincide: for every $k\geq1$ there exists an
$\mathbb R$-invariant weakly mixing simple point process that belongs to
$\HU_j$ for all $j\leq k$ but not to $\HU_{k+1}$.  Randomly translated
lattices belong to $\HU_k$ for every $k$, whereas sufficiently small
non-degenerate iid perturbations of lattices and projection determinantal
point processes already belong to $\HU_1\setminus\HU_2$.

For regular Euclidean cut-and-project processes we give an exact criterion
for $\HU_k$.  For ball windows, $\HU_k$ is equivalent to $\HU_1$ in internal
dimensions two and three, whereas in every internal dimension $m\geq4$ we
construct ball-window examples in $\HU_1\setminus\HU_2$.  Finally, the nonperiodic Kurasov--Sarnak Fourier-quasicrystalline point process
is stealthy---its first-order spectrum has a gap at the origin---yet does not
belong to $\HU_2$.
\end{abstract}

\maketitle

\section{Introduction}
\label{sec:introduction}

Let $V$ be a finite-dimensional real Euclidean vector space of dimension
$d$, let $B_R\subset V$ be the closed ball of radius $R$ centred at the
origin, and let $\Vol_d$ denote Lebesgue measure.  We study stationary random
measures on $V$, represented by $V$-invariant Borel probability measures
$\mu$ on the space $\cM(V)$ of locally finite positive Radon measures.  When
$\mu$ is supported on simple counting measures, we call it a \emph{point
process}.

For a bounded complex-valued Borel function $f$ with compact support, define
$Sf(p):=p(f)$; we call $Sf$ the \emph{first statistic}.  We call $\mu$ \emph{hyperuniform} if
\begin{equation}\label{eq:intro-hu1}
  \lim_{R\to\infty}
  \frac{\Var_\mu(S\chi_{B_R})}{\Vol_d(B_R)}=0.
\end{equation}
The terminology was introduced by Torquato and Stillinger
\cite{TS03}; see \cite{Tor18} for a broad account and \cite{BB26} for the
first-order theory of invariant random measures.

In this paper, we replace the first statistic $Sf$ by statistics depending on
local $k$-point configurations.  The index $k$ refers to the number of points
in the configuration, not to the order of a moment or cumulant.

For $v=(v_1,\ldots,v_k)\in V^k$, write
\[
  \bar v:=\frac1k\sum_{j=1}^k v_j,
  \qquad
  v^\circ:=(v_1-\bar v,\ldots,v_k-\bar v).
\]
Let $\Sh_k(V)$ denote the space of centred $k$-tuples modulo coordinate
permutations, and write $[v^\circ]$ for the class of $v^\circ$; the precise
quotient topology is fixed in Section~\ref{subsec:pattern-statistics}.  Here
$C_c(Y)$ denotes the real-valued continuous functions on $Y$ with compact
support.

\begin{samepage}
For $f\in C_c(V)$ and $\varphi\in C_c(\Sh_k(V))$, define
\begin{equation}\label{eq:intro-pattern-kernel}
  F_{f,\varphi}(v_1,\ldots,v_k)
  :=f(\bar v)\,\varphi([v^\circ])
\end{equation}
and
\begin{equation}\label{eq:intro-pattern-statistic}
  T_k[f,\varphi](p):=p^{\otimes k}(F_{f,\varphi}).
\end{equation}
\end{samepage}
For a simple counting measure $p=\sum_{x\in\Lambda}\delta_x$, this is the
sum of $F_{f,\varphi}(v_1,\ldots,v_k)$ over all ordered $k$-tuples
$v_1,\ldots,v_k\in\Lambda$, without requiring the entries of the tuple to
be distinct.

For $k=2$, define the even function
\[
  \widetilde\varphi^{(2)}(r)
  :=\varphi\!\left(\left[\left(\frac r2,-\frac r2\right)\right]\right),
  \qquad r\in V.
\]
Then
\[
  T_2[f,\varphi](p)
  =\sum_{x,y\in\Lambda}
  f\!\left(\frac{x+y}{2}\right)\widetilde\varphi^{(2)}(x-y).
\]
Thus $T_2[f,\varphi]$ records pairs according to their relative displacement,
while $f$ localizes their barycentres.

For $k=3$, set
\[
  \widetilde\varphi^{(3)}(r,s)
  :=\varphi\!\left(
  \left[
  \left(
  -\frac{r+s}{3},
  \frac{2r-s}{3},
  \frac{2s-r}{3}
  \right)
  \right]\right),
  \qquad r,s\in V.
\]
Then
\[
  T_3[f,\varphi](p)
  =\sum_{x,y,z\in\Lambda}
  f\!\left(\frac{x+y+z}{3}\right)
  \widetilde\varphi^{(3)}(y-x,z-x).
\]
Thus a three-point shape is described by two relative displacement vectors,
with the permutation symmetry of the three points encoded in
$\widetilde\varphi^{(3)}$; for example,
\[
  \widetilde\varphi^{(3)}(r,s)
  =\widetilde\varphi^{(3)}(s,r)
  =\widetilde\varphi^{(3)}(-r,s-r).
\]

We say that $\mu$ is \emph{$k$-hyperuniform} if
\begin{equation}\label{eq:intro-huk}
  \lim_{R\to\infty}
  \frac{\Var_\mu(T_k[\chi_{B_R},\varphi])}{\Vol_d(B_R)}=0
  \qquad\text{for every }\varphi\in C_c(\Sh_k(V)).
\end{equation}
We write $\HU_k$ for this class.  At $k=1$ this is ordinary
hyperuniformity.

\subsection{Bartlett spectra}
\label{subsec:intro-bartlett}

For fixed $\varphi$, define
$\mathcal P_{k,\varphi}(p)(f):=T_k[f,\varphi](p)$.  A common translation of
the $k$ points leaves $[v^\circ]$ unchanged and translates $\bar v$ by the
same amount, so
$\mathcal P_{k,\varphi}(a.p)=a.\mathcal P_{k,\varphi}(p)$.  Hence, under a
$V$-invariant law, $\mathcal P_{k,\varphi}$ is a stationary signed random
measure on $V$ and has an ordinary scalar Bartlett spectrum.  As $\varphi$
varies, these spectra and their cross-covariances assemble into the
\emph{$k$-th Bartlett spectrum} $\boldsymbol{\sigma}^{(k)}$.  Its coefficient
measures satisfy
\begin{equation}\label{eq:intro-bartlett}
  \Cov_\mu\bigl(T_k[f,\varphi],T_k[g,\psi]\bigr)
  =\int_{V^*}\widehat f(\xi)\overline{\widehat g(\xi)}\,
    d\sigma^{(k)}_{\varphi,\psi}(\xi),
\end{equation}
where the Fourier convention and the weak operator-valued construction are given
in Section~\ref{subsec:bartlett}.  Write $B_\varepsilon^*\subset V^*$ for the
closed dual ball of radius $\varepsilon$.

\begin{maintheorem}[Geometric--spectral equivalence]
\label{thm:intro-geometric-spectral}
Let $\mu$ be a $V$-invariant probability measure on $\cM(V)$ which is locally
$L^{2k}$-integrable.  Then
\begin{equation}\label{eq:intro-geometric-spectral}
  \mu\in\HU_k
  \quad\Longleftrightarrow\quad
  \lim_{\varepsilon\to0}
  \frac{\sigma^{(k)}_{\varphi,\varphi}(B_\varepsilon^*)}
       {\Vol_d(B_\varepsilon^*)}=0
  \quad\text{for every }\varphi\in C_c(\Sh_k(V)).
\end{equation}
\end{maintheorem}

For fixed $\varphi$, the diagonal coefficient
$\sigma^{(k)}_{\varphi,\varphi}$ is the ordinary Bartlett spectrum of this
random measure.

\subsection{Main results}
\label{subsec:intro-main-results}

The conditions $\HU_k$ need not coincide.  Randomly translated lattices
belong to $\HU_k$ for every $k$, while iid lattice perturbations and
projection determinantal point processes can already belong to
$\HU_1\setminus\HU_2$.

\begin{maintheorem}[Strictness under weak mixing]
\label{thm:intro-weak-mixing}
For every integer $k\geq1$ there exists an $\mathbb R$-invariant weakly
mixing simple point process $\mu$ on $\mathbb R$ such that
\begin{equation}\label{eq:intro-weak-mixing}
  \mu\in\HU_{\leq k}\setminus\HU_{k+1},
  \qquad
  \HU_{\leq k}:=\bigcap_{j=1}^k\HU_j.
\end{equation}
\end{maintheorem}

For lattices, Theorem~\ref{thm:lattice-all-order-stealthy} gives
$\mu\in\HU_k$ for every $k$.  In every dimension, sufficiently small
non-degenerate iid perturbations of a lattice belong to
$\HU_1\setminus\HU_2$; see Theorem~\ref{thm:perturbed-lattice-separation}.  The same separation holds
for projection determinantal point processes; see
Theorem~\ref{thm:projection-dpp-hu1-not-hu2}.  The first-order
hyperuniformity of projection determinantal processes is standard; the new
conclusion is failure of $\HU_2$.

We recall the cut-and-project construction used below.  Let $H$ be a locally
compact second countable abelian group and let $\Gamma<V\times H$ be a
lattice such that the projection of $\Gamma$ to $V$ is injective and its
projection to $H$ is dense.  We call $(V,H,\Gamma)$ a cut-and-project
scheme.  Writing $\gamma=(\gamma_V,\gamma_H)\in\Gamma$, a compact window
$W\subset H$ defines, on $\Omega:=\Gamma\backslash(V\times H)$,
\[
  (p_W)_{\Gamma+(v,h)}
  :=\sum_{\gamma\in\Gamma}
     \chi_W(\gamma_H+h)\,\delta_{\gamma_V+v}.
\]
The map $\omega\mapsto(p_W)_\omega$ is $V$-equivariant, and $\mu_W$ denotes
the pushforward of Haar probability on $\Omega$.  When $H$ is Euclidean, we
call a compact window $W$ \emph{regular} if
$W=\overline{\operatorname{int}W}$, its interior is nonempty, and its boundary
has Lebesgue measure zero.

\begin{maintheorem}[Cut-and-project characterization and ball-window threshold]
\label{thm:intro-cut-project}
Let $(V,H,\Gamma)$ be as above, with $H$ Euclidean, and let $W\subset H$ be
a regular window.  For
$\boldsymbol\gamma=(0,\gamma_2,\ldots,\gamma_j)$ define
\[
  W[\boldsymbol\gamma]
  :=\bigcap_{r=1}^j(W-\gamma_{r,H}).
\]
Then the following assertions hold.
\begin{enumerate}[label=\textup{(\roman*)}]
\item For every $k\geq1$,
\begin{equation}\label{eq:intro-cap-characterization}
  \mu_W\in\HU_k
  \quad\Longleftrightarrow\quad
  \mu_{W[\boldsymbol\gamma]}\in\HU_1
\end{equation}
for every $1\leq j\leq k$ and every tuple of distinct lattice vectors
$\boldsymbol\gamma=(0,\gamma_2,\ldots,\gamma_j)$ for which
$W[\boldsymbol\gamma]$ has nonempty interior.  In particular, for regular
Euclidean cut-and-project processes, $\HU_k$ implies $\HU_j$ whenever
$j\leq k$.
\item If $H=\mathbb R^m$ with $m\in\{2,3\}$ and $W$ is a Euclidean ball,
then, for every $k\geq1$,
\begin{equation}\label{eq:intro-cap-low-dim}
  \mu_W\in\HU_k\quad\Longleftrightarrow\quad\mu_W\in\HU_1.
\end{equation}
\item For every $m\geq4$ there is an explicit cut-and-project scheme with
$V=H=\mathbb R^m$ and a Euclidean ball $W\subset H$ for which
\begin{equation}\label{eq:intro-cap-high-dim}
  \mu_W\in\HU_1\setminus\HU_2.
\end{equation}
\end{enumerate}
\end{maintheorem}

For a two-ball lens the axial Fourier decay is $r^{-2}$, whereas for a
ball in $\mathbb R^m$ it is $r^{-(m+1)/2}$; the former is slower precisely
when $m\geq4$.
In internal dimension one, Theorem~\ref{thm:cap-codim-one-separation} gives an
interval-window example in physical dimension three which belongs to
$\HU_1\setminus\HU_2$.

A process is \emph{$1$-stealthy} if its first Bartlett spectrum vanishes on a
neighbourhood of the origin.  By Theorem~\ref{thm:intro-geometric-spectral},
this implies $\HU_1$.  For stealthy hyperuniform point patterns in the
physics literature, see \cite{TZS15}.  Consider the following nonperiodic
Fourier quasicrystal.  Let
$\mathbb T=\mathbb R/\mathbb Z$ and let $\Sigma\subset\mathbb T^2$ be the
torus zero set of
\[
  P(z_1,z_2)=1-\frac13z_1+\frac13z_2^2-z_1z_2^2.
\]
For irrational $\alpha>0$ and Haar-uniform $\theta\in\mathbb T^2$, set
\[
  \Lambda_\theta:=\{t\in\mathbb R:\theta+t(\alpha,1)\in\Sigma\},
  \qquad
  p_\theta:=\sum_{t\in\Lambda_\theta}\delta_t,
\]
and let $\mu_\alpha$ be the law of $p_\theta$.

\begin{maintheorem}[Kurasov--Sarnak point process]
\label{thm:intro-stealthy-pair}
For every irrational $\alpha>0$, the point process $\mu_\alpha$ is
$1$-stealthy and does not belong to $\HU_2$.
\end{maintheorem}

Proposition~\ref{prop:common-gap-periodic} shows that a common spectral gap at
all pattern orders forces an ergodic uniformly discrete simple point process
of positive intensity to be periodic.  Proposition~\ref{prop:periodic-full-reciprocal-spectrum}
gives, for a periodic configuration $C$ with full translation lattice
$\Gamma_C:=\{t\in\mathbb R^d:C+t=C\}$, writing
$\mathcal C_k:=C_c(\Sh_k(\mathbb R^d);\mathbb C)$,
\[
  \bigcup_{k\geq1}\ \bigcup_{\varphi\in\mathcal C_k}
  \operatorname{supp}\sigma^{(k)}_{\varphi,\varphi}
  =\Gamma_C^*\setminus\{0\},
  \qquad
  \operatorname{st}_\infty(C)=\lambda_1(\Gamma_C^*),
\]
where $\lambda_1(\Gamma_C^*):=\min_{0\neq\xi\in\Gamma_C^*}|\xi|$ and
$\operatorname{st}_\infty$ is defined in
\eqref{eq:all-order-stealth-radius}.
Corollaries~\ref{cor:periodic-all-order-hermite} and~\ref{cor:common-gap-sphere-packing} identify the unit-intensity optimization
with the Hermite problem, equivalently lattice sphere packing under reciprocal
duality.  Proposition~\ref{prop:flc-first-order-periodic} gives periodicity in
dimension one under finite local complexity and $1$-stealthiness, whereas
Proposition~\ref{prop:all-order-diffuse-stealthy} gives a nonperiodic diffuse
random measure that is $k$-stealthy for every finite $k$ with no common positive
gap.  The remaining point-process problem with order-dependent gaps is stated
in Remark~\ref{rem:all-order-open-question}.

\subsection{Related work and questions}
\label{subsec:intro-discussion}

Higher-moment expansions of the ordinary number statistic keep the statistic
fixed and vary the moment order \cite{QT97,TKK21}.  Classical polyspectra are
Fourier transforms of higher cumulants and involve several frequency variables
\cite{Bri94}.  The $k$-th Bartlett spectrum considered here is instead a
second-order covariance spectrum in the single translation-frequency variable
of the derived measures $\mathcal P_{k,\varphi}$.  Higher-order truncated
correlations enter variance and cumulant expansions for ordinary linear
statistics \cite{KY24}, with correlation or cumulant order distinct from
pattern size.

Gowers uniformity norms and their ergodic Host--Kra analogues average products
over affine cubes \cite{Gow01,HK05}; $\HU_k$ concerns second-order large-scale
fluctuations of additive counts of arbitrary local $k$-point shapes.
Multi-hyperuniformity concerns simultaneous first-order hyperuniformity of
components in a multicomponent system \cite{JLHMC14}.  Weighted hyperuniformity
attaches weights to particles or local attributes \cite{TKKCS26}; here the
weights are attached to local $k$-point configurations.  Higher-order
structural statistics can distinguish hyperuniform systems with matching pair
statistics \cite{WT23}, but do not measure the large-scale variance of local
pattern measures.

Symmetric maps applied to all distinct $k$-tuples of Poisson or binomial
processes induce point processes on the target space \cite{ST12,DST16}.
Taking barycentre and shape as the target variables gives the pairwise-distinct
part of the pattern process used here; the results of \cite{ST12,DST16} concern
Poisson approximation and scaling limits.  Variance asymptotics and limit
theory for nonlinear local scores and geometric U-statistics, including
associated weighted point measures, are developed in \cite{BYY19}.  Invariant
transports of stationary random measures, their asymptotic variance,
hyperuniformity and Bartlett spectra are studied in \cite{KLLY25}.  For
cut-and-project sets, patch frequencies and discrepancy are studied in
\cite{HJKW19}.

Frommer and Hanke study volume-normalized variances of pair statistics
over distinct points and prove positive volume-order pair variance for projection
determinantal processes under their assumptions \cite{FH26}.  Their class of admissible functions and truncation convention differ from ours; see
Proposition~\ref{prop:FH-comparison}.  
The Kurasov--Sarnak point process comes from the nonperiodic crystalline
measures of \cite{KS20}; see also \cite{AKKV25} for diffraction of Fourier
quasicrystals.  The return-time description and gap distribution for
one-dimensional Fourier quasicrystals were developed in \cite{AV24}.

Open questions include the following.  For cut-and-project processes, when do
Fourier decay of finite intersections of window translates and the arithmetic
distribution of dual lattice points force $\HU_1\Rightarrow\HU_k$ for every
finite $k$?  Do the realization results of \cite{Bjo26a,Bjo26b} extend so that
every essentially free ergodic probability-preserving $V$-action admits a
generating Delone cross-section whose return-time point process lies in
$\bigcap_{k\geq1}\HU_k$?  In dimension one, must an ergodic uniformly discrete
point process of positive intensity which is $k$-stealthy for every $k$ be
periodic?  Section~\ref{sec:all-order-stealthiness} proves the last assertion
for a gap uniform in $k$, and already under $1$-stealthiness when the
configurations have finite local complexity; it fails for diffuse random
measures.

\section{Geometric and spectral hyperuniformity}
\label{sec:general-theory}

Let $(V,\langle\cdot,\cdot\rangle)$ be a finite-dimensional real inner-product
space of dimension $d$.  We write $V^*$ for its real dual, endowed with the
dual inner product induced by the inner product on $V$.  More generally, if $E$ is a
locally convex space, $E^*$ denotes its continuous dual when $E$ is real and
its continuous anti-dual when $E$ is complex.

For a locally compact space $Y$, $C_c(Y)$ denotes the real-valued continuous
compactly supported functions on $Y$, and $C_c(Y;\mathbb C)$ their
complex-valued counterparts.  Both carry the inductive-limit topology of
the Banach spaces $C_K(Y)$, $K\subset Y$ compact, equipped with the supremum
norm.  We regard a locally finite positive measure on $V$ as a positive
continuous linear functional on $C_c(V)$ and write
$\cM(V):=C_c(V)^*_+$.  Thus, if $p\in\cM(V)$ and
$f\in C_c(V)$, its action on $f$ is denoted by $p(f)$.  By the Riesz--Markov theorem, $\cM(V)$ is canonically identified with the
space of locally finite positive Radon measures on $V$.
We equip $\cM(V)$ with the vague topology, that is, the weakest topology for
which $p\mapsto p(f)$ is continuous for every $f\in C_c(V)$.

For $a\in V$ and $f\in C_c(V)$, set $\tau_a f(v):=f(v+a)$.  The translation
action of $V$ on $\cM(V)$ is given by $(a.p)(f):=p(\tau_a f)$.  Throughout
the paper, $\mu$ denotes a $V$-invariant Borel probability measure on
$\cM(V)$.  When $\mu$ is supported on functionals of the form
$p(f)=\sum_{v\in\Lambda}f(v)$, where $\Lambda\subset V$ is locally finite and
simple, we call $\mu$ a \emph{point process}.  Invariance, ergodicity and
mixing always refer to the $V$-action.

For $r\geq1$, we say that $\mu$ is \emph{locally $L^r$-integrable} if
$p\mapsto p(f)$ belongs to $L^r(\mu)$ for every $f\in C_c(V)$.

We denote by $\Vol_m$ the Lebesgue measure on any $m$-dimensional Euclidean
space determined by its Euclidean structure and normalized so that a unit
cube spanned by an orthonormal basis has volume one.  The notation $\Vol_d$ is used for this normalized Lebesgue measure on both
$V$ and $V^*$.

We use
\begin{equation}\label{eq:covariance-convention}
  \Cov_\mu(F,G)
  :=\int_{\cM(V)}(F-\mu(F))\overline{(G-\mu(G))}\,d\mu,
  \qquad \mu(F):=\int_{\cM(V)}F\,d\mu,
\end{equation}
so covariance is linear in the first variable and conjugate-linear in the
second.  The Fourier transform is normalized by
\begin{equation}\label{eq:fourier-convention}
  \widehat f(\xi)
  :=\int_V f(v)e^{-2\pi i\xi(v)}\,d\Vol_d(v),
  \qquad f\in L^1(V),\quad \xi\in V^*.
\end{equation}
Whenever complex-valued functions are used, $S$ and $T_k$ are extended
complex linearly in their function arguments.

For the rest of this section, local $L^{2k}$-integrability is understood
whenever $\HU_k$ or the $k$-th Bartlett spectrum is defined or used.

\subsection{Pattern statistics and shape spaces}
\label{subsec:pattern-statistics}

For $f\in C_c(V)$, define $Sf:\cM(V)\to\mathbb R$ by $Sf(p):=p(f)$.
We call $Sf$ the \emph{first statistic} associated with $f$.

For $k\geq 1$, let
\begin{equation}\label{eq:centered-tuples}
  V_0^k:=\Big\{(v_1,\ldots,v_k)\in V^k:
  \sum_{j=1}^k v_j=0\Big\},
\end{equation}
and let $\mathfrak S_k$ act on $V_0^k$ by permuting coordinates.  The
quotient $\Sh_k(V):=V_0^k/\mathfrak S_k$ is the \emph{$k$-point shape
space}.  If
$v=(v_1,\ldots,v_k)\in V^k$, its barycentre and centred part are
\begin{equation}\label{eq:barycenter-centered-part}
  \bar v:=\frac1k\sum_{j=1}^k v_j,
  \qquad
  v^\circ:=(v_1-\bar v,\ldots,v_k-\bar v)\in V_0^k,
\end{equation}
and $[v^\circ]\in\Sh_k(V)$ denotes the corresponding unordered centred
shape.

Given $f\in C_c(V)$ and $\varphi\in C_c(\Sh_k(V))$, the function
\begin{equation}\label{eq:pattern-kernel}
  F_{f,\varphi}(v_1,\ldots,v_k)
  :=f(\bar v)\,\varphi([v^\circ])
\end{equation}
weights the barycentre by $f$ and the centred shape by $\varphi$.
The function $F_{f,\varphi}$ has compact support in $V^k$.  The associated
\emph{$k$-point pattern statistic} is
\begin{equation}\label{eq:pattern-statistic}
  T_k[f,\varphi](p):=p^{\otimes k}(F_{f,\varphi}),
  \qquad p\in\cM(V),
\end{equation}
where $p^{\otimes k}$ is the $k$-fold tensor product of the positive
functional $p$.  Under the Riesz--Markov identification,
\eqref{eq:pattern-statistic} also makes sense for bounded Borel functions
$f$ with compact support, in particular for indicators of bounded Borel
sets.

\begin{remark}[The case $k=2$]\label{rem:k2-expanded}
For $x,y\in V$, put
\[
  s(x,y):=\left[\left(\frac{x-y}{2},\frac{y-x}{2}\right)\right]
  \in\Sh_2(V).
\]
Then
\begin{equation}\label{eq:k2-expanded-statistic}
  T_2[f,\varphi](p)
  =\sum_{x,y\in\Lambda}
     f\!\left(\frac{x+y}{2}\right)\varphi(s(x,y))
\end{equation}
for $p=\sum_{x\in\Lambda}\delta_x$, with $x=y$ allowed.  Hence
\[
\begin{aligned}
&T_2[f,\varphi](p)\,\overline{T_2[g,\psi](p)} \\
&\quad=\sum_{x_1,x_2,x_3,x_4\in\Lambda}
  f\!\left(\frac{x_1+x_2}{2}\right)
  \overline{g\!\left(\frac{x_3+x_4}{2}\right)}
  \varphi(s(x_1,x_2))\,
  \overline{\psi(s(x_3,x_4))}.
\end{aligned}
\]
Thus the covariance may involve correlation data of the underlying point
process up to order four.  If $\mu$ is locally $L^4$-integrable, these
statistics belong to $L^2(\mu)$ and
\[
  \Cov_\mu\bigl(T_2[f,\varphi],T_2[g,\psi]\bigr)
\]
is sesquilinear in $(f,g)$ and separately in $(\varphi,\psi)$.  In
particular, $\Var_\mu(T_2[f,\varphi])$ is quadratic in $f$ for fixed
$\varphi$ and quadratic in $\varphi$ for fixed $f$.  Thus $k=2$ refers to
the number of points in the local pattern, even though its variance can
depend on fourth-order correlations of $p$.
\end{remark}

We reserve Latin letters $f,g,\ldots$ for functions of the barycentre variable in $V$ and Greek letters $\varphi,\psi,\ldots$ for functions on $\Sh_k(V)$.

\begin{lemma}\label{lem:pattern-statistics-L2}
Assume that $\mu$ is locally $L^{2k}$-integrable.  Then
$T_k[f,\varphi]\in L^2(\mu)$ for every $f\in C_c(V)$ and
$\varphi\in C_c(\Sh_k(V))$.
\end{lemma}

\begin{proof}
Let $C\subset V_0^k$ be the inverse image of $\operatorname{supp}\varphi$
under the quotient map $V_0^k\to\Sh_k(V)$.  Since $\mathfrak S_k$ is finite,
$C$ is compact.  Choose $r>0$ such that $\|v_j\|\leq r$ for every
$(v_1,\ldots,v_k)\in C$ and every $j$.  If
$F_{f,\varphi}(v_1,\ldots,v_k)\neq0$, then
$\bar v\in\operatorname{supp}f$ and $v_j-\bar v\in B_r$ for every $j$.
Hence each $v_j$ lies in
$K:=\operatorname{supp}f+B_r$.  Choose $h\in C_c(V)$ with $h\geq1$ on $K$
and $h\geq0$.  Then
\begin{equation}\label{eq:pattern-statistic-L2-bound}
  |T_k[f,\varphi](p)|
  \leq \|f\|_\infty\,\|\varphi\|_\infty\,p(h)^k,
\end{equation}
and the right-hand side belongs to $L^2(\mu)$ by local
$L^{2k}$-integrability.
\end{proof}

For $k=1$, the space $\Sh_1(V)$ consists of one point.  Identifying its
constant function $1$ with the scalar $1$, the pattern statistic satisfies
\begin{equation}\label{eq:first-order-pattern-statistic}
  T_1[f,1]=Sf,
  \qquad f\in C_c(V).
\end{equation}

\subsection{Derived barycentre measures}
\label{subsec:derived-barycentre}

For a fixed $\varphi\in C_c(\Sh_k(V);\mathbb C)$ and
$p\in\cM(V)$, define the complex Radon measure
$\mathcal P_{k,\varphi}(p)$ on $V$ by
\begin{equation}\label{eq:derived-barycentre-measure}
  \mathcal P_{k,\varphi}(p)(f):=T_k[f,\varphi](p),
  \qquad f\in C_c(V;\mathbb C).
\end{equation}
Local finiteness follows from the same compact-support estimate as in
Lemma~\ref{lem:pattern-statistics-L2}: if the barycentre is restricted to a
compact set and the centred shape lies in $\operatorname{supp}\varphi$, then
all $k$ coordinates lie in one fixed compact subset of $V$.  If
$\varphi\geq0$, the measure is positive; general complex-valued $\varphi$ give signed or complex measures.  Translation of the underlying configuration translates
the derived measure:
\begin{equation}\label{eq:derived-barycentre-equivariance}
  \mathcal P_{k,\varphi}(a.p)=a.\mathcal P_{k,\varphi}(p),
  \qquad a\in V.
\end{equation}
Thus, under a $V$-invariant law $\mu$, each
$\mathcal P_{k,\varphi}$ is a stationary random measure in the barycentre
variable, in the signed/complex sense.  Equivalently,
\[
  T_k[\tau_a f,\varphi](p)=T_k[f,\varphi](a.p),
  \qquad a\in V,
\]
so for fixed $\varphi$ the covariance form
\[
  (f,g)\longmapsto
  \Cov_\mu\bigl(T_k[f,\varphi],T_k[g,\varphi]\bigr)
\]
is positive and invariant under simultaneous translation of $f$ and $g$.
This is the translation-invariance used below to obtain a scalar spectral
measure for each $\varphi$; the $k$-th Bartlett spectrum assembles these
measures and their cross-covariances as $\varphi$ varies.

\subsection{Geometric hyperuniformity}
\label{subsec:geometric-hu}

Let $B_R\subset V$ denote the closed Euclidean ball of radius $R$ centred at the
origin.

\begin{definition}\label{def:geometric-k-hu}
We say that $\mu$ is \emph{geometrically $k$-hyperuniform}, and write
$\mu\in\HU_k$, if, for every $\varphi\in C_c(\Sh_k(V))$,
\begin{equation}\label{eq:geometric-k-hu}
  \lim_{R\to\infty}
  \frac{\Var_\mu\big(T_k[\chi_{B_R},\varphi]\big)}{\Vol_d(B_R)}=0.
\end{equation}
\end{definition}

We also write $\HU_{\leq k}:=\bigcap_{j=1}^k\HU_j$.  No inclusion between
$\HU_j$ and $\HU_{j+1}$ is assumed.

We use the notation $\HU_2$ rather than ``pair hyperuniformity'', since
related but inequivalent pair statistics and truncation conventions occur in
the literature; see Proposition~\ref{prop:FH-comparison}.

At first order,
\begin{equation}\label{eq:first-order-ball-statistic}
  T_1[\chi_{B_R},1]=S\chi_{B_R},
\end{equation}
so geometric $1$-hyperuniformity is ordinary geometric hyperuniformity.

\subsection{Operator-valued Bartlett spectra}
\label{subsec:bartlett}

By Lemma~\ref{lem:pattern-statistics-L2} and
\eqref{eq:derived-barycentre-equivariance}, for fixed $\varphi$ the covariance
form
\[
  (f,g)\longmapsto
  \Cov_\mu\bigl(T_k[f,\varphi],T_k[g,\varphi]\bigr)
\]
is continuous, positive and invariant under simultaneous translations of
$f$ and $g$.  Proposition~\ref{prop:operator-bartlett} applies the Schwartz
kernel theorem and the Bochner--Schwartz theorem to obtain the corresponding
positive scalar measure on $V^*$.  Polarization in the shape variable gives
cross-covariance measures for $\varphi$ and $\psi$, which are assembled
into the $k$-th Bartlett spectrum.  The space $C_c(\Sh_k(V))$ is equipped
with the inductive-limit topology determined by compact supports, as described
below; Radon regularity and countable additivity are imposed coefficientwise.

\subsubsection{Operator-valued Radon measures}

For compact $K\subset\Sh_k(V)$, let $C_K(\Sh_k(V))$ be the Banach
space of complex-valued continuous functions supported in $K$, with the
supremum norm.  The complex vector space
$\mathcal C_k:=C_c(\Sh_k(V);\mathbb C)$ is equipped with the inductive-limit
topology of the spaces $C_K(\Sh_k(V))$: a linear functional is continuous
exactly when its restriction to every $C_K(\Sh_k(V))$ is continuous.  Its
\emph{continuous anti-dual}
$\mathcal C_k^*$ consists of the continuous conjugate-linear functionals on
$\mathcal C_k$ and is equipped with the weak-$*$ topology.  Thus a linear map
$A:\mathcal C_k\to\mathcal C_k^*$ is continuous precisely when, for every
$\psi\in\mathcal C_k$, the scalar map $\varphi\mapsto(A\varphi)(\psi)$ is
continuous.  We write $\mathcal L(\mathcal C_k,\mathcal C_k^*)$ for these
continuous linear maps.

An operator $A\in\mathcal L(\mathcal C_k,\mathcal C_k^*)$ is called
\emph{positive} if $(A\varphi)(\varphi)\geq0$ for every
$\varphi\in\mathcal C_k$.  We denote by
\begin{equation}\label{eq:borel-delta-ring}
  \mathscr B_c(V^*):=\{E\subset V^*: E\text{ is Borel and relatively compact}\}
\end{equation}
the \emph{Borel $\delta$-ring} of $V^*$.  It is closed under finite unions,
relative complements and countable intersections, but not in general under
countable unions.

\begin{definition}\label{def:operator-valued-radon-measure}
A \emph{locally finite positive weak operator-valued Radon measure} on $V^*$
is a map
\begin{equation}\label{eq:operator-valued-measure-map}
  \boldsymbol M:\mathscr B_c(V^*)\longrightarrow
  \mathcal L(\mathcal C_k,\mathcal C_k^*)
\end{equation}
such that $\boldsymbol M(E)$ is positive for every $E\in\mathscr B_c(V^*)$
and, for every $\varphi,\psi\in\mathcal C_k$, the \emph{coefficient measure}
\begin{equation}\label{eq:operator-coefficient-measure}
  M_{\varphi,\psi}(E):=(\boldsymbol M(E)\varphi)(\psi),
  \qquad E\in\mathscr B_c(V^*),
\end{equation}
extends to a complex Radon measure on the Borel $\sigma$-algebra of $V^*$.
In particular, $\boldsymbol M$ is \emph{weakly countably additive}: if
$E=\bigsqcup_{n\geq1}E_n$ with $E,E_n\in\mathscr B_c(V^*)$, then
\begin{equation}\label{eq:weak-countable-additivity}
  (\boldsymbol M(E)\varphi)(\psi)
  =\sum_{n\geq1}(\boldsymbol M(E_n)\varphi)(\psi),
  \qquad \varphi,\psi\in\mathcal C_k.
\end{equation}
\end{definition}

\begin{remark}
Here ``operator-valued'' refers to the weak notion in
Definition~\ref{def:operator-valued-radon-measure}: countable additivity is
coefficientwise, not in an operator norm or strong operator topology.  No joint continuity on
$\mathcal C_k\times\mathcal C_k$ is assumed.
\end{remark}

The restriction to $\mathscr B_c(V^*)$ concerns only the operator values;
the coefficient measures are defined on all Borel sets.  Compactly supported
functions can be integrated coefficientwise: for $h\in C_c(V^*)$ there is a
unique operator
$\boldsymbol M(h)\in\mathcal L(\mathcal C_k,\mathcal C_k^*)$ satisfying
\begin{equation}\label{eq:operator-valued-integration}
  (\boldsymbol M(h)\varphi)(\psi)
  =\int_{V^*}h(\xi)\,dM_{\varphi,\psi}(\xi),
  \qquad \varphi,\psi\in\mathcal C_k.
\end{equation}
For $h\geq0$ the operator $\boldsymbol M(h)$ is positive, and these
integrals determine $\boldsymbol M$ uniquely.

\subsubsection{The Bartlett spectrum}

\begin{definition}[Bartlett spectrum]\label{def:bartlett-spectrum}
A locally finite positive weak operator-valued Radon measure $\bsigmak$ on
$V^*$ is called the \emph{$k$-th Bartlett spectrum of $\mu$} if its
coefficient measures
\begin{equation}\label{eq:bartlett-coefficient-definition}
  \sigmak_{\varphi,\psi}(E)
  :=(\bsigmak(E)\varphi)(\psi),
  \qquad E\in\mathscr B_c(V^*),
\end{equation}
satisfy, for all $\varphi,\psi\in\mathcal C_k$ and
$f,g\in C_c^\infty(V;\mathbb C)$,
\begin{equation}\label{eq:bartlett-defining-identity}
  \Cov_\mu\big(T_k[f,\varphi],T_k[g,\psi]\big)
  =\int_{V^*}\widehat f(\xi)\overline{\widehat g(\xi)}\,
    d\sigmak_{\varphi,\psi}(\xi).
\end{equation}
\end{definition}

\begin{proposition}[Existence and uniqueness of the Bartlett spectrum]
\label{prop:operator-bartlett}
If $\mu$ is locally $L^{2k}$-integrable, then $\mu$ admits a unique $k$-th
Bartlett spectrum.
\end{proposition}

\begin{proof}
\smallskip\noindent\emph{Scalar spectrum for fixed $\varphi$.}
Fix $\varphi\in\mathcal C_k$ and set
$Q_\varphi(f,g):=\Cov_\mu(T_k[f,\varphi],T_k[g,\varphi])$.  By
\Cref{lem:pattern-statistics-L2}, this is a continuous positive
sesquilinear form on $C_c(V;\mathbb C)$.  The $V$-invariance of $\mu$
implies
\begin{equation}\label{eq:bartlett-translation-invariance}
  Q_\varphi(\tau_a f,\tau_a g)=Q_\varphi(f,g),
  \qquad a\in V.
\end{equation}

Restrict first to $C_c^\infty(V)$.  By the Schwartz kernel theorem
\cite{Sch66}, $Q_\varphi$ is represented by a distribution on
$V\times V$.  In the
coordinates $(u,w)=(v-v',v')$, simultaneous translation of $v$ and $v'$ is
translation in the second coordinate.  To make the resulting factorization
explicit, choose $\rho\in C_c^\infty(V)$ with $\int_V\rho=1$.  If $K_\varphi$
denotes the kernel in these coordinates, set
$\gamma_\varphi(h):=K_\varphi(h\otimes\rho)$.  For
$\Phi\in C_c^\infty(V\times V)$, write
$A_\Phi(u):=\int_V\Phi(u,w)\,dw$.  Then
$\Phi-A_\Phi\otimes\rho$ has zero integral in the $w$-variable and hence is
a finite sum of $w$-derivatives of compactly supported smooth functions.
Translation invariance gives
$K_\varphi(\partial_{w_j}\Psi)=0$, so
$K_\varphi(\Phi)=\gamma_\varphi(A_\Phi)$.  Thus
$K_\varphi=\gamma_\varphi\otimes\Vol_d$, and $\gamma_\varphi$ is unique.
Equivalently, with
$\widetilde g(v):=\overline{g(-v)}$,
\begin{equation}\label{eq:covariance-distribution}
  Q_\varphi(f,g)=\gamma_\varphi(f*\widetilde g),
  \qquad f,g\in C_c^\infty(V).
\end{equation}
Since $Q_\varphi(f,f)\geq0$, the distribution $\gamma_\varphi$ is positive
definite.  The Bochner--Schwartz theorem
\cite[Chapter~VII, \S9, Theorem~XVIII]{Sch66} yields a positive tempered
Radon measure $\nu_{\varphi,\varphi}$ on $V^*$ such that
\begin{equation}\label{eq:diagonal-bochner-schwartz}
  \gamma_\varphi(h)
  =\int_{V^*}\widehat h(\xi)\,d\nu_{\varphi,\varphi}(\xi),
  \qquad h\in C_c^\infty(V).
\end{equation}
This measure is unique because the Fourier transform is injective on tempered
distributions.  Since $\widehat{f*\widetilde g}=\widehat f\,\overline{\widehat g}$,
polarization of the diagonal measures in the shape variable gives tempered
complex Radon measures $\nu_{\varphi,\psi}$ satisfying
\eqref{eq:bartlett-defining-identity}.  For every relatively compact Borel
set $E\subset V^*$ and every finite family
$\varphi_1,\ldots,\varphi_N\in\mathcal C_k$,
\begin{equation}\label{eq:bartlett-coefficient-positivity}
  \sum_{i,j=1}^N c_i\overline{c_j}\,
    \nu_{\varphi_i,\varphi_j}(E)
  =\nu_{\Phi,\Phi}(E)\geq0,
  \qquad
  \Phi:=\sum_{i=1}^N c_i\varphi_i.
\end{equation}
Thus $(\varphi,\psi)\mapsto\nu_{\varphi,\psi}(E)$ is a positive
sesquilinear form, and in particular
\begin{equation}\label{eq:bartlett-measure-cauchy-schwarz}
  |\nu_{\varphi,\psi}(E)|^2
  \leq\nu_{\varphi,\varphi}(E)\nu_{\psi,\psi}(E).
\end{equation}

\smallskip\noindent\emph{Assembly in the shape variable.}
To assemble the coefficient measures, fix a compact set
$K\subset\Sh_k(V)$ and $E\in\mathscr B_c(V^*)$.  Such an $E$ admits
$h\in C_c^\infty(V)$ with $|\widehat h|\geq1$ on $E$: take
$h_0\in C_c^\infty(V)$ with $\widehat h_0(0)=1$, choose $r>0$ such that
$|\widehat h_0|\geq1/2$ on $B_r^*$, and set
$h_\lambda(v):=2\lambda^{-d}h_0(v/\lambda)$ with $\lambda>0$ small enough
that $\lambda E\subset B_r^*$.  The bound
\eqref{eq:pattern-statistic-L2-bound} then gives a constant $C=C(E,K)$ such that
\begin{equation}\label{eq:bartlett-local-bound}
  \nu_{\varphi,\varphi}(E)
  \leq \Var_\mu\big(T_k[h,\varphi]\big)
  \leq C\|\varphi\|_\infty^2,
  \qquad \varphi\in C_K(\Sh_k(V)).
\end{equation}
Equation~\eqref{eq:bartlett-measure-cauchy-schwarz} gives
\begin{equation}\label{eq:bartlett-coefficient-bound}
  |\nu_{\varphi,\psi}(E)|
  \leq \nu_{\varphi,\varphi}(E)^{1/2}
         \nu_{\psi,\psi}(E)^{1/2}
  \leq C\|\varphi\|_\infty\|\psi\|_\infty.
\end{equation}
For fixed $\psi\in\mathcal C_k$, choose a compact set containing
$\operatorname{supp}\psi$.  The same estimate, with a common compact set for
$\varphi$ and $\psi$, shows that
$\varphi\mapsto\nu_{\varphi,\psi}(E)$ is continuous on every
$C_K(\Sh_k(V))$; hence it is continuous on $\mathcal C_k$ by the universal
property of this inductive-limit topology.  The same holds in the second variable.
Thus the coefficient form is separately continuous, which is exactly what is
needed for a linear map into the weak-$*$ anti-dual.  It determines a unique
operator $\bsigmak(E)\in\mathcal L(\mathcal C_k,\mathcal C_k^*)$ by
\begin{equation}\label{eq:operator-bartlett}
  (\bsigmak(E)\varphi)(\psi)=\nu_{\varphi,\psi}(E),
  \qquad \varphi,\psi\in\mathcal C_k.
\end{equation}
Positivity and weak countable additivity hold coefficientwise, so
$\bsigmak$ is a locally finite positive weak operator-valued Radon measure and
satisfies \eqref{eq:bartlett-defining-identity}.  Uniqueness follows from the uniqueness of the scalar measures just noted and polarization.
\end{proof}

By construction,
\begin{equation}\label{eq:bartlett-coefficients}
  (\bsigmak(E)\varphi)(\psi)=\sigmak_{\varphi,\psi}(E),
  \qquad E\in\mathscr B_c(V^*),\quad
  \varphi,\psi\in\mathcal C_k.
\end{equation}

At the first level the shape space consists of one point, and
\Cref{prop:operator-bartlett} reduces to the ordinary scalar Bartlett
spectrum, which we denote by $\sigma$.  For higher orders the
operator-valued measure $\bsigmak$ is the primary spectral object; the
measures $\sigmak_{\varphi,\psi}$ are its coefficients.

\begin{lemma}[Translation boundedness of the coefficient spectra]
\label{lem:bartlett-translation-bounded}
For every $\varphi\in\mathcal C_k$, the positive measure
$\sigmak_{\varphi,\varphi}$ is \emph{translation bounded}: for every compact
$E\subset V^*$ there is a constant $C_{E,\varphi}$ such that
\begin{equation}\label{eq:bartlett-translation-bounded}
  \sigmak_{\varphi,\varphi}(E+\xi_0)\leq C_{E,\varphi},
  \qquad \xi_0\in V^*.
\end{equation}
\end{lemma}

\begin{proof}
Choose $h\in C_c^\infty(V)$ with $|\widehat h|\geq1$ on $E$.  For
$\xi_0\in V^*$ set
\begin{equation}\label{eq:bartlett-modulated-test}
  h_{\xi_0}(v):=e^{2\pi i\xi_0(v)}h(v),
  \qquad v\in V.
\end{equation}
Then $\widehat h_{\xi_0}(\xi)=\widehat h(\xi-\xi_0)$, and hence
\begin{equation}\label{eq:bartlett-translation-bound-proof}
  \sigmak_{\varphi,\varphi}(E+\xi_0)
  \leq \int_{V^*}|\widehat h_{\xi_0}(\xi)|^2\,
       d\sigmak_{\varphi,\varphi}(\xi)
  =\Var_\mu\big(T_k[h_{\xi_0},\varphi]\big).
\end{equation}
The estimate in \eqref{eq:pattern-statistic-L2-bound} is uniform in
$\xi_0$, because $|h_{\xi_0}|=|h|$ and
$\operatorname{supp}h_{\xi_0}=\operatorname{supp}h$.  This proves
\eqref{eq:bartlett-translation-bounded}.
\end{proof}

\subsection{Spectral hyperuniformity}
\label{subsec:spectral-hu}

Let $B_\varepsilon^*\subset V^*$ denote the closed ball of radius $\varepsilon$
centred at the origin for the dual Euclidean structure.

\begin{definition}\label{def:spectral-k-hu}
We say that $\mu$ is \emph{spectrally $k$-hyperuniform} if, for every
$\varphi,\psi\in\mathcal C_k$,
\begin{equation}\label{eq:spectral-k-hu}
  \lim_{\varepsilon\to0}
  \frac{(\bsigmak(B_\varepsilon^*)\varphi)(\psi)}
       {\Vol_d(B_\varepsilon^*)}=0.
\end{equation}
Equivalently, by positivity and polarization, it is enough to require
\begin{equation}\label{eq:spectral-k-hu-diagonal}
  \lim_{\varepsilon\to0}
  \frac{\sigmak_{\varphi,\varphi}(B_\varepsilon^*)}
       {\Vol_d(B_\varepsilon^*)}=0,
  \qquad \varphi\in\mathcal C_k.
\end{equation}
\end{definition}

\begin{lemma}[Ball variance formula]\label{lem:ball-variance-bartlett}
For every $R>0$ and $\varphi\in\mathcal C_k$,
\begin{equation}\label{eq:ball-variance-bartlett}
  \Var_\mu\big(T_k[\chi_{B_R},\varphi]\big)
  =\int_{V^*}|\widehat{\chi}_{B_R}(\xi)|^2\,
    d\sigmak_{\varphi,\varphi}(\xi).
\end{equation}
\end{lemma}

\begin{proof}
Choose $f_n\in C_c^\infty(V)$ with $0\leq f_n\leq1$, with $f_n=1$ on
$B_R$, and with $\operatorname{supp}f_n\subset B_{R+1/n}$.  Then
$f_n\to\chi_{B_R}$ pointwise.  A fixed function $h\in C_c(V)$ with
$h\geq1$ on $B_{R+1}$ gives, by \eqref{eq:pattern-statistic-L2-bound}, an
$L^2(\mu)$-majorant for $T_k[f_n,\varphi]$.  Hence
\begin{equation}\label{eq:ball-statistic-L2-convergence}
  T_k[f_n,\varphi]\longrightarrow
  T_k[\chi_{B_R},\varphi]
  \quad\text{in }L^2(\mu).
\end{equation}
Applying \eqref{eq:bartlett-defining-identity} to $f_n-f_m$ shows that
$(\widehat f_n)$ is Cauchy in
$L^2(\sigmak_{\varphi,\varphi})$.  On the other hand,
$\widehat f_n(\xi)\to\widehat\chi_{B_R}(\xi)$ for every $\xi\in V^*$ by
dominated convergence on $V$.  An $L^2(\sigmak_{\varphi,\varphi})$-convergent
subsequence converges $\sigmak_{\varphi,\varphi}$-almost everywhere to the
$L^2$-limit; the pointwise convergence therefore identifies that limit with
$\widehat\chi_{B_R}$.  Taking limits in the Bartlett identity gives
\eqref{eq:ball-variance-bartlett}.
\end{proof}

At first order, the geometric--spectral equivalence is a folklore
principle.  For the Euclidean-ball argument at first order, we refer to
\cite[Theorem~3.6(i),(iii)]{BH24}.

\begin{theorem}[Geometric--spectral equivalence]\label{thm:geometric-spectral}
Let $\mu$ be a $V$-invariant probability measure on $\cM(V)$ which is
locally $L^{2k}$-integrable.  Then $\mu$ is geometrically $k$-hyperuniform
if and only if it is spectrally $k$-hyperuniform.
\end{theorem}

\begin{proof}
Fix $\varphi\in\mathcal C_k$ and write
$\nu:=\sigmak_{\varphi,\varphi}$.  By
\Cref{lem:ball-variance-bartlett}, geometric $k$-hyperuniformity for this $\varphi$ is equivalent to
\begin{equation}\label{eq:geometric-spectral-variance-form}
  \lim_{R\to\infty}
  \frac{1}{\Vol_d(B_R)}
  \int_{V^*}|\widehat\chi_{B_R}(\xi)|^2\,d\nu(\xi)=0.
\end{equation}

Assume first that \eqref{eq:geometric-spectral-variance-form} holds.  Since
$\widehat\chi_{B_1}$ is continuous and
$\widehat\chi_{B_1}(0)=\Vol_d(B_1)$, there are $c,\delta>0$ such that
$|\widehat\chi_{B_1}(\xi)|\geq c$ for $\|\xi\|\leq\delta$.  Using
$\widehat\chi_{B_R}(\xi)=R^d\widehat\chi_{B_1}(R\xi)$ gives
\begin{equation}\label{eq:geometric-implies-spectral-estimate}
  \frac{\nu(B_{\delta/R}^*)}{\Vol_d(B_{\delta/R}^*)}
  \leq C
  \frac{\Var_\mu(T_k[\chi_{B_R},\varphi])}{\Vol_d(B_R)}.
\end{equation}
Letting $R\to\infty$ proves spectral $k$-hyperuniformity.  This is the
lower-bound argument used at first order in
\cite[Theorem~3.6(i)]{BH24}.

Conversely, assume
\begin{equation}\label{eq:spectral-smallness-fixed-coefficient}
  \lim_{r\to0}\frac{\nu(B_r^*)}{\Vol_d(B_r^*)}=0.
\end{equation}
The Fourier transform of a Euclidean ball satisfies
\begin{equation}\label{eq:ball-fourier-decay}
  |\widehat\chi_{B_1}(\xi)|^2
  \leq C_d(1+\|\xi\|)^{-(d+1)},
  \qquad \xi\in V^*;
\end{equation}
see, for example, the estimate used in the proof of
\cite[Theorem~3.6(iii)]{BH24}.  Given $\eta>0$, choose $r_0>0$ so that
\begin{equation}\label{eq:spectral-small-ball-uniform}
  \nu(B_r^*)\leq \eta\,\Vol_d(B_r^*)
  \qquad(0<r\leq r_0).
\end{equation}
For $R$ large, choose $J\geq0$ with
$2^J/R\leq r_0<2^{J+1}/R$.  Decompose $B_{2^J/R}^*$ into $B_{1/R}^*$ and
the annuli
\begin{equation}\label{eq:spectral-dyadic-annuli}
  A_j:=B_{2^{j+1}/R}^*\setminus B_{2^j/R}^*,
  \qquad 0\leq j<J.
\end{equation}
Equations \eqref{eq:ball-fourier-decay} and
\eqref{eq:spectral-small-ball-uniform} imply
\begin{equation}\label{eq:spectral-inner-region-estimate}
  \frac{1}{\Vol_d(B_R)}
  \int_{B_{2^J/R}^*}|\widehat\chi_{B_R}(\xi)|^2\,d\nu(\xi)
  \leq C_d\eta\left(1+\sum_{j=0}^{J-1}2^{-j}\right)
  \leq C_d'\eta.
\end{equation}
By \Cref{lem:bartlett-translation-bounded}, $\nu$ is translation bounded.
Covering $B_t^*$ by $O(t^d)$ translates of a fixed unit ball gives
$\nu(B_t^*)=O(t^d)$ as $t\to\infty$; a dyadic annulus decomposition then
gives
\begin{equation}\label{eq:spectral-weighted-tail-integrability}
  \int_{\{\|\xi\|\geq r_0/2\}}
       \|\xi\|^{-(d+1)}\,d\nu(\xi)<\infty.
\end{equation}
Since $2^J/R\geq r_0/2$, every
$\xi\notin B_{2^J/R}^*$ satisfies $\|\xi\|\geq r_0/2$.  Using
$\widehat\chi_{B_R}(\xi)=R^d\widehat\chi_{B_1}(R\xi)$ and
\eqref{eq:ball-fourier-decay}, we obtain on this region
\[
  \frac{|\widehat\chi_{B_R}(\xi)|^2}{\Vol_d(B_R)}
  \leq \frac{C_d}{R}\,\|\xi\|^{-(d+1)}.
\]
The integrability in \eqref{eq:spectral-weighted-tail-integrability} therefore
gives
\begin{equation}\label{eq:spectral-outer-region-estimate}
  \frac{1}{\Vol_d(B_R)}
  \int_{V^*\setminus B_{2^J/R}^*}
       |\widehat\chi_{B_R}(\xi)|^2\,d\nu(\xi)
  \leq \frac{C_{d,r_0,\nu}}{R}\longrightarrow0.
\end{equation}
Combining \eqref{eq:spectral-inner-region-estimate} and
\eqref{eq:spectral-outer-region-estimate}, taking $R\to\infty$, and then
letting $\eta\to0$ proves \eqref{eq:geometric-spectral-variance-form}.

Thus, for every $\varphi\in\mathcal C_k$,
\begin{equation}\label{eq:geometric-spectral-coefficient-equivalence}
  \lim_{R\to\infty}
  \frac{\Var_\mu(T_k[\chi_{B_R},\varphi])}{\Vol_d(B_R)}=0
  \quad\Longleftrightarrow\quad
  \lim_{\varepsilon\to0}
  \frac{\sigmak_{\varphi,\varphi}(B_\varepsilon^*)}
       {\Vol_d(B_\varepsilon^*)}=0.
\end{equation}
Positivity and polarization then give the operator-valued formulation in
\eqref{eq:spectral-k-hu}.
\end{proof}

\subsection{Pair statistics and comparison with Frommer--Hanke}
\label{subsec:full-factorial}

Our definition of $\HU_2$ differs from the second-order statistics of
Frommer--Hanke \cite{FH26} in two respects: $T_2$ includes pairs with equal
entries, and the large ball is applied to the barycentre rather than to both
endpoints.  Under $\HU_1$ the first difference disappears; under the
correlation-density hypotheses of \cite[Proposition~4.1]{FH26}, so does the
second.

For a simple point process $p=\sum_{v\in\Lambda_p}\delta_v$, define the
pairwise-distinct contribution by
\begin{equation}\label{eq:factorial-tensor}
  p_{\neq}^{\otimes k}
  :=\sum_{(v_1,\ldots,v_k)\in(\Lambda_p^k)_{\neq}}
       \delta_{(v_1,\ldots,v_k)},
\end{equation}
where $(\Lambda_p^k)_{\neq}$ consists of ordered tuples with pairwise distinct
coordinates, and set
\begin{equation}\label{eq:factorial-pattern-statistic}
  T_k^{\neq}[f,\varphi](p)
  :=p_{\neq}^{\otimes k}(F_{f,\varphi}).
\end{equation}

Every ordered tuple determines a partition of $\{1,\ldots,k\}$ by declaring
two indices equivalent when the corresponding coordinates are equal.  Let
$\Pi_k$ be the set of these partitions.  If $\pi\in\Pi_k$ has blocks
$B_1,\ldots,B_j$, ordered by their least elements, then a tuple with equality
pattern $\pi$ is obtained from $j$ pairwise distinct points
$u_1,\ldots,u_j$ by repeating $u_r$ in the coordinates indexed by $B_r$.
Write $\Delta_\pi:V^j\to V^k$ for this repetition map.  For example, if $k=3$ and $\pi=\{\{1,2\},\{3\}\}$, then
$\Delta_\pi(u,v)=(u,u,v)$, whose barycentre is $(2u+v)/3$.  Then
\begin{equation}\label{eq:full-factorial-partition}
  p^{\otimes k}
  =\sum_{\pi\in\Pi_k}(\Delta_\pi)_*
       p_{\neq}^{\otimes |\pi|}.
\end{equation}
Correspondingly, if $\pi$ has $j$ blocks, define
\begin{equation}\label{eq:equality-partition-statistic}
  T_{k,\pi}[f,\varphi](p)
  :=p_{\neq}^{\otimes j}\!\left(
       (u_1,\ldots,u_j)\mapsto
       f\!\left(\frac1k\sum_{r=1}^j|B_r|u_r\right)
       \varphi\!\left([\Delta_\pi(u)^\circ]\right)
     \right),
\end{equation}
so that
\begin{equation}\label{eq:full-as-equality-partitions}
  T_k[f,\varphi]=\sum_{\pi\in\Pi_k}T_{k,\pi}[f,\varphi].
\end{equation}
The partition of $\{1,\ldots,k\}$ into singletons gives
$T_k^{\neq}$.  Whenever a block has more than one element, the barycentre in
\eqref{eq:equality-partition-statistic} weights the corresponding point by
the size of that block.

For $k=2$ there are only the distinct-point term and the diagonal term.
Denoting the diagonal shape in $\Sh_2(V)$ by $0$, we have
\begin{equation}\label{eq:full-factorial-pair-identity}
  T_2[f,\varphi]
  =T_2^{\neq}[f,\varphi]+\varphi(0)Sf.
\end{equation}

\begin{proposition}\label{prop:full-factorial-pair-hu}
Let $\mu$ be a simple point process which is locally $L^4$-integrable and
belongs to $\HU_1$.  Then $\mu\in\HU_2$ if and only if, for every
$\varphi\in C_c(\Sh_2(V))$,
\begin{equation}\label{eq:factorial-pair-hu}
  \lim_{R\to\infty}
  \frac{\Var_\mu\big(T_2^{\neq}[\chi_{B_R},\varphi]\big)}
       {\Vol_d(B_R)}=0.
\end{equation}
\end{proposition}

\begin{proof}
For a fixed $\varphi$, set
$D_R:=\varphi(0)S\chi_{B_R}$.  By \eqref{eq:full-factorial-pair-identity},
\begin{equation}\label{eq:pair-full-factorial-difference}
  T_2^{\neq}[\chi_{B_R},\varphi]
  =T_2[\chi_{B_R},\varphi]-D_R.
\end{equation}
Since $\mu\in\HU_1$, one has
$\Var_\mu(D_R)=o(\Vol_d(B_R))$.  The inequalities
\begin{equation}\label{eq:pair-variance-comparison}
  \Var_\mu(A-B)\leq2\Var_\mu(A)+2\Var_\mu(B),
  \qquad
  \Var_\mu(A)\leq2\Var_\mu(A-B)+2\Var_\mu(B),
\end{equation}
show that the original and distinct-point pair variances are simultaneously
$o(\Vol_d(B_R))$.
\end{proof}

\begin{proposition}[Comparison with Frommer--Hanke]\label{prop:FH-comparison}
Let $\mu$ be a stationary simple point process on $\mathbb R^d$ with
correlation densities.  Assume that $\rho^{(2)}$ and $\rho^{(3)}$ are bounded
and that, with
\[
  \chi^{(4)}(x_1,x_2,x_3,x_4)
  :=\rho^{(4)}(x_1,x_2,x_3,x_4)
    -\rho^{(2)}(x_1,x_2)\rho^{(2)}(x_3,x_4),
\]
one has
\begin{equation}\label{eq:FH-four-point-integrability}
  \sup_{x_1,x_2,x_4}
  \int_{\mathbb R^d}
  \big|\chi^{(4)}(x_1,x_2,x_3,x_3+x_4)\big|\,dx_3<\infty.
\end{equation}
For an even $u\in L^1(\mathbb R^d)\cap L^2(\mathbb R^d)$, define
\begin{align*}
  E_R(u)
  &:=\sum_{x\neq y}\chi_{B_R}(x)\chi_{B_R}(y)u(x-y),\\
  C_R(u)
  &:=\sum_{x\neq y}\chi_{B_R}\!\left(\frac{x+y}{2}\right)u(x-y).
\end{align*}
Then
\begin{equation}\label{eq:endpoint-barycentre-difference}
  \Var_\mu\big(E_R(u)-C_R(u)\big)=o(R^d).
\end{equation}
If either $\Var_\mu(E_R(u))$ or $\Var_\mu(C_R(u))$ is $O(R^d)$, then
\begin{equation}\label{eq:endpoint-barycentre-normalized-difference}
  \frac{\Var_\mu(E_R(u))-\Var_\mu(C_R(u))}{\Vol_d(B_R)}\longrightarrow0.
\end{equation}
In particular, endpoint and barycentre truncation give the same
subvolume-variance condition and the same finite volume-normalized variance
limit whenever that limit exists for either statistic.

If, in addition, $u$ is continuous and compactly supported, then
\begin{equation}\label{eq:FH-barycentre-is-T2neq}
  C_R(u)=T_2^{\neq}[\chi_{B_R},\varphi_u],
  \qquad
  \varphi_u\!\left(\left[\left(\frac z2,-\frac z2\right)\right]\right):=u(z).
\end{equation}
Consequently, if $\mu\in\HU_1$, then
\begin{equation}\label{eq:FH-HU2-equivalence}
  \mu\in\HU_2
  \quad\Longleftrightarrow\quad
  \Var_\mu(E_R(u))=o(R^d)
  \quad\text{for every even }u\in C_c(\mathbb R^d).
\end{equation}
Thus, on the common class of even compactly supported continuous pair
functions, their endpoint truncation and our barycentre truncation give the
same second-order hyperuniformity condition for $\HU_1$ point processes.
\end{proposition}

\begin{proof}
Put
\[
  q_R(b,z)
  :=\chi_{B_R}(b+z/2)\chi_{B_R}(b-z/2)-\chi_{B_R}(b),
  \qquad
  a_R(z):=\frac1{\Vol_d(B_R)}\int_{\mathbb R^d}|q_R(b,z)|\,db.
\]
For every fixed $z$, $a_R(z)\to0$ as $R\to\infty$, while
$0\leq a_R(z)\leq2$.  The difference $D_R:=E_R(u)-C_R(u)$ is the pair
statistic with the symmetric kernel
\[
  h_R(x,y):=u(x-y)q_R\!\left(\frac{x+y}{2},x-y\right).
\]
The coincidence-pattern expansion underlying
\cite[Equations~(4.3)--(4.5)]{FH26} is purely combinatorial and therefore
applies to this kernel.  It gives
\begin{align}
  \Var_\mu(D_R)
  ={}&
  \int h_R(x_1,x_2)\overline{h_R(x_3,x_4)}
  \chi^{(4)}(x_1,x_2,x_3,x_4)\,dx_1dx_2dx_3dx_4
  \notag\\
  &+4\int h_R(x_1,x_2)\overline{h_R(x_1,x_3)}
  \rho^{(3)}(x_1,x_2,x_3)\,dx_1dx_2dx_3
  \notag\\
  &+2\int |h_R(x_1,x_2)|^2
  \rho^{(2)}(x_1,x_2)\,dx_1dx_2.
  \label{eq:FH-coincidence-expansion}
\end{align}

After the changes of variables $b=(x+y)/2$, $z=x-y$ in the two-point term,
and $z=x_2-x_1$, $w=x_3-x_1$ in the three-point term, their absolute values,
divided by $\Vol_d(B_R)$, are bounded respectively by
\[
  \|\rho^{(2)}\|_\infty
  \int |u(z)|^2a_R(z)\,dz
  \quad\text{and}\quad
  \|\rho^{(3)}\|_\infty
  \iint |u(z)u(w)|a_R(z)\,dz\,dw.
\]
Both tend to zero by dominated convergence.  For the four-point term, write
$x_2=t$, $x_1=t+z$, $x_3=t+y$, and $x_4=t+y+w$.  Translation invariance and
\eqref{eq:FH-four-point-integrability} bound its normalized absolute value by
\[
  C\iint |u(z)u(w)|a_R(z)\,dz\,dw,
\]
which again tends to zero.  This proves
\eqref{eq:endpoint-barycentre-difference}.

The $L^2(\mu)$ triangle inequality gives
\[
  \left|\sqrt{\Var_\mu(E_R(u))}
       -\sqrt{\Var_\mu(C_R(u))}\right|
  \leq \sqrt{\Var_\mu(D_R)}=o(R^{d/2}).
\]
If either variance is $O(R^d)$, so is the other, and therefore
\[
  |\Var_\mu(E_R(u))-\Var_\mu(C_R(u))|
  \leq
  \sqrt{\Var_\mu(D_R)}
  \left(\sqrt{\Var_\mu(E_R(u))}+\sqrt{\Var_\mu(C_R(u))}\right)
  =o(R^d),
\]
which proves \eqref{eq:endpoint-barycentre-normalized-difference}.
Equation~\eqref{eq:FH-barycentre-is-T2neq} is immediate from the definition of
$T_2^{\neq}$.  Every $\varphi\in C_c(\Sh_2(\mathbb R^d))$ is of the form
$\varphi_u$ for a unique even $u\in C_c(\mathbb R^d)$, so
\eqref{eq:FH-HU2-equivalence} follows from
Proposition~\ref{prop:full-factorial-pair-hu}.
\end{proof}

Frommer and Hanke prove an explicit volume-normalized variance limit for
$E_R(u)$ under the assumptions of Proposition~\ref{prop:FH-comparison}; see
\cite[Proposition~4.1]{FH26}.  For Fourier projection processes,
\cite[Theorem~4.2]{FH26} is therefore a direct antecedent to
Theorem~\ref{thm:projection-dpp-hu1-not-hu2}.  Our formulation in terms of
$\varphi\in C_c(\Sh_2(V))$ also applies when the pair-difference measure is
singular, since $\varphi$ is specified pointwise rather than only as an
$L^1\cap L^2$ equivalence class.

\subsection{Stealthiness}
\label{subsec:stealthiness}

A spectral gap at the origin is stronger than the decay condition in spectral
$k$-hyperuniformity.

\begin{definition}\label{def:k-stealthy}
The $k$-th Bartlett spectrum is said to have a \emph{spectral gap at the
origin} if there exists $\varepsilon_0>0$ such that
\begin{equation}\label{eq:k-spectral-gap}
  \bsigmak(B_{\varepsilon_0}^*)=0
  \quad\text{in }\mathcal L(\mathcal C_k,\mathcal C_k^*).
\end{equation}
We say that $\mu$ is \emph{$k$-stealthy} if its $k$-th Bartlett spectrum has
a spectral gap at the origin.
\end{definition}

By positivity, \eqref{eq:k-spectral-gap} is equivalent to
$\sigmak_{\varphi,\varphi}(B_{\varepsilon_0}^*)=0$ for every
$\varphi\in\mathcal C_k$.

\begin{corollary}\label{cor:stealthy-implies-hu}
Every $k$-stealthy $V$-invariant probability measure is in $\HU_k$.
\end{corollary}

\begin{proof}
A spectral gap implies the limit in \eqref{eq:spectral-k-hu-diagonal}.
The assertion follows from Theorem~\ref{thm:geometric-spectral}.
\end{proof}

\section{Lattices and perturbations}
\label{sec:lattices}

Invariant lattice processes belong to $\HU_k$ for every $k$.  In every
dimension, sufficiently small non-degenerate iid perturbations of a lattice
belong to $\HU_1\setminus\HU_2$.

\subsection{Lattices}
\label{subsec:lattices}

Let $\Gamma\subset V$ be a full-rank lattice and let $m_\Gamma$ be normalized Haar
measure on $V/\Gamma$.  For $u\in V/\Gamma$, set
\begin{equation}\label{eq:lattice-configuration}
  p_u:=\sum_{\gamma\in\Gamma}\delta_{\gamma+u},
  \qquad
  \mu_\Gamma:=(u\mapsto p_u)_*m_\Gamma.
\end{equation}
We call $\mu_\Gamma$ the \emph{invariant lattice process} associated with $\Gamma$; it
is also commonly called a randomly translated lattice.  Its \emph{dual lattice}
is
\begin{equation}\label{eq:dual-lattice}
  \Gamma^*:=\{\xi\in V^*: \xi(\gamma)\in\mathbb Z
  \text{ for every }\gamma\in\Gamma\}.
\end{equation}

\begin{theorem}\label{thm:lattice-all-order-stealthy}
For every $k\geq1$, the invariant lattice process $\mu_\Gamma$ is $k$-stealthy.
In particular, $\mu_\Gamma\in\HU_k$ for every $k\geq1$.
\end{theorem}

\begin{proof}
The point counts of $p_u$ in a fixed compact set are bounded uniformly in
$u$, so $\mu_\Gamma$ is locally $L^r$-integrable for every $r<\infty$.

Fix $\varphi\in\mathcal C_k$.  Given
$\delta=(\delta_2,\ldots,\delta_k)\in\Gamma^{k-1}$, put
\begin{equation}\label{eq:lattice-pattern-data}
  c_\delta:=\frac1k\sum_{j=2}^k\delta_j,
  \qquad
  s_\delta:=
  [(-c_\delta,\delta_2-c_\delta,\ldots,\delta_k-c_\delta)]
  \in\Sh_k(V).
\end{equation}
Since $\varphi$ has compact support and $\Gamma$ is discrete, only finitely many
$\delta\in\Gamma^{k-1}$ satisfy $\varphi(s_\delta)\neq0$.  Writing a lattice
tuple as
$(\gamma,\gamma+\delta_2,\ldots,\gamma+\delta_k)$ gives
\begin{equation}\label{eq:lattice-pattern-periodization}
  T_k[f,\varphi](p_u)
  =\sum_{\delta\in\Gamma^{k-1}}\varphi(s_\delta)
     \sum_{\gamma\in\Gamma}f(u+\gamma+c_\delta),
\end{equation}
where the outer sum is finite.

Let $\operatorname{covol}(\Gamma)$ denote the $\Vol_d$-volume of a fundamental
domain of $\Gamma$.  The Fourier characters of $V/\Gamma$ are indexed by $\Gamma^*$.
For $\xi\in\Gamma^*$, the Fourier coefficient of the right-hand side of
\eqref{eq:lattice-pattern-periodization} is
\begin{equation}\label{eq:lattice-pattern-fourier-coefficient}
  \frac{\widehat f(\xi)}{\operatorname{covol}(\Gamma)}
  A_\varphi(\xi),
  \qquad
  A_\varphi(\xi)
  :=\sum_{\delta\in\Gamma^{k-1}}
      \varphi(s_\delta)e^{2\pi i\xi(c_\delta)}.
\end{equation}
The sum defining $A_\varphi$ is finite.  After subtracting the mean, the
zero Fourier mode disappears.  Parseval's identity on $V/\Gamma$ therefore gives
\begin{equation}\label{eq:lattice-bartlett-coefficients}
  \sigmak_{\varphi,\psi}
  =\frac{1}{\operatorname{covol}(\Gamma)^2}
    \sum_{\xi\in\Gamma^*\setminus\{0\}}
      A_\varphi(\xi)\overline{A_\psi(\xi)}\,\delta_\xi.
\end{equation}
Thus every coefficient of $\bsigmak$ is supported on
$\Gamma^*\setminus\{0\}$.  Since this set is discrete and separated from the
origin, there exists $\varepsilon_0>0$ such that
$\bsigmak(B_{\varepsilon_0}^*)=0$.  Hence $\mu_\Gamma$ is $k$-stealthy.  The
last assertion follows from Corollary~\ref{cor:stealthy-implies-hu}.
\end{proof}

\subsection{Perturbed lattices}
\label{subsec:perturbed-lattices}

Let $\Gamma\subset V$ be a full-rank lattice and let $U$ be Haar-uniform on
$V/\Gamma$.  Let $(Y_\gamma)_{\gamma\in\Gamma}$ be iid $V$-valued random
variables, independent of $U$, with common law $\nu$.  The random measure
\begin{equation}\label{eq:iid-perturbed-lattice}
  p_{U,Y}:=\sum_{\gamma\in\Gamma}
     \delta_{\gamma+U+Y_\gamma}
\end{equation}
defines a translation-invariant random measure; if the support of $\nu$ is
sufficiently small compared with the shortest nonzero vector of $\Gamma$,
it is a simple point process.  We denote its law by $\mu_{\Gamma,\nu}$ and
call it the \emph{iid-perturbed lattice} associated with $(\Gamma,\nu)$.
A deterministic translation of the support of $\nu$ can be absorbed into
$U$, so below we may assume that $\operatorname{supp}(\nu)\subset B_a$.

\begin{theorem}\label{thm:perturbed-lattice-separation}
For every full-rank lattice $\Gamma\subset V$ there exists $a_\Gamma>0$ such
that, whenever $\nu$ is non-degenerate and
$\operatorname{supp}(\nu)\subset B_a$ for some $a<a_\Gamma$, one has
\[
  \mu_{\Gamma,\nu}\in\HU_1\setminus\HU_2.
\]
\end{theorem}

\begin{proof}
Set $\rho:=\operatorname{covol}(\Gamma)^{-1}$ and
\[
  \widehat\nu(\xi):=\int_V e^{-2\pi i\xi(y)}\,d\nu(y),
  \qquad \xi\in V^*.
\]
For an iid-displaced lattice, the first Bartlett spectrum is
\begin{equation}\label{eq:perturbed-lattice-first-spectrum}
  \boldsymbol\sigma^{(1)}
  =\rho\bigl(1-|\widehat\nu|^2\bigr)\,\Vol_d
   +\rho^2\sum_{\xi\in\Gamma^*\setminus\{0\}}
      |\widehat\nu(\xi)|^2\,\delta_\xi .
\end{equation}
Indeed, expanding the covariance separates the terms with equal lattice
labels from those with distinct labels; averaging over $U$ gives the dual
lattice sum, while independence of the displacements contributes the factor
$|\widehat\nu|^2$ off the diagonal.
Since $\nu$ has bounded support,
\begin{equation}\label{eq:perturbed-lattice-characteristic-small-frequency}
  1-|\widehat\nu(\xi)|^2=O(\|\xi\|^2)
  \qquad(\xi\to0).
\end{equation}
The nonzero points of $\Gamma^*$ are separated from the origin, and hence
\[
  \boldsymbol\sigma^{(1)}(B_\varepsilon^*)
  =O(\varepsilon^{d+2})=o(\varepsilon^d).
\]
Theorem~\ref{thm:geometric-spectral} therefore gives
$\mu_{\Gamma,\nu}\in\HU_1$.

Fix $\gamma_0\in\Gamma\setminus\{0\}$.  Since $\Gamma$ is discrete, after
reducing $a_\Gamma$ if necessary, the $2a_\Gamma$-neighbourhoods of
$\gamma_0$ and $-\gamma_0$ contain no lattice vectors other than
$\gamma_0$ and $-\gamma_0$, respectively.  Since $\nu$ is non-degenerate,
there is $\ell\in V^*$ such that $\ell(Y_0)$ is non-degenerate.  Choose an
even function $u\in C_c(V)$ supported in these two neighbourhoods and such
that
\begin{equation}\label{eq:perturbed-lattice-shape-test}
  u(\gamma_0+z)=\ell(z)^2,
  \qquad
  u(-\gamma_0-z)=\ell(z)^2,
  \qquad \|z\|\leq2a.
\end{equation}
Let $\varphi\in C_c(\Sh_2(V))$ be determined by
\[
  \varphi([(-z/2,z/2)])=u(z).
\]
For $a<a_\Gamma$, this function sees precisely the pairs whose lattice
labels differ by $\pm\gamma_0$.

Write
\begin{equation}\label{eq:perturbed-lattice-Zgamma}
  Z_\gamma
  :=\ell(Y_{\gamma+\gamma_0}-Y_\gamma)^2,
  \qquad
  b_\gamma
  :=\gamma+U+\frac{\gamma_0}{2}
    +\frac{Y_\gamma+Y_{\gamma+\gamma_0}}2 .
\end{equation}
Both orientations of each pair occur in $T_2$, and therefore
\begin{equation}\label{eq:perturbed-lattice-pair-statistic}
  T_2[\chi_{B_R},\varphi]
  =2\sum_{\gamma\in\Gamma}Z_\gamma\chi_{B_R}(b_\gamma).
\end{equation}
Let $W_\gamma:=\ell(Y_\gamma)-\mathbb E[\ell(Y_\gamma)]$ and
$m_j:=\mathbb E[W_0^j]$.  Then $m_4>0$, and a direct computation gives
\begin{align}\label{eq:perturbed-lattice-covariances}
  \Var(Z_0)&=2m_4+2m_2^2,\notag\\
  \Cov(Z_0,Z_{\gamma_0})
     &=\Cov(Z_0,Z_{-\gamma_0})=m_4-m_2^2,\notag\\
  \Cov(Z_0,Z_\gamma)&=0,
  \qquad \gamma\notin\{0,\pm\gamma_0\}.
\end{align}

Condition on $U$ and introduce the statistic with unperturbed barycentres
\begin{equation}\label{eq:perturbed-lattice-reference-statistic}
  A_R:=2\sum_{\gamma\in\Gamma}Z_\gamma
       \chi_{B_R}\!\left(\gamma+U+\frac{\gamma_0}{2}\right).
\end{equation}
Uniformly in $U$, the number of lattice points in $B_R$ and the number of
pairs of such points separated by $\gamma_0$ are both
\begin{equation}\label{eq:perturbed-lattice-window-count}
  \rho\Vol_d(B_R)+O(R^{d-1}).
\end{equation}
Using \eqref{eq:perturbed-lattice-covariances}, we obtain
\begin{equation}\label{eq:perturbed-lattice-reference-variance}
  \Var(A_R\mid U)
  =16\rho m_4\Vol_d(B_R)+O(R^{d-1}).
\end{equation}
On the other hand,
$\|b_\gamma-(\gamma+U+\gamma_0/2)\|\leq a$.  Hence
$T_2[\chi_{B_R},\varphi]-A_R$ involves only those $\gamma$ whose reference
barycentre lies in a fixed-width neighbourhood of $\partial B_R$.  There
are $O(R^{d-1})$ such indices, and the summands have a dependency graph of
uniformly bounded degree.  Since the $Z_\gamma$ are bounded,
\begin{equation}\label{eq:perturbed-lattice-boundary-error}
  \Var\bigl(T_2[\chi_{B_R},\varphi]-A_R\mid U\bigr)
  =O(R^{d-1})
\end{equation}
uniformly in $U$.  Cauchy--Schwarz together with
\eqref{eq:perturbed-lattice-reference-variance} and
\eqref{eq:perturbed-lattice-boundary-error} yields
\begin{equation}\label{eq:perturbed-lattice-conditional-pair-variance}
  \Var\bigl(T_2[\chi_{B_R},\varphi]\mid U\bigr)
  =16\rho m_4\Vol_d(B_R)+o(R^d).
\end{equation}
Finally, the law of total variance gives
\[
  \liminf_{R\to\infty}
  \frac{\Var_{\mu_{\Gamma,\nu}}
     (T_2[\chi_{B_R},\varphi])}{\Vol_d(B_R)}
  \geq16\rho m_4>0.
\]
Thus $\mu_{\Gamma,\nu}\notin\HU_2$.
\end{proof}

\section{Poisson and determinantal point processes}
\label{sec:determinantal-poisson}

The homogeneous Poisson process does not belong to $\HU_1$, whereas the
projection determinantal processes considered below belong to $\HU_1$ but not
to $\HU_2$.

\subsection{Poisson processes}
\label{subsec:poisson}

Let $\mu_{\mathrm P,\rho}$ denote the $V$-invariant Poisson point process of
intensity $\rho>0$ relative to $\Vol_d$.  For our Fourier normalization, its
first Bartlett spectrum is $\rho\,\Vol_d$.

\begin{proposition}\label{prop:poisson-not-hu1}
The first Bartlett spectrum of $\mu_{\mathrm P,\rho}$ is
$\rho\,\Vol_d$ on $V^*$.  In particular,
$\mu_{\mathrm P,\rho}\notin\HU_1$.
\end{proposition}

\begin{proof}
For $f,g\in C_c(V)$, the Poisson covariance identity gives
\begin{equation}\label{eq:poisson-covariance}
  \Cov_{\mu_{\mathrm P,\rho}}(Sf,Sg)
  =\rho\int_V f(v)\overline{g(v)}\,d\Vol_d(v)
  =\rho\int_{V^*}\widehat f(\xi)
       \overline{\widehat g(\xi)}\,d\Vol_d(\xi),
\end{equation}
where the second equality is Plancherel's identity.  Hence
\begin{equation}\label{eq:poisson-bartlett}
  \boldsymbol\sigma^{(1)}=\rho\,\Vol_d.
\end{equation}
Equivalently, for every $R>0$,
\begin{equation}\label{eq:poisson-number-variance}
  \Var_{\mu_{\mathrm P,\rho}}\big(S\chi_{B_R}\big)
  =\rho\,\Vol_d(B_R).
\end{equation}
Thus the normalized number variance is identically $\rho$.
\end{proof}

\subsection{Determinantal point processes}
\label{subsec:dpp}

A determinantal point process is $V$-invariant whenever translation of its
kernel changes the kernel only by a gauge transformation.  We therefore
assume that, for every $t\in V$, there is a measurable function
$c_t:V\to\mathbb T$, where $\mathbb T$ is the complex unit circle, such that
\begin{equation}\label{eq:dpp-phase-covariance}
  K(v+t,w+t)=c_t(v)\overline{c_t(w)}K(v,w),
  \qquad v,w\in V.
\end{equation}
Simultaneous translation of the arguments conjugates every correlation
matrix by a diagonal unitary matrix and leaves its determinant unchanged.
Thus \eqref{eq:dpp-phase-covariance} is sufficient for invariance of the
associated determinantal process.
Translation-invariant kernels correspond to the special case $c_t\equiv1$;
the gauge-covariant formulation also includes Weyl--Heisenberg kernels.

For a kernel satisfying \eqref{eq:dpp-phase-covariance}, set
\begin{equation}\label{eq:dpp-h-function}
  \rho:=K(0,0),
  \qquad
  h_K(v):=|K(0,v)|^2,
  \qquad v\in V.
\end{equation}
Taking absolute values in \eqref{eq:dpp-phase-covariance} gives
\begin{equation}\label{eq:dpp-modulus-difference}
  |K(v,w)|^2=h_K(w-v),
  \qquad v,w\in V.
\end{equation}
If $P_K$ is an orthogonal projection, the reproducing identity gives
\begin{equation}\label{eq:dpp-projection-identity-ae}
  \int_V |K(v,w)|^2\,d\Vol_d(w)=K(v,v)
\end{equation}
for almost every $v\in V$.  By \eqref{eq:dpp-modulus-difference}, the
left-hand side equals $\int_Vh_K(r)\,d\Vol_d(r)$ and is independent of
$v$, while $K(v,v)=\rho$ for every $v$.  Therefore
\begin{equation}\label{eq:dpp-projection-identity}
  h_K\in L^1(V),
  \qquad
  \int_Vh_K(r)\,d\Vol_d(r)=\rho.
\end{equation}

\begin{lemma}[Pair-variance density for projection DPPs]
\label{lem:dpp-pair-variance-density}
Assume $d\geq1$.  Let $K:V\times V\to\mathbb C$ be a continuous Hermitian locally trace-class
kernel whose integral operator is an orthogonal projection on $L^2(V)$.
Assume \eqref{eq:dpp-phase-covariance} and $\rho>0$.
Let $\mu_K$ be the associated determinantal point process.  Write $\rho^{(j)}$ for the factorial correlation densities and set
\begin{equation}\label{eq:dpp-pair-density}
  q(r):=\rho^{(2)}(0,r)=\rho^2-h_K(r),
  \qquad r\in V.
\end{equation}
For a real even $\psi\in C_c(V)$, let $\varphi_\psi\in C_c(\Sh_2(V))$ be given by
\begin{equation}\label{eq:dpp-pair-shape-even-lift}
  \varphi_\psi\big([(-r/2,r/2)]\big)=\psi(r),
  \qquad r\in V,
\end{equation}
and set
\begin{equation}\label{eq:dpp-pair-observable}
  A_R:=T_2^{\neq}[\chi_{B_R},\varphi_\psi].
\end{equation}
Then the limit
\begin{equation}\label{eq:dpp-pair-density-limit-short}
  D_K(\psi):=\lim_{R\to\infty}
  \frac{\Var_{\mu_K}(A_R)}{\Vol_d(B_R)}
\end{equation}
exists.  If
\begin{equation}\label{eq:dpp-connected-four-point}
  \kappa^{(4)}_K(r,t,s)
  :=\rho^{(4)}(0,r,t,t+s)
    -\rho^{(2)}(0,r)\rho^{(2)}(0,s),
\end{equation}
then
\begin{align}\label{eq:dpp-general-pair-variance-limit}
  D_K(\psi)
  ={}&\int_{V^3}\psi(r)\psi(s)\,
       \kappa^{(4)}_K(r,t,s)\,
       d\Vol_d(r)\,d\Vol_d(s)\,d\Vol_d(t)\notag\\
   &+4\int_{V^2}\psi(r)\psi(s)\rho^{(3)}(0,r,s)\,
       d\Vol_d(r)\,d\Vol_d(s)\notag\\
   &+2\int_V \psi(r)^2q(r)\,d\Vol_d(r).
\end{align}
Moreover,
\begin{equation}\label{eq:dpp-connected-four-point-bound}
  \sup_{r,s\in V}\int_V
       |\kappa^{(4)}_K(r,t,s)|\,d\Vol_d(t)
  \leq20\rho^3.
\end{equation}
\end{lemma}

\begin{proof}
\smallskip\noindent\emph{Variance decomposition.}
Recall that Hadamard's inequality for a positive semidefinite
Hermitian matrix $A=(a_{rs})$ states that
$\det A\leq\prod_r a_{rr}$.  The matrices $(K(v_r,v_s))_{r,s}$ are positive semidefinite, so applying this
to the determinantal correlation matrices and using $K(v,v)=\rho$ gives
$0\leq\rho^{(j)}\leq\rho^j$, so all local moments are finite.  For the
statistic in \eqref{eq:dpp-pair-observable},
\begin{equation}\label{eq:dpp-pair-observable-sum}
  A_R
  =\sum_{v\neq w\in\Lambda_p}
       \chi_{B_R}\!\left(\frac{v+w}{2}\right)\psi(w-v).
\end{equation}
Set
\[
  H_R(v,w):=\chi_{B_R}\!\left(\frac{v+w}{2}\right)\psi(w-v).
\]
Expanding $A_R^2$ according to whether the two ordered pairs contain four,
three, or two distinct points gives multiplicities $1$, $4$, and $2$.
After subtracting $(\mathbb EA_R)^2$, this gives the exact decomposition
\begin{equation}\label{eq:dpp-pair-variance-decomposition}
  \Var(A_R)=I_{4,R}+4I_{3,R}+2I_{2,R},
\end{equation}
where
\begin{align}
I_{4,R}
  :={}&\int_{V^4}H_R(v_1,v_2)H_R(v_3,v_4)\notag\\
     &\qquad\times
       \big(\rho^{(4)}(v_1,v_2,v_3,v_4)
       -\rho^{(2)}(v_1,v_2)\rho^{(2)}(v_3,v_4)\big)
       \,d\Vol_d^{\otimes4},\label{eq:dpp-I4R}\\
I_{3,R}
  :={}&\int_{V^3}H_R(x,y)H_R(x,z)
       \rho^{(3)}(x,y,z)\,d\Vol_d^{\otimes3},\label{eq:dpp-I3R}\\
I_{2,R}
  :={}&\int_{V^2}H_R(x,y)^2
       \rho^{(2)}(x,y)\,d\Vol_d^{\otimes2}.\label{eq:dpp-I2R}
\end{align}
The four three-point identifications give the same integral because $\psi$ is
even, and the two two-point identifications correspond to the same ordered
pair and its reversal.

\smallskip\noindent\emph{Barycentric coordinates and boundary factors.}
For $a\in V$, put
\begin{equation}\label{eq:dpp-overlap-function}
  \omega_R(a)
  :=\frac{\Vol_d(B_R\cap(B_R-a))}{\Vol_d(B_R)}.
\end{equation}
For the four-point term, use the change of variables
\begin{equation}\label{eq:dpp-four-point-coordinates}
  (v_1,v_2,v_3,v_4)
  =\big(b-r/2,b+r/2,b-r/2+t,b-r/2+t+s\big).
\end{equation}
The two barycentres are then $b$ and $b+t+(s-r)/2$, so the window factor is
\begin{equation}\label{eq:dpp-four-point-overlap}
  \frac{
    \Vol_d\big(B_R\cap(B_R-t-(s-r)/2)\big)
  }{\Vol_d(B_R)}
  \longrightarrow 1,
  \qquad R\to\infty,
\end{equation}
for every fixed $r,s,t$.  In the three-point term the corresponding factor is
\begin{equation}\label{eq:dpp-three-point-overlap}
  \frac{
    \Vol_d\big(B_R\cap(B_R-(s-r)/2)\big)
  }{\Vol_d(B_R)}
  \longrightarrow 1,
  \qquad R\to\infty,
\end{equation}
while the two-point term has no boundary factor after changing to the
barycentre variable.  Consequently,
\begin{align}
  \frac{I_{4,R}}{\Vol_d(B_R)}
  ={}&\int_{V^3}\psi(r)\psi(s)\kappa_K^{(4)}(r,t,s)
       \omega_R\!\left(t+\frac{s-r}{2}\right)
       \,d\Vol_d(r)\,d\Vol_d(s)\,d\Vol_d(t),
       \label{eq:dpp-I4R-overlap}\\
  \frac{I_{3,R}}{\Vol_d(B_R)}
  ={}&\int_{V^2}\psi(r)\psi(s)\rho^{(3)}(0,r,s)
       \omega_R\!\left(\frac{s-r}{2}\right)
       \,d\Vol_d(r)\,d\Vol_d(s),
       \label{eq:dpp-I3R-overlap}\\
  \frac{I_{2,R}}{\Vol_d(B_R)}
  ={}&\int_V\psi(r)^2q(r)\,d\Vol_d(r).
       \label{eq:dpp-I2R-overlap}
\end{align}
For every fixed $a\in V$, $0\leq\omega_R(a)\leq1$ and
$\omega_R(a)\to1$ as $R\to\infty$.  The three- and two-point terms can
therefore be passed to the limit directly.  For the four-point term we need
an integrable bound on $\kappa_K^{(4)}$ uniform in $r$ and $s$.

\smallskip\noindent\emph{The connected four-point term.}
Expand the $4\times4$ determinant defining
$\rho^{(4)}(0,r,t,t+s)$.  In the difference
\eqref{eq:dpp-connected-four-point}, the four permutation terms preserving the
two blocks $\{1,2\}$ and $\{3,4\}$ cancel.  Every one of the remaining
twenty permutations contains a kernel factor from the first block to the
second and another from the second block to the first.  Cauchy--Schwarz for
the positive kernel gives
\begin{equation}\label{eq:dpp-pointwise-kernel-bound}
  |K(v,w)|^2\leq K(v,v)K(w,w)=\rho^2,
  \qquad v,w\in V.
\end{equation}
For the two cross-block factors, \eqref{eq:dpp-projection-identity} and
Cauchy--Schwarz give, for fixed $a,b,\alpha,\beta\in V$,
\begin{align}\label{eq:dpp-cross-block-bound}
  &\int_V |K(a,t+\alpha)K(t+\beta,b)|\,d\Vol_d(t)\notag\\
  &\qquad\leq
  \left(\int_V|K(a,t+\alpha)|^2\,d\Vol_d(t)\right)^{1/2}
  \left(\int_V|K(b,t+\beta)|^2\,d\Vol_d(t)\right)^{1/2}
  =\rho.
\end{align}
The other two factors are bounded by $\rho$.  Summing the twenty
contributions gives \eqref{eq:dpp-connected-four-point-bound}.  Since $\psi$ is
compactly supported, this gives an integrable majorant for
\eqref{eq:dpp-I4R-overlap}, while the three-point integrand is bounded by
$\rho^3|\psi(r)\psi(s)|$.  Dominated convergence in
\eqref{eq:dpp-I4R-overlap}--\eqref{eq:dpp-I2R-overlap} therefore gives
\eqref{eq:dpp-general-pair-variance-limit}.

\end{proof}

\begin{theorem}[Projection determinantal processes]
\label{thm:projection-dpp-hu1-not-hu2}
Assume $d\geq1$.  Let $K:V\times V\to\mathbb C$ be a continuous Hermitian locally trace-class
kernel whose integral operator $P_K$ is an orthogonal projection on $L^2(V)$.
Assume \eqref{eq:dpp-phase-covariance} and $\rho>0$.  Then the associated
determinantal point process $\mu_K$ is $V$-invariant and
\begin{equation}\label{eq:dpp-hierarchy-separation}
  \mu_K\in\HU_1\setminus\HU_2.
\end{equation}
More precisely, the first Bartlett spectrum has the continuous density
\begin{equation}\label{eq:dpp-first-bartlett-density}
  s_K(\xi)=\rho-\widehat h_K(\xi),
  \qquad \xi\in V^*,
\end{equation}
with $s_K(0)=0$.  There is also a nonzero
$\varphi\in C_c(\Sh_2(V))$, with $\varphi(0)=0$, such that
\begin{equation}\label{eq:dpp-positive-pair-volume}
  \lim_{R\to\infty}
  \frac{\Var_{\mu_K}\big(T_2[\chi_{B_R},\varphi]\big)}
       {\Vol_d(B_R)}>0.
\end{equation}
\end{theorem}

\begin{proof}
The Macchi--Soshnikov criterion \cite[Theorem~3]{Sos00} gives a unique simple determinantal point process associated with $P_K$.  Setting $v=w$ in \eqref{eq:dpp-phase-covariance} shows that
$K(v,v)=\rho$ for every $v\in V$.  Simultaneous translation of
$v_1,\ldots,v_j$ conjugates the matrix $[K(v_r,v_s)]_{r,s}$ by a diagonal
unitary matrix.  Equivalently, the $r$-th row acquires the factor
$c_t(v_r)$ and the $s$-th column the factor $\overline{c_t(v_s)}$; their
products cancel in the determinant.  Hence all factorial correlation
functions are $V$-invariant, and therefore so is $\mu_K$.

The determinantal one- and two-point identities give, for $f,g\in C_c(V)$,
\begin{align}\label{eq:dpp-first-covariance}
  \Cov_{\mu_K}(Sf,Sg)
  ={}&\rho\int_V f(v)\overline{g(v)}\,d\Vol_d(v)\notag\\
     &-\int_{V^2}f(v)\overline{g(w)}h_K(w-v)\,
       d\Vol_d(v)\,d\Vol_d(w).
\end{align}
Thus the first Bartlett density is \eqref{eq:dpp-first-bartlett-density}.
Since $h_K\in L^1(V)$, this density is continuous, and
\begin{equation}\label{eq:dpp-zero-frequency}
  s_K(0)=\rho-\int_V h_K(v)\,d\Vol_d(v)=0.
\end{equation}
Theorem~\ref{thm:geometric-spectral} now gives $\mu_K\in\HU_1$.

For the second assertion, let $q$ be as in
\eqref{eq:dpp-pair-density}.  It is continuous, even and nonnegative.  Since $h_K\in L^1(V)$ and $h_K(0)=\rho^2$, there is an
$a\neq0$ with $q(a)>0$.

Choose $c>0$ and a symmetric relatively compact open set
$O\subset V\setminus\{0\}$, containing $a$ and $-a$, so small that
\begin{equation}\label{eq:dpp-small-pair-support}
  q(r)\geq c\quad(r\in O),
  \qquad
  \Vol_d(O)<\frac{c}{10\rho^3}.
\end{equation}
Take a nonzero even function $\psi\in C_c(V)$ with $\psi\geq0$ and
$\operatorname{supp}\psi\subset O$.  Lemma~\ref{lem:dpp-pair-variance-density} gives the limit formula
\eqref{eq:dpp-general-pair-variance-limit} and the bound
\eqref{eq:dpp-connected-four-point-bound}.  Since the three-point term is
nonnegative, Cauchy--Schwarz on the support of $\psi$ gives
\begin{align}\label{eq:dpp-pair-positive-lower-bound}
  D_K(\psi)
  &\geq2c\|\psi\|_2^2-20\rho^3\|\psi\|_1^2\notag\\
  &\geq\big(2c-20\rho^3\Vol_d(O)\big)\|\psi\|_2^2>0.
\end{align}
Because $0\notin O$, we have $\varphi_\psi(0)=\psi(0)=0$.  The original and distinct-point
pair statistics therefore agree by \eqref{eq:full-factorial-pair-identity},
and \eqref{eq:dpp-pair-positive-lower-bound} proves
\eqref{eq:dpp-positive-pair-volume}.  Hence $\mu_K\notin\HU_2$.
\end{proof}

\begin{example}[Weyl--Heisenberg ensembles]
Let $V=\mathbb R^{2m}$ and write $z=(x,\xi)\in\mathbb R^m\times\mathbb R^m$.
For a unit vector $g\in L^2(\mathbb R^m)$, let
\begin{equation}\label{eq:wh-time-frequency-shift}
  (\pi(x,\xi)g)(t):=e^{2\pi i\,\xi\cdot t}g(t-x),
  \qquad t\in\mathbb R^m,
\end{equation}
and consider the Weyl--Heisenberg kernel
\begin{equation}\label{eq:wh-kernel}
  K_g(z,w):=\langle \pi(w)g,\pi(z)g\rangle_{L^2(\mathbb R^m)}.
\end{equation}
It is the reproducing kernel of the range of the short-time Fourier transform,
so its integral operator $P_{K_g}$ is an orthogonal projection.  For every
relatively compact $D\subset V$,
\[
  \|P_{K_g}\chi_D\|_{\mathrm{HS}}^2
  =\int_D\int_V |K_g(z,w)|^2\,d\Vol_d(z)\,d\Vol_d(w)
  =\int_D K_g(w,w)\,d\Vol_d(w)
  =\Vol_d(D),
\]
so $\chi_D P_{K_g}\chi_D=(P_{K_g}\chi_D)^*(P_{K_g}\chi_D)$ is trace
class.  Thus $K_g$ is locally trace class.  The projective composition law for
the operators $\pi(z)$ gives \eqref{eq:dpp-phase-covariance}.
Theorem~\ref{thm:projection-dpp-hu1-not-hu2} therefore yields
\begin{equation}\label{eq:wh-hierarchy-separation}
  \mu_{K_g}\in\HU_1\setminus\HU_2.
\end{equation}
This includes the Ginibre process for a Gaussian window and the pure
polyanalytic, or higher Landau-level, processes for Hermite windows.  Their
first-order hyperuniformity is standard; under the decay assumption in
\cite[Definition~5.1]{APRT17}, surface-order number variance and its sharp
asymptotic are proved in \cite[Theorems~5.6 and~5.8]{APRT17}, with the pure
polyanalytic case recorded in \cite[Corollary~5.11]{APRT17}.
\end{example}

\begin{remark}[Comparison with Frommer--Hanke]
Let $E\subset V^*$ be measurable with $E=-E$ and $\Vol_d(E)=1$.  If
$\mathcal F$ denotes the Fourier transform from
\eqref{eq:fourier-convention} and $M_{\chi_E}$ multiplication by $\chi_E$,
then the Fourier projection
\begin{equation}\label{eq:fourier-projection}
  P_E:=\mathcal F^{-1}M_{\chi_E}\mathcal F
\end{equation}
has the translation-invariant kernel
\begin{equation}\label{eq:fourier-projection-kernel}
  K(v,w)
  =\int_E e^{2\pi i\xi(v-w)}\,d\Vol_d(\xi),
  \qquad v,w\in V.
\end{equation}
For these projection DPPs, Frommer and Hanke prove positive volume-order
variance for every nonzero even
$\psi\in L^1(\mathbb R^d)\cap L^2(\mathbb R^d)$ in their endpoint-truncated pair
statistic; see \cite[Theorem~4.2]{FH26}.  Theorem~\ref{thm:projection-dpp-hu1-not-hu2}
gives the corresponding failure of $\HU_2$ for the larger class satisfying
\eqref{eq:dpp-phase-covariance}, with the barycentre truncation used here.
\end{remark}

\section{Strictness of the hierarchy and weak mixing}
\label{sec:weak-mixing}

By
\cite[Theorem~4.1]{Bjo26a}, if an ergodic point process on $\mathbb R$ has
positive intensity and local second moments, then the condition
\begin{equation}\label{eq:wm-bjo-spectral-criterion}
  \int_{0<|\xi|<1}\frac{1}{\xi^2}\,
    d\boldsymbol\sigma^{(1)}(\xi)<\infty,
\end{equation}
where $\boldsymbol\sigma^{(1)}$ is its scalar first Bartlett spectrum, forces
a nonzero eigenvalue and lattice induction.  Theorem~\ref{thm:weak-mixing-strictness}
therefore shows that no finite collection $\HU_1,\ldots,\HU_k$ forces
\eqref{eq:wm-bjo-spectral-criterion}.

\begin{theorem}[Strictness under weak mixing]
\label{thm:weak-mixing-strictness}
For every integer $k\geq1$ there exists an $\mathbb R$-invariant weakly
mixing simple point process $\mu$ such that
\begin{equation}\label{eq:weak-mixing-strictness}
  \mu\in\HU_{\leq k}\setminus\HU_{k+1}.
\end{equation}
\end{theorem}

Fix $q\geq2$.  It suffices to construct a weakly mixing point process
$\mu_q\in\HU_{\leq q-1}\setminus\HU_q$; then take $q=k+1$.

\subsection{The symbolic construction}
\label{subsec:wm-substitution}

\subsubsection{Substitution systems}
Let $\mathsf A$ be a finite alphabet and let $Q\geq2$.
A \emph{constant-length substitution of length $Q$} is a map
$\vartheta:\mathsf A\to\mathsf A^Q$.  It extends to finite words by
concatenation.  For $x=(x_n)_{n\in\mathbb Z}\in\mathsf A^{\mathbb Z}$,
every integer $k$ has a unique decomposition
\[
  k=Qn+j,\qquad n\in\mathbb Z,\quad 0\leq j<Q,
\]
and we define
\begin{equation}\label{eq:wm-bi-infinite-substitution}
  \bigl(\vartheta(x)\bigr)_k
  :=\bigl(\vartheta(x_n)\bigr)_j.
\end{equation}
Its \emph{substitution matrix}
$\mathbf M_\vartheta=(M_{a,b})_{a,b\in\mathsf A}$ is defined by
\begin{equation}\label{eq:wm-substitution-matrix-definition}
  M_{a,b}
  :=\#\{0\leq j<Q:(\vartheta(b))_j=a\}.
\end{equation}
The substitution is \emph{primitive} if some power of
$\mathbf M_\vartheta$ has strictly positive entries.

Let $T$ denote the left shift on $\mathsf A^{\mathbb Z}$.  The
\emph{substitution subshift} is
\begin{equation}\label{eq:wm-substitution-subshift}
  X_\vartheta
  :=\{x\in\mathsf A^{\mathbb Z}:\text{ every finite word occurring in $x$
  occurs in some $\vartheta^m(a)$}\}.
\end{equation}
A point $x\in X_\vartheta$ is \emph{periodic} if $T^p x=x$ for some
$p\geq1$, and $(X_\vartheta,T)$ is \emph{aperiodic} if it contains no
periodic point.  A constant-length substitution is \emph{recognizable} on
$X_\vartheta$ if the representation
\begin{equation}\label{eq:wm-recognizability-definition}
  x=T^j\vartheta(y),
  \qquad y\in X_\vartheta,\quad 0\leq j<Q,
\end{equation}
is unique whenever it exists.  For a primitive substitution,
$X_\vartheta=\bigcup_{j=0}^{Q-1}T^j\vartheta(X_\vartheta)$, so every
$x\in X_\vartheta$ admits a representation \eqref{eq:wm-recognizability-definition}.
If $\vartheta$ is recognizable, set $x^{(0)}:=x$ and write successively
\[
  x^{(r)}=T^{j_r}\vartheta(x^{(r+1)}),
  \qquad r\geq0,\quad 0\leq j_r<Q.
\]
Using $\vartheta(T^j y)=T^{Qj}\vartheta(y)$ gives, for every $m\geq1$,
\begin{equation}\label{eq:wm-level-m-representation}
  x=T^{j_m}\vartheta^m(y^{(m)}),
  \qquad
  j_m:=j_0+Qj_1+\cdots+Q^{m-1}j_{m-1},
  \qquad 0\leq j_m<Q^m,
\end{equation}
where $y^{(m)}:=x^{(m)}$.  The one-step uniqueness implies uniqueness of
$(y^{(m)},j_m)$.

For $a\in\mathsf A$, the word $\vartheta^m(a)$ has length $Q^m$; we call it
a \emph{level-$m$ substitution word}.  If
\eqref{eq:wm-level-m-representation} holds, then for every $n\in\mathbb Z$
and $0\leq r<Q^m$,
\begin{equation}\label{eq:wm-level-m-coordinate-block}
  x_{-j_m+nQ^m+r}
  =\bigl(\vartheta^m(y^{(m)}_n)\bigr)_r.
\end{equation}
Hence the intervals
\begin{equation}\label{eq:wm-level-m-blocks-definition}
  I_{m,n}(x)
  :=[-j_m+nQ^m,-j_m+(n+1)Q^m)\cap\mathbb Z,
  \qquad n\in\mathbb Z,
\end{equation}
partition $\mathbb Z$ into copies of level-$m$ substitution words.  We call
them the \emph{level-$m$ blocks of $x$}.  Since
$\vartheta^{m+1}=\vartheta^m\circ\vartheta$, the block partitions are nested:
each level-$(m+1)$ block is the union of $Q$ consecutive level-$m$ blocks.

Every primitive substitution subshift is uniquely ergodic; see
\cite[Theorem~5.6]{Que10}.  For primitive aperiodic injective substitutions,
recognizability follows from Moss\'e's theorem in the form
\cite[Theorem~4.34]{Bru22}; in the present constant-length setting this is
precisely uniqueness in \eqref{eq:wm-recognizability-definition}, whose
existence was noted above.

A function $F:X_\vartheta\to\mathbb C$ is a \emph{finite-coordinate
function} if there is $\ell\geq0$ such that $F(x)$ depends only on
$x_{-\ell},\ldots,x_\ell$.  Such an $\ell$ will be called a
\emph{dependency radius} for $F$.

\subsubsection{The finite-group substitution}
For the construction, set
\begin{equation}\label{eq:wm-finite-group}
  \mathsf A:=\mathbb F_2^q\times\mathbb F_2,
  \qquad M:=|\mathsf A|=2^{q+1},
  \qquad Q:=M^4.
\end{equation}
Write $z=(g,s)$ with $g=(g_1,\ldots,g_q)\in\mathbb F_2^q$ and
$s\in\mathbb F_2$.  We identify the character group
$\widehat{\mathsf A}$ with $\mathsf A$ by
\begin{equation}\label{eq:wm-character-identification}
  \chi_{(u,t)}(g,s):=(-1)^{u\cdot g+ts}.
\end{equation}
Set
\begin{equation}\label{eq:wm-distinguished-characters}
  \alpha:=((1,\ldots,1),0),
  \qquad
  \beta:=(0,1).
\end{equation}
Thus $\chi_\alpha(g,s)=(-1)^{g_1+\cdots+g_q}$ depends on all $q$
coordinates of $g$, whereas $\chi_\beta(g,s)=(-1)^s$ depends only on the
last coordinate.

We construct a substitution of the form
\begin{equation}\label{eq:wm-translation-form}
  \vartheta(z)=(z+c_0)(z+c_1)\cdots(z+c_{Q-1}),
  \qquad c_j\in\mathsf A.
\end{equation}
If
\[
  m_{\mathbf c}(w):=\#\{0\leq j<Q:c_j=w\},
\]
then
\begin{equation}\label{eq:wm-convolution-matrix}
  (\mathbf M_\vartheta)_{a,b}=m_{\mathbf c}(a-b),
\end{equation}
so the substitution matrix is convolution by $m_{\mathbf c}$ on the finite
abelian group $\mathsf A$.  Moreover, for
\begin{equation}\label{eq:wm-character-multiplier-predefinition}
  S_\gamma:=\sum_{j=0}^{Q-1}\chi_\gamma(c_j)
  =\sum_{w\in\mathsf A}m_{\mathbf c}(w)\chi_\gamma(w),
\end{equation}
one has
\begin{equation}\label{eq:wm-one-step-character-sum}
  \sum_{j=0}^{Q-1}\chi_\gamma\bigl((\vartheta(z))_j\bigr)
  =S_\gamma\chi_\gamma(z).
\end{equation}
Thus $S_\gamma$ is the finite-group Fourier coefficient of
$m_{\mathbf c}$ at $\gamma$.

We require
\begin{equation}\label{eq:wm-target-character-multipliers}
  S_0=Q=M^4,
  \qquad S_\alpha=M^3,
  \qquad S_\beta=M,
  \qquad S_\gamma=0\quad
  (\gamma\notin\{0,\alpha,\beta\}).
\end{equation}
Since $\mathsf A$ is a finite abelian group, its characters are orthogonal;
in the present $\mathbb F_2$ setting,
\begin{equation}\label{eq:wm-character-orthogonality-special}
  \sum_{w\in\mathsf A}\chi_\gamma(w)=M\mathbf 1_{\{\gamma=0\}},
  \qquad
  \sum_{w\in\mathsf A}\chi_\eta(w)\chi_\gamma(w)
  =M\mathbf 1_{\{\gamma=\eta\}}.
\end{equation}
Finite-group Fourier inversion applied to
\eqref{eq:wm-target-character-multipliers} gives
\begin{equation}\label{eq:wm-multiplicity}
  m(w)
  :=\frac1M\bigl(M^4+M^3\chi_\alpha(w)+M\chi_\beta(w)\bigr)
  =M^3+M^2\chi_\alpha(w)+\chi_\beta(w).
\end{equation}
By \eqref{eq:wm-character-orthogonality-special}, its Fourier coefficients
are exactly \eqref{eq:wm-target-character-multipliers}.  Moreover,
\[
  m(w)\geq M^3-M^2-1>0,
\]
so every $m(w)$ is a positive integer, and
\begin{equation}\label{eq:wm-substitution-length}
  \sum_{w\in\mathsf A}m(w)=Q.
\end{equation}

It remains to choose the word $\mathbf c$ in
\eqref{eq:wm-translation-form}.  Since
$m(0)=M^3+M^2+1\geq2$, choose
\begin{equation}\label{eq:wm-multiplicity-word}
  \mathbf c=c_0c_1\cdots c_{Q-1}\in\mathsf A^Q,
  \qquad
  c_0=c_1=0,
  \qquad
  \#\{j:c_j=w\}=m(w)
\end{equation}
for every $w\in\mathsf A$; apart from the first two entries, the ordering is
irrelevant.  The multiplicity condition and
\eqref{eq:wm-character-multiplier-predefinition} give the multipliers in
\eqref{eq:wm-target-character-multipliers}.  The extra
condition $c_0=c_1=0$ does not change the multipliers and gives
\[
  \vartheta^{m+1}(a)
  =\vartheta^m(a+c_0)\,\vartheta^m(a+c_1)\cdots
  =\vartheta^m(a)\,\vartheta^m(a)\cdots,
\]
hence
\begin{equation}\label{eq:wm-two-copy-prefix}
  \vartheta^{m+1}(a)_{[0,2Q^m)}
  =\vartheta^m(a)\,\vartheta^m(a),
  \qquad a\in\mathsf A,\quad m\geq0.
\end{equation}

For $h\in\mathsf A$, let $R_h$ add $h$ to every coordinate of a finite or
bi-infinite word.  The additive group structure gives, coordinatewise,
\[
  (\vartheta(z+h))_j=z+h+c_j=(\vartheta(z))_j+h,
\]
and hence
\begin{equation}\label{eq:wm-substitution-translation-covariance}
  \vartheta(z+h)=R_h\vartheta(z).
\end{equation}
By concatenation,
$\vartheta(R_hx)=R_h\vartheta(x)$ on $\mathsf A^{\mathbb Z}$, and the same
identity holds for every iterate.  Hence a word $w$ occurs in some
$\vartheta^m(a)$ if and only if $R_hw$ occurs in
$\vartheta^m(a+h)$, so $R_hX_\vartheta=X_\vartheta$.  The maps $R_h$ also
commute with the shift $T$.

\begin{lemma}[Basic properties of the substitution]
\label{lem:wm-substitution-estimates}
The substitution $\vartheta$ satisfies the following statements.

\begin{enumerate}
\item The system $(X_\vartheta,T)$ is minimal, uniquely ergodic, aperiodic,
and recognizable.  Its unique invariant probability measure, denoted by
$\nu$, is invariant under every $R_h$.

\item For the multipliers $S_\gamma$ defined in
\eqref{eq:wm-character-multiplier-predefinition}, one has
\begin{equation}\label{eq:wm-character-multipliers}
  S_0=Q,
  \qquad
  S_\alpha=M^3=Q^{3/4},
  \qquad
  S_\beta=M=Q^{1/4},
  \qquad
  S_\gamma=0
\end{equation}
for $\gamma\notin\{0,\alpha,\beta\}$.  Moreover, for every
$a\in\mathsf A$, $m\geq0$, and $\gamma\in\widehat{\mathsf A}$,
\begin{equation}\label{eq:wm-exact-character-word-sum}
  \sum_{n=0}^{Q^m-1}
  \chi_\gamma\bigl((\vartheta^m(a))_n\bigr)
  =S_\gamma^m\chi_\gamma(a).
\end{equation}
\end{enumerate}
\end{lemma}

\begin{proof}
\smallskip\noindent\emph{Primitivity and symmetry.}
The substitution matrix is
\begin{equation}\label{eq:wm-specific-substitution-matrix}
  (\mathbf M_\vartheta)_{a,b}=m(a-b)>0.
\end{equation}
Thus $\vartheta$ is primitive.  Hence $(X_\vartheta,T)$ is minimal, and by
\cite[Theorem~5.6]{Que10} it is uniquely ergodic.  The maps $R_h$ preserve
$X_\vartheta$ and commute with $T$ by
\eqref{eq:wm-substitution-translation-covariance}; uniqueness of the
$T$-invariant probability measure therefore gives $(R_h)_*\nu=\nu$.

\smallskip\noindent\emph{Character multipliers.}
The values in \eqref{eq:wm-character-multipliers} are
\eqref{eq:wm-target-character-multipliers}, because the word $\mathbf c$
has multiplicity function $m$.  The group law gives
$\chi_\gamma(a+c_j)=\chi_\gamma(a)\chi_\gamma(c_j)$, so
\eqref{eq:wm-one-step-character-sum} is the case $m=1$ of
\eqref{eq:wm-exact-character-word-sum}.  If the latter identity holds at
level $m$, then
\[
  \vartheta^{m+1}(a)
  =\vartheta^m(a+c_0)\cdots\vartheta^m(a+c_{Q-1})
\]
and therefore
\[
  \sum_{n=0}^{Q^{m+1}-1}
    \chi_\gamma\bigl((\vartheta^{m+1}(a))_n\bigr)
  =S_\gamma^m\chi_\gamma(a)
     \sum_{j=0}^{Q-1}\chi_\gamma(c_j)
  =S_\gamma^{m+1}\chi_\gamma(a).
\]
This proves \eqref{eq:wm-exact-character-word-sum} by induction.

\smallskip\noindent\emph{Aperiodicity and recognizability.}
Suppose that $x\in X_\vartheta$ has period $p$.  Minimality then makes
$X_\vartheta$ the finite orbit of $x$.  Since $\nu$ is invariant under
$R_h$, the nontrivial character $\chi_\alpha$ has mean zero:
\[
  \int\chi_\alpha(x_0)\,d\nu(x)=0.
\]
The invariant measure on a periodic orbit is uniform, hence
\begin{equation}\label{eq:wm-periodic-alpha-sum-zero}
  \sum_{n=0}^{p-1}\chi_\alpha(x_n)=0.
\end{equation}
For any $r\in\mathbb Z$ and $N\geq1$, write $N=ap+b$ with
$0\leq b<p$.  Using \eqref{eq:wm-periodic-alpha-sum-zero} on the $a$
complete periods gives
\begin{equation}\label{eq:wm-periodic-alpha-bound}
  \left|\sum_{n=r}^{r+N-1}\chi_\alpha(x_n)\right|\leq p.
\end{equation}
On the other hand, primitivity implies that for every $a$ and $m$ the word
$\vartheta^m(a)$ occurs in some sufficiently high substitution word and
hence in the language of $X_\vartheta$.  Thus it occurs as a block of some
point of $X_\vartheta$, and \eqref{eq:wm-exact-character-word-sum} gives on
that block
\begin{equation}\label{eq:wm-alpha-word-growth}
  \left|
  \sum_{n=0}^{Q^m-1}
  \chi_\alpha\bigl((\vartheta^m(a))_n\bigr)
  \right|=M^{3m},
\end{equation}
contradicting \eqref{eq:wm-periodic-alpha-bound} for large $m$.  Thus the
subshift is aperiodic.  Moreover $\vartheta$ is injective, since
$c_0=0$ and the first letter of $\vartheta(z)$ is $z$.  Recognizability
follows from \cite[Theorem~4.34]{Bru22}.
\end{proof}

\subsubsection{Fourier decomposition for the finite-group action}
By \eqref{eq:wm-substitution-translation-covariance}, the action of
$\mathsf A$ by the maps $R_h$ commutes with the shift.  We decompose
observables according to this action.

For $h\in\mathsf A$, let
\[
  U_hF:=F\circ R_h,
  \qquad F\in L^2(X_\vartheta,\nu).
\]
Since $(R_h)_*\nu=\nu$, each $U_h$ is unitary.  For
$\gamma\in\widehat{\mathsf A}$ define
\begin{equation}\label{eq:wm-isotypic-projection}
  P_\gamma F
  :=\frac1M\sum_{h\in\mathsf A}
       \overline{\chi_\gamma(h)}\,U_hF.
\end{equation}
Let
\begin{equation}\label{eq:wm-character-subspace}
  \mathcal H_\gamma
  :=\{F\in L^2(\nu):U_hF=\chi_\gamma(h)F
       \text{ for every }h\in\mathsf A\}.
\end{equation}

\begin{lemma}[Finite-group Fourier projections]
\label{lem:wm-fourier-projections}
For every $\gamma\in\widehat{\mathsf A}$, $P_\gamma$ is the orthogonal
projection of $L^2(\nu)$ onto $\mathcal H_\gamma$.  Moreover,
\begin{equation}\label{eq:wm-fourier-decomposition}
  F=\sum_{\gamma\in\widehat{\mathsf A}}P_\gamma F
  \qquad(F\in L^2(\nu)),
\end{equation}
and the subspaces $\mathcal H_\gamma$ are pairwise orthogonal.  If $F$
depends only on $x_{-\ell},\ldots,x_\ell$, then every $P_\gamma F$ has the
same dependency radius.  Finally, $P_\gamma$ commutes with the shift
operator $F\mapsto F\circ T$.
\end{lemma}

\begin{proof}
For $k\in\mathsf A$, changing variables $h'=h+k$ in
\eqref{eq:wm-isotypic-projection} gives
\[
  U_kP_\gamma F=\chi_\gamma(k)P_\gamma F,
\]
so $P_\gamma F\in\mathcal H_\gamma$.  If $F\in\mathcal H_\eta$, then
character orthogonality gives
\[
  P_\gamma F
  =\frac1M\sum_{h\in\mathsf A}
    \overline{\chi_\gamma(h)}\chi_\eta(h)F
  =\mathbf 1_{\{\gamma=\eta\}}F.
\]
If $\gamma\neq\eta$, the spaces $\mathcal H_\gamma$ and
$\mathcal H_\eta$ are orthogonal because the $U_h$ are unitary.  Moreover,
$U_h^*=U_{-h}$, and changing $h$ to $-h$ in
\eqref{eq:wm-isotypic-projection} gives $P_\gamma^*=P_\gamma$.  Finally,
character orthogonality gives, for every $F\in L^2(\nu)$,
\[
  \sum_{\gamma\in\widehat{\mathsf A}}P_\gamma F
  =\frac1M\sum_{h\in\mathsf A}
     \left(\sum_\gamma\overline{\chi_\gamma(h)}\right)U_hF
  =F.
\]
Thus $P_\gamma$ is the orthogonal projection onto $\mathcal H_\gamma$ and
\eqref{eq:wm-fourier-decomposition} holds.  Since $R_h$ changes the value of
each coordinate but not its index, $P_\gamma$ does not enlarge the set of
coordinates on which $F$ depends.  The commutation with the shift follows
from $R_hT=TR_h$.
\end{proof}

The relation defining $\mathcal H_\gamma$ concerns the finite-group action
$R_h$, not the shift $T$; in particular, $P_\gamma F$ need not be a shift
eigenfunction.

By \eqref{eq:wm-exact-character-word-sum}, powers of $S_\gamma$ are
the exact growth factors for $\gamma$-character sums on level-$m$
substitution words.  The two nonzero nontrivial scales are therefore
\begin{equation}\label{eq:wm-scale-reminder}
  |S_\beta|^m=M^m=Q^{m/4},
  \qquad
  |S_\alpha|^m=M^{3m}=Q^{3m/4}.
\end{equation}
Thus $\alpha$ is the only nontrivial character with growth larger than
$Q^{m/4}$.  The same bound holds for finite-coordinate observables with
vanishing $\alpha$-component.

\begin{proposition}[Uniform discrepancy away from $\alpha$]
\label{prop:wm-symbolic-discrepancy}
Let $F:X_\vartheta\to\mathbb C$ be a finite-coordinate function such that
\begin{equation}\label{eq:wm-discrepancy-hypotheses}
  \int F\,d\nu=0,
  \qquad
  P_\alpha F=0.
\end{equation}
Then there is $C_F<\infty$ such that
\begin{equation}\label{eq:wm-arbitrary-interval-discrepancy}
  \left|
    \sum_{n=r}^{r+N-1}F(T^nx)
  \right|
  \leq C_FN^{1/4}
\end{equation}
for every $x\in X_\vartheta$, $r\in\mathbb Z$, and $N\geq1$.
\end{proposition}

\begin{proof}
\smallskip\noindent\emph{Character decomposition.}
Write
\begin{equation}\label{eq:wm-character-decomposition}
  F=\sum_{\gamma\in\widehat{\mathsf A}}F_\gamma,
  \qquad
  F_\gamma:=P_\gamma F.
\end{equation}
By hypothesis $F_\alpha=0$.  If $\gamma\neq0$, choose $h\in\mathsf A$ with
$\chi_\gamma(h)\neq1$.  The $R_h$-invariance of $\nu$ gives
\[
  \int F_\gamma\,d\nu
  =\int F_\gamma\circ R_h\,d\nu
  =\chi_\gamma(h)\int F_\gamma\,d\nu,
\]
so every nontrivial component has mean zero.  Hence
$\int F_0\,d\nu=\int F\,d\nu=0$.  Since the sum in
\eqref{eq:wm-character-decomposition} is finite, fix a common dependency
radius $\ell$ for all $F_\gamma$.

\smallskip\noindent\emph{Sums over level-$m$ substitution words.}
Fix $m$ with $Q^m>2\ell$.  If
$x_{[0,Q^m)}=\vartheta^m(a)$, define
\begin{equation}\label{eq:wm-interior-sum-definition}
  A_m^\gamma(a)
  :=\sum_{n=\ell}^{Q^m-1-\ell}F_\gamma(T^nx).
\end{equation}
Every dependency window occurring in this sum is contained in
$[0,Q^m)$, so $A_m^\gamma(a)$ depends only on the word
$\vartheta^m(a)$ and not on its extension to $X_\vartheta$.  Since
$\vartheta^m(a+h)=R_h\vartheta^m(a)$,
\begin{equation}\label{eq:wm-isotypic-word-sum}
  A_m^\gamma(a+h)=\chi_\gamma(h)A_m^\gamma(a).
\end{equation}

The word $\vartheta^{m+1}(0)$ is the concatenation of the $Q$ words
$\vartheta^m(c_j)$.  Replacing its interior sum by the sum of the $Q$
interior block sums can change only those terms whose dependency window
meets one of the $Q-1$ block boundaries.  There are at most $2\ell$ such
terms per boundary.  Using \eqref{eq:wm-isotypic-word-sum},
\begin{equation}\label{eq:wm-word-sum-recurrence}
  A_{m+1}^\gamma(0)
  =\sum_{j=0}^{Q-1}A_m^\gamma(c_j)+E_m^\gamma
  =S_\gamma A_m^\gamma(0)+E_m^\gamma,
  \qquad
  |E_m^\gamma|\leq2\ell(Q-1)\|F_\gamma\|_\infty.
\end{equation}
For $\gamma=\beta$, $S_\beta=M$.  If $m_*$ is the least integer with
$Q^{m_*}>2\ell$, then \eqref{eq:wm-word-sum-recurrence} gives, for
$m\geq m_*$,
\[
  |A_m^\beta(0)|
  \leq M^{m-m_*}|A_{m_*}^\beta(0)|
     +C_F\sum_{r=m_*}^{m-1}M^{m-1-r}
  =O_F(M^m).
\]
Equation \eqref{eq:wm-isotypic-word-sum} gives the same bound for every
initial letter $a$, and hence
\begin{equation}\label{eq:wm-beta-complete-word-bound}
  |A_m^\beta(a)|=O_F(M^m)=O_F(Q^{m/4}).
\end{equation}
For every nontrivial $\gamma\notin\{\alpha,\beta\}$ one has
$S_\gamma=0$, and therefore
\begin{equation}\label{eq:wm-other-character-complete-word-bound}
  |A_m^\gamma(a)|=O_F(1).
\end{equation}
The only nontrivial component with larger block growth is $\alpha$, which
is absent by assumption.

\smallskip\noindent\emph{The trivial character.}
For $\gamma=0$, put $A_m=A_m^0(0)$ and $E_m=E_m^0$.  Then
\eqref{eq:wm-word-sum-recurrence} reads
$A_{m+1}=QA_m+E_m$, with $|E_m|\leq C_F$.  Hence, for $n>m$,
\begin{equation}\label{eq:wm-trivial-backward-series}
  \frac{A_m}{Q^m}
  =\frac{A_n}{Q^n}
   -\sum_{r=m}^{n-1}\frac{E_r}{Q^{r+1}}.
\end{equation}
Choose an extension $x^{(n)}\in X_\vartheta$ of the word
$\vartheta^n(0)$.  By \eqref{eq:wm-interior-sum-definition},
\[
  \left|
    A_n-\sum_{r=0}^{Q^n-1}F_0(T^rx^{(n)})
  \right|
  \leq2\ell\|F_0\|_\infty.
\]
Since $F_0$ has mean zero, unique ergodicity gives
\[
  \sup_{x\in X_\vartheta}
  \left|\frac1{Q^n}\sum_{r=0}^{Q^n-1}F_0(T^rx)\right|\longrightarrow0.
\]
Thus $A_n/Q^n\to0$.  Letting $n\to\infty$ in
\eqref{eq:wm-trivial-backward-series} yields
\[
  |A_m|
  \leq C_F\sum_{r=m}^\infty Q^{m-r-1}
  \leq \frac{C_F}{Q-1}.
\]
Since $F_0$ is $R_h$-invariant,
\eqref{eq:wm-isotypic-word-sum} gives the same bound for every initial
letter $a$ whenever $Q^m>2\ell$.

\smallskip\noindent\emph{Decomposition into complete blocks.}
Let $J=[r,r+N)\cap\mathbb Z$ and
$m_0:=\lfloor\log_QN\rfloor$.  Set $I_{0,n}(x):=\{n\}$.  There are
collections $\mathcal B_m=\mathcal B_m(J,x)$ of complete level-$m$ blocks,
$0\leq m\leq m_0$, such that
\begin{equation}\label{eq:wm-block-decomposition}
  J=\bigsqcup_{m=0}^{m_0}\ \bigsqcup_{B\in\mathcal B_m}B,
  \qquad
  |\mathcal B_{m_0}|\leq Q,
  \qquad
  |\mathcal B_m|\leq2(Q-1)\quad(0\leq m<m_0).
\end{equation}
Indeed, take first all complete level-$m_0$ blocks contained in $J$.  The
complement has at most two components, each contained in one level-$m_0$
block.  Inside each component, take all complete level-$(m_0-1)$ subblocks;
there are at most $Q-1$ on each side.  Repeating this operation on the at
most two remaining endpoint components gives
\eqref{eq:wm-block-decomposition}; at level zero the blocks are singletons.

For $B\in\mathcal B_m$, write $a(B)$ for the initial letter of its
level-$m$ substitution word.  If $Q^m>2\ell$, then
\begin{equation}\label{eq:wm-complete-block-vs-interior}
  \left|
    \sum_{n\in B}F_\gamma(T^nx)-A_m^\gamma(a(B))
  \right|
  \leq2\ell\|F_\gamma\|_\infty.
\end{equation}
If $Q^m\leq2\ell$, then
\[
  \left|\sum_{n\in B}F_\gamma(T^nx)\right|
  \leq2\ell\|F_\gamma\|_\infty.
\]
Therefore \eqref{eq:wm-beta-complete-word-bound} and
\eqref{eq:wm-block-decomposition} give
\begin{equation}\label{eq:wm-geometric-discrepancy-sum}
  \left|\sum_{n\in J}F_\beta(T^nx)\right|
  \leq C_F\sum_{m=0}^{m_0}(1+M^m)
  =O_F(M^{m_0})
  =O_F(N^{1/4}).
\end{equation}
For $\gamma=0$ and every nontrivial
$\gamma\notin\{\alpha,\beta\}$, the same decomposition and the complete
block bounds give
\[
  \left|\sum_{n\in J}F_\gamma(T^nx)\right|
  =O_F(m_0+1)=O_F(\log N)=O_F(N^{1/4}).
\]
Since $F_\alpha=0$, summing the finitely many components in
\eqref{eq:wm-character-decomposition} proves
\eqref{eq:wm-arbitrary-interval-discrepancy}.
\end{proof}

\subsection{From the symbolic system to a point process}
\label{subsec:wm-suspension}

For $x\in X_\vartheta$ and $n\in\mathbb Z$, write the $n$-th coordinate
uniquely as
\begin{equation}\label{eq:wm-coordinate-decomposition}
  x_n=(g_n(x),s_n(x))\in\mathbb F_2^q\times\mathbb F_2.
\end{equation}
The $s$-coordinate controls the roof through $\beta$, while the $q$
coordinates of $g$ control occupancy of the $q$ cluster positions.

\subsubsection{The roof and the special flow}

Choose $L>\varepsilon>0$ such that
\begin{equation}\label{eq:wm-roof-parameters}
  L-\varepsilon>2,
  \qquad
  \frac{L}{\varepsilon}\notin\mathbb Q,
\end{equation}
and define the \emph{roof function}
\begin{equation}\label{eq:wm-roof}
  r(x):=L+\varepsilon\chi_\beta(x_0).
\end{equation}
Since $\nu$ is $R_h$-invariant, $\int r\,d\nu=L$.  The corresponding
\emph{special flow} is the translation flow on
\begin{equation}\label{eq:wm-special-flow-space}
  X_\vartheta^r
  :=\{(x,u):x\in X_\vartheta,\ 0\leq u<r(x)\}/
    ((x,r(x))\sim(Tx,0)),
\end{equation}
with invariant probability measure
\begin{equation}\label{eq:wm-suspension-measure}
  d\nu^r(x,u):=\frac1L\,d\nu(x)\,du.
\end{equation}
The roof fluctuation is $\varepsilon\chi_\beta(x_0)$.  Since the function $x\mapsto\chi_\beta(x_0)$ belongs to
$\mathcal H_\beta$, its $\alpha$-projection is zero, and
Proposition~\ref{prop:wm-symbolic-discrepancy} gives
an $O(N^{1/4})$ bound for the discrepancy of its Birkhoff sums.  The
irrationality of $L/\varepsilon$ enters only in the weak-mixing argument.

\subsubsection{Encoding a symbol by a local cluster}

We require the full $q$-position cluster to have a shape not realized by a
$q$-tuple with a repeated position.  Choose
\begin{equation}\label{eq:wm-offsets}
  0<\delta_1<\cdots<\delta_q<1
\end{equation}
so that
\begin{equation}\label{eq:wm-position-shape-separation}
  [(\delta_{i_1},\ldots,\delta_{i_q})^\circ]
  =[(\delta_1,\ldots,\delta_q)^\circ]
  \quad\Longrightarrow\quad
  (i_1,\ldots,i_q)\text{ is a permutation of }(1,\ldots,q).
\end{equation}
Such a choice exists.  Indeed, fix a tuple
$(i_1,\ldots,i_q)$ which is not a permutation of $(1,\ldots,q)$ and fix
$\pi\in\mathfrak S_q$.  Equality of the two centered ordered tuples after
applying $\pi$ would require
\[
  \delta_{i_{\pi(\ell)}}-\frac1q\sum_{m=1}^q\delta_{i_m}
  =\delta_\ell-\frac1q\sum_{m=1}^q\delta_m,
  \qquad 1\leq\ell\leq q.
\]
These equations do not hold identically in
$(\delta_1,\ldots,\delta_q)$: subtracting two of them would otherwise give
\[
  \delta_{i_{\pi(\ell)}}-\delta_{i_{\pi(\ell')}}
  =\delta_\ell-\delta_{\ell'}
\]
identically for every $\ell,\ell'$, which forces
$i_{\pi(\ell)}=\ell$ for all $\ell$.  Thus each pair consisting of a
non-permutation tuple and a permutation $\pi$ imposes a proper linear
subspace.  There are finitely many such subspaces, so
\eqref{eq:wm-offsets} can be chosen outside their union.

For $x\in X_\vartheta$, define the \emph{return-time cocycle} by
$\tau_0(x)=0$ and
\begin{equation}\label{eq:wm-return-time-cocycle}
  \tau_n(x):=\sum_{j=0}^{n-1}r(T^jx),
  \qquad n\geq1,
\end{equation}
with
$\tau_{-n}(x):=-\sum_{j=-n}^{-1}r(T^jx)$ for $n\geq1$.  Then
\begin{equation}\label{eq:wm-return-time-shift}
  \tau_n(Tx)=\tau_{n+1}(x)-r(x),
  \qquad n\in\mathbb Z.
\end{equation}
For $g=(g_1,\ldots,g_q)\in\mathbb F_2^q$, put
\begin{equation}\label{eq:wm-occupancy-indicator}
  a_i(g):=\frac{1+(-1)^{g_i}}2=\mathbf 1_{\{g_i=0\}}.
\end{equation}
At symbolic index $n$, we place a point at $\tau_n(x)+\delta_i$ precisely
when $a_i(g_n(x))=1$:
\begin{equation}\label{eq:wm-point-configuration}
  p_x
  :=\sum_{n\in\mathbb Z}\sum_{i=1}^q
     a_i(g_n(x))\,\delta_{\tau_n(x)+\delta_i}.
\end{equation}
Since $\tau_{n+1}(x)-\tau_n(x)\geq L-\varepsilon>2$, this sum is locally
finite.  Moreover $g_n(Tx)=g_{n+1}(x)$, and
\eqref{eq:wm-return-time-shift} gives
\begin{equation}\label{eq:wm-point-configuration-shift}
  p_{Tx}
  =\sum_{n,i}a_i(g_{n+1}(x))
     \delta_{\tau_{n+1}(x)-r(x)+\delta_i}
  =(-r(x)).p_x.
\end{equation}

For every $f\in C_c(\mathbb R)$, the quantity $p_x(f)$ is a finite sum of
Borel functions of $x$; hence $x\mapsto p_x$ is Borel for the vague
topology.  Define
\begin{equation}\label{eq:wm-suspension-factor-map}
  \pi(x,u):=(-u).p_x,
  \qquad
  \mu_q:=\pi_*\nu^r.
\end{equation}
Equation \eqref{eq:wm-point-configuration-shift} gives
\[
  \pi(x,r(x))=(-r(x)).p_x=p_{Tx}=\pi(Tx,0),
\]
so $\pi$ is well defined on the suspension quotient.  Translation in the
suspension coordinate, together with the same identity whenever a roof is
crossed, gives
\begin{equation}\label{eq:wm-factor-anti-equivariance}
  \pi(\Phi_t z)=(-t).\pi(z),
  \qquad t\in\mathbb R.
\end{equation}
Thus $\pi$ is an equivariant factor map from the reversed special flow
$t\mapsto\Phi_{-t}$ to the translation action.

\subsubsection{Return-time estimates}

For $R>0$, set $I_R=[0,R]$.  Since $[-R,R]$ is a translate of
$I_{2R}$, invariance gives
\begin{equation}\label{eq:weak-mixing-interval-ball-bridge}
  \Var_{\mu_q}\bigl(T_j[\chi_{[-R,R]},\varphi]\bigr)
  =\Var_{\mu_q}\bigl(T_j[\chi_{I_{2R}},\varphi]\bigr).
\end{equation}
Thus it suffices to estimate the statistics with $I_R$.  Define
\begin{equation}\label{eq:wm-return-index-count}
  N_R(x,u)
  :=\#\{n\in\mathbb Z:\tau_n(x)-u\in I_R\}.
\end{equation}

\begin{proposition}[Return-time estimates]
\label{prop:wm-return-time-estimates}
The measure $\mu_q$ is an $\mathbb R$-invariant ergodic simple point process
and is locally $L^p$-integrable for every finite $p$.  Moreover,
\begin{equation}\label{eq:wm-return-time-discrepancy}
  \tau_{a+N}(x)-\tau_a(x)=LN+O(N^{1/4})
\end{equation}
uniformly in $x\in X_\vartheta$, $a\in\mathbb Z$, and $N\geq1$, and
\begin{equation}\label{eq:wm-return-index-count-asymptotic}
  N_R(x,u)=\frac RL+O(R^{1/4})
\end{equation}
uniformly in $(x,u)\in X_\vartheta^r$ and $R\geq1$.
\end{proposition}

\begin{proof}
Since $T$ is ergodic and $r>0$, the special flow $(X_\vartheta^r,\nu^r,\Phi_t)$
is ergodic.  By \eqref{eq:wm-factor-anti-equivariance}, its factor law $\mu_q$
is therefore invariant and ergodic under translations.

If $n<m$, then for any
$i,j\in\{1,\ldots,q\}$,
\begin{align*}
 (\tau_m(x)+\delta_j)-(\tau_n(x)+\delta_i)
 &\geq (m-n)(L-\varepsilon)-(\delta_q-\delta_1)\\
 &>1.
\end{align*}
For a fixed $n$, the points $\tau_n(x)+\delta_i$ are distinct.  Hence every
interval of length one contains at most $q$ points of $p_x$, and the same
holds for every translate $(-u).p_x$.  Consequently, for every bounded interval $J$ and every
$(x,u)\in X_\vartheta^r$,
\[
  \pi(x,u)(J)\leq q(\lceil |J|\rceil+1).
\]
Thus $\mu_q=\pi_*\nu^r$ is a simple point process and is locally
$L^p$-integrable for every finite $p$.

Let
\[
  A(x):=\sum_{i=1}^q a_i(g_0(x))
\]
be the number of points attached to the return at symbolic index $0$.
For the $i$-th standard basis vector $e_i\in\mathbb F_2^q$, invariance of
$\nu$ under $R_{(e_i,0)}$ and
$a_i(g+e_i)=1-a_i(g)$ give
\[
  \int a_i(g_0(x))\,d\nu(x)
  =\int \bigl(1-a_i(g_0(x))\bigr)\,d\nu(x)
  =\frac12.
\]
Since $0<\delta_i<1<L-\varepsilon\leq r(x)$, these are precisely the points
attached inside one suspension roof segment.  The suspension intensity
formula therefore gives
\begin{equation}\label{eq:wm-intensity}
  \rho_{\mu_q}
  =\frac{1}{\int r\,d\nu}\int A\,d\nu
  =\frac1L\sum_{i=1}^q\frac12
  =\frac{q}{2L}>0.
\end{equation}

For $N\geq1$, the cocycle definition and \eqref{eq:wm-roof} give the exact
identity
\begin{equation}\label{eq:wm-return-time-beta-sum}
  \tau_{a+N}(x)-\tau_a(x)-LN
  =\varepsilon\sum_{n=a}^{a+N-1}\chi_\beta(x_n).
\end{equation}
The function $x\mapsto\chi_\beta(x_0)$ belongs to $\mathcal H_\beta$, hence
has zero $\alpha$-projection.  Proposition~\ref{prop:wm-symbolic-discrepancy}
applied to this function yields
\[
  \left|\tau_{a+N}(x)-\tau_a(x)-LN\right|
  \leq C N^{1/4},
\]
uniformly in $x,a,N$.  This proves
\eqref{eq:wm-return-time-discrepancy}.

It remains to estimate $N_R(x,u)$.  Since
\[
  L-\varepsilon
  \leq \tau_{n+1}(x)-\tau_n(x)
  \leq L+\varepsilon,
\]
the sequence $n\mapsto\tau_n(x)-u$ is strictly increasing.  Hence
\[
  J_R(x,u):=\{n\in\mathbb Z:0\leq\tau_n(x)-u\leq R\}
\]
is an interval of consecutive integers and
$N_R(x,u)=|J_R(x,u)|$.  Moreover,
\[
  N_R(x,u)
  \leq 1+\frac{R}{L-\varepsilon}.
\]
Thus \eqref{eq:wm-return-index-count-asymptotic} is immediate for
$1\leq R\leq2(L+\varepsilon)$ after increasing the implicit constant.

Assume now that $R>2(L+\varepsilon)$.  Since
$0\leq u<r(x)\leq L+\varepsilon<R$,
\[
  0<\tau_1(x)-u=r(x)-u<R,
\]
so $1\in J_R(x,u)$ and this set is nonempty.  Write
\[
  J_R(x,u)=\{a,a+1,\ldots,a+N-1\},
  \qquad N=N_R(x,u).
\]
The defining inequalities for $a$ and $a+N-1$, together with maximality of
this interval of indices, are
\[
  \tau_{a-1}(x)-u<0\leq\tau_a(x)-u,
\]
and
\[
  \tau_{a+N-1}(x)-u\leq R<\tau_{a+N}(x)-u.
\]
Using the roof bound once at each endpoint gives
\[
  0\leq\tau_a(x)-u<L+\varepsilon,
  \qquad
  0<\tau_{a+N}(x)-(u+R)\leq L+\varepsilon.
\]
Subtracting the two quantities yields
\begin{equation}\label{eq:wm-return-time-endpoint-error}
  \left|R-\bigl(\tau_{a+N}(x)-\tau_a(x)\bigr)\right|
  \leq L+\varepsilon.
\end{equation}
Since
\[
  N(L-\varepsilon)
  \leq \tau_{a+N}(x)-\tau_a(x)
  \leq N(L+\varepsilon),
\]
\eqref{eq:wm-return-time-endpoint-error} and $R>2(L+\varepsilon)$ imply
\[
  \frac{R}{2(L+\varepsilon)}
  \leq N
  \leq \frac{3R}{2(L-\varepsilon)}.
\]
In particular $N\asymp R$ uniformly.  Finally,
\eqref{eq:wm-return-time-discrepancy} and
\eqref{eq:wm-return-time-endpoint-error} give
\[
  R
  =\tau_{a+N}(x)-\tau_a(x)+O(1)
  =LN+O(N^{1/4})+O(1)
  =LN+O(R^{1/4}),
\]
which is \eqref{eq:wm-return-index-count-asymptotic}.
\end{proof}

\subsection{Lower-order hyperuniformity}
\label{subsec:wm-lower-order}

Fix $1\leq j<q$ and $\varphi\in C_c(\Sh_j(\mathbb R))$.  We express
$T_j[\chi_{I_R},\varphi]$ as a finite sum of Birkhoff sums of
finite-coordinate functions on $X_\vartheta$.  Under coordinatewise addition
$g_n\mapsto g_n+u$, $u\in\mathbb F_2^q$, each such function involves at most
$j$ coordinate labels.  Lemma~\ref{lem:wm-character-exclusion} below proves
that its $\alpha$-Fourier component vanishes, after which
Proposition~\ref{prop:wm-symbolic-discrepancy} applies.

\subsubsection{Finite pattern types}

Every point of $p_x$ has a unique representation
\begin{equation}\label{eq:wm-point-index-label}
  \tau_n(x)+\delta_i,
  \qquad n\in\mathbb Z,\quad 1\leq i\leq q,
\end{equation}
with $a_i(g_n(x))=1$.  Uniqueness follows from
$\tau_{n+1}(x)-\tau_n(x)\geq L-\varepsilon>2$ and
$0<\delta_i<1$.

Consider an ordered $j$-tuple of points of $p_x$, with repetitions allowed,
and write its entries uniquely as
\[
  \tau_{n_\ell}(x)+\delta_{i_\ell},
  \qquad 1\leq\ell\leq j.
\]
Take the first symbolic index $n_1$ as reference and set
\begin{equation}\label{eq:wm-relative-indices-definition}
  n:=n_1,
  \qquad
  r_\ell:=n_\ell-n_1,
  \qquad
  \mathbf r=(0,r_2,\ldots,r_j),
  \qquad
  \mathbf i=(i_1,\ldots,i_j).
\end{equation}
Conversely, a choice of $n$, $\mathbf r$ with $r_1=0$, and $\mathbf i$
with
\[
  a_{i_\ell}(g_{n+r_\ell}(x))=1,
  \qquad 1\leq\ell\leq j,
\]
determines exactly one ordered $j$-tuple, again allowing repeated entries.
Thus the parametrization is bijective.
The cocycle identity gives
\begin{equation}\label{eq:wm-relative-point-coordinates}
  \tau_{n_\ell}(x)+\delta_{i_\ell}
  =\tau_n(x)+z_\ell^{\mathbf r,\mathbf i}(T^nx),
  \qquad
  z_\ell^{\mathbf r,\mathbf i}(y)
  :=\tau_{r_\ell}(y)+\delta_{i_\ell}.
\end{equation}

The diameter
\begin{equation}\label{eq:wm-shape-diameter}
  \operatorname{diam}([v_1,\ldots,v_j])
  :=\max_\ell v_\ell-\min_\ell v_\ell
\end{equation}
is well defined on $\Sh_j(\mathbb R)$.  Choose $D_\varphi<\infty$ with
$\operatorname{diam}(w)\leq D_\varphi$ on $\operatorname{supp}\varphi$,
and choose $A_\varphi\in\mathbb N$ so that
\begin{equation}\label{eq:wm-relative-index-bound}
  (L-\varepsilon)A_\varphi-1>D_\varphi.
\end{equation}
Since $r_1=0$ and $\tau_0=0$,
\[
  |z_\ell^{\mathbf r,\mathbf i}(y)-z_1^{\mathbf r,\mathbf i}(y)|
  \geq (L-\varepsilon)|r_\ell|-|\delta_{i_\ell}-\delta_{i_1}|
  >(L-\varepsilon)|r_\ell|-1.
\]
Thus, if $|r_\ell|\geq A_\varphi$, the first and $\ell$-th entries in
\eqref{eq:wm-relative-point-coordinates} are more than $D_\varphi$ apart.
Hence only
\begin{equation}\label{eq:wm-relative-index-set}
  \mathscr R_\varphi
  :=\{\mathbf r=(r_1,\ldots,r_j)\in\mathbb Z^j:
       r_1=0,\ |r_\ell|<A_\varphi\}
\end{equation}
can contribute.

For $\mathbf r\in\mathscr R_\varphi$ and
$\mathbf i\in\{1,\ldots,q\}^j$, define
\begin{equation}\label{eq:wm-finite-coordinate-pattern-function}
  F_{\mathbf r,\mathbf i,\varphi}(y)
  :=\varphi\bigl([(z_1^{\mathbf r,\mathbf i}(y),\ldots,
                    z_j^{\mathbf r,\mathbf i}(y))^\circ]\bigr)
    \prod_{\ell=1}^j a_{i_\ell}(g_{r_\ell}(y))
\end{equation}
and
\begin{equation}\label{eq:wm-relative-barycenter}
  b_{\mathbf r,\mathbf i}(y)
  :=\frac1j\sum_{\ell=1}^j z_\ell^{\mathbf r,\mathbf i}(y).
\end{equation}
Both are finite-coordinate functions.  Indeed,
\[
  \tau_r(y)=\sum_{n=0}^{r-1}r(T^ny)\quad(r>0),
  \qquad
  \tau_r(y)=-\sum_{n=r}^{-1}r(T^ny)\quad(r<0),
\]
and each roof value $r(T^ny)$ depends only on the symbol $y_n$.
Thus all $\tau_{r_\ell}(y)$ and all occupancy factors in
\eqref{eq:wm-finite-coordinate-pattern-function} depend on finitely many
coordinates.

Equations \eqref{eq:wm-point-index-label}--\eqref{eq:wm-relative-barycenter}
give the exact decomposition
\begin{equation}\label{eq:wm-exact-pattern-decomposition}
  \begin{aligned}
  T_j[\chi_{I_R},\varphi]((-u).p_x)
  ={}&\sum_{\mathbf r\in\mathscr R_\varphi}
      \sum_{\mathbf i\in\{1,\ldots,q\}^j}
      \sum_{n\in\mathbb Z}
      \chi_{I_R}\bigl(\tau_n(x)-u
          +b_{\mathbf r,\mathbf i}(T^nx)\bigr) \\
     &\hspace{38mm}\cdot
      F_{\mathbf r,\mathbf i,\varphi}(T^nx).
  \end{aligned}
\end{equation}
Since $\mathscr R_\varphi$ is finite and the roof is bounded, there is
$B_\varphi<\infty$ such that
$|b_{\mathbf r,\mathbf i}(y)|\leq B_\varphi$.  For fixed
$(\mathbf r,\mathbf i)$, the two indicators
\[
  \chi_{I_R}\bigl(\tau_n(x)-u+b_{\mathbf r,\mathbf i}(T^nx)\bigr),
  \qquad
  \chi_{I_R}(\tau_n(x)-u)
\]
can differ only when $\tau_n(x)-u$ lies within $B_\varphi$ of an endpoint of
$I_R$.  Since consecutive return times differ by at least
$L-\varepsilon$, the number of such $n$ is at most
\begin{equation}\label{eq:wm-boundary-index-number}
  2\left(1+\left\lceil\frac{2B_\varphi}{L-\varepsilon}\right\rceil\right).
\end{equation}

\subsubsection{Exclusion of the $\alpha$-component}

\begin{lemma}[Character exclusion]
\label{lem:wm-character-exclusion}
For $j<q$,
\begin{equation}\label{eq:wm-alpha-component-vanishes}
  P_\alpha F_{\mathbf r,\mathbf i,\varphi}=0
\end{equation}
for every $\mathbf r\in\mathscr R_\varphi$ and
$\mathbf i\in\{1,\ldots,q\}^j$.
\end{lemma}

\begin{proof}
Fix $x\in X_\vartheta$ and write $h=(u,t)\in\mathbb F_2^q\times\mathbb F_2$.
Adding $h$ changes $g_{r_\ell}(x)$ to $g_{r_\ell}(x)+u$ and changes only the
last coordinate in the roof sums defining $\tau_{r_\ell}$.  Hence the shape
factor in \eqref{eq:wm-finite-coordinate-pattern-function} is independent of
$u$.  For the occupancy factor,
\begin{equation}\label{eq:wm-occupancy-character-expansion}
\begin{aligned}
  \prod_{\ell=1}^j a_{i_\ell}(g_{r_\ell}(x)+u)
  &=2^{-j}\prod_{\ell=1}^j
    \left(1+(-1)^{(g_{r_\ell}(x))_{i_\ell}+u_{i_\ell}}\right) \\
  &=2^{-j}\sum_{J\subset\{1,\ldots,j\}}
    \left(\prod_{\ell\in J}(-1)^{(g_{r_\ell}(x))_{i_\ell}}\right)
    (-1)^{\sum_{\ell\in J}u_{i_\ell}}.
\end{aligned}
\end{equation}
For each $J$, the character of $u$ in the last line has support contained in
$\{i_1,\ldots,i_j\}$.  Since this set has cardinality at most $j<q$, none
of these characters is $u\mapsto(-1)^{u_1+\cdots+u_q}$, the restriction of
$\chi_\alpha$ to $\mathbb F_2^q\times\{0\}$.  Orthogonality of characters on
$\mathbb F_2^q$ therefore gives \eqref{eq:wm-alpha-component-vanishes}.
\end{proof}

Thus every observable arising from a $j$-point pattern with $j<q$ has zero
$\alpha$-component.  For the full $q$-point cluster this fails; see
\eqref{eq:wm-exceptional-alpha-component}.

\subsubsection{The variance estimate}

\begin{proposition}[Lower-order hyperuniformity]
\label{prop:wm-lower-order-hu}
For every $1\leq j<q$ and every
$\varphi\in C_c(\Sh_j(\mathbb R))$ there is $C_\varphi<\infty$ such that
\begin{equation}\label{eq:wm-lower-order-variance}
  \Var_{\mu_q}\bigl(T_j[\chi_{I_R},\varphi]\bigr)
  \leq C_\varphi R^{1/2},
  \qquad R\geq1.
\end{equation}
In particular, $\mu_q\in\HU_{\leq q-1}$.
\end{proposition}

\begin{proof}
Let
\begin{equation}\label{eq:wm-return-index-set}
  J_R(x,u):=\{n\in\mathbb Z:\tau_n(x)-u\in I_R\}.
\end{equation}
It is an integer interval of cardinality $N_R(x,u)$.  For a fixed
$(\mathbf r,\mathbf i)$, the indicators in
\eqref{eq:wm-exact-pattern-decomposition} and
$\chi_{I_R}(\tau_n(x)-u)$ differ for at most the number of indices in
\eqref{eq:wm-boundary-index-number}.  Since
$F_{\mathbf r,\mathbf i,\varphi}$ is bounded, and since the sets of
$\mathbf r$ and $\mathbf i$ are finite, \eqref{eq:wm-exact-pattern-decomposition}
gives the uniform reduction
\begin{equation}\label{eq:wm-pattern-boundary-reduction}
  T_j[\chi_{I_R},\varphi]((-u).p_x)
  =\sum_{\mathbf r\in\mathscr R_\varphi}
    \sum_{\mathbf i\in\{1,\ldots,q\}^j}
    \sum_{n\in J_R(x,u)}
      F_{\mathbf r,\mathbf i,\varphi}(T^nx)
    +O_\varphi(1).
\end{equation}

Fix $(\mathbf r,\mathbf i)$ and write
$F=F_{\mathbf r,\mathbf i,\varphi}$ and
$F^\circ=F-\int F\,d\nu$.  Since $\alpha\neq0$, constants have zero
$\alpha$-projection, and Lemma~\ref{lem:wm-character-exclusion} gives
$P_\alpha F^\circ=0$.  Proposition~\ref{prop:wm-symbolic-discrepancy}
applied to the integer interval $J_R(x,u)$ therefore yields
\[
  \left|\sum_{n\in J_R(x,u)}F^\circ(T^nx)\right|
  \leq C_F N_R(x,u)^{1/4}.
\]
By \eqref{eq:wm-return-index-count-asymptotic}, uniformly in $(x,u)$,
\begin{align}
  \sum_{n\in J_R(x,u)}F(T^nx)
  &=N_R(x,u)\int F\,d\nu+O_F(R^{1/4}) \\
  &=\frac RL\int F\,d\nu+O_F(R^{1/4}).
  \label{eq:wm-pattern-sum-asymptotic}
\end{align}
Substituting \eqref{eq:wm-pattern-sum-asymptotic} into
\eqref{eq:wm-pattern-boundary-reduction} gives
\begin{equation}\label{eq:wm-uniform-pattern-discrepancy}
  T_j[\chi_{I_R},\varphi]((-u).p_x)
  =c_\varphi R+O_\varphi(R^{1/4}),
  \qquad
  c_\varphi:=\frac1L
  \sum_{\mathbf r\in\mathscr R_\varphi}
  \sum_{\mathbf i\in\{1,\ldots,q\}^j}
  \int F_{\mathbf r,\mathbf i,\varphi}\,d\nu,
\end{equation}
uniformly in $(x,u)\in X_\vartheta^r$.  Finally,
$\Var(X)\leq\mathbb E|X-c|^2$ for every constant $c$; applying this with
$c=c_\varphi R$ and using \eqref{eq:wm-uniform-pattern-discrepancy} proves
\eqref{eq:wm-lower-order-variance}.
\end{proof}

\subsection{Failure at the next order}
\label{subsec:wm-failure}

\subsubsection{The full $q$-point cluster}
Set
\begin{equation}\label{eq:wm-offset-barycenter}
  \overline\delta:=\frac1q\sum_{i=1}^q\delta_i,
\end{equation}
and let $w_\delta\in\Sh_q(\mathbb R)$ be the centered unordered shape of
$(\delta_1,\ldots,\delta_q)$.  The set of shapes obtained from ordered
$q$-tuples of the positions $\delta_1,\ldots,\delta_q$, with repetitions
allowed, is finite.  By \eqref{eq:wm-position-shape-separation}, none of
these shapes equals $w_\delta$ unless every position is used exactly once.
If two points come from distinct symbolic indices $n<m$, then
\[
  (\tau_m(x)+\delta_{i'})-(\tau_n(x)+\delta_i)
  \geq (L-\varepsilon)-(\delta_q-\delta_1)>1.
\]
Thus every $q$-tuple using more than one symbolic index has diameter greater
than one, while $\operatorname{diam}(w_\delta)<1$.  Let $\mathcal W_q$ be
the finite set of centered shapes obtained from ordered $q$-tuples of
$\delta_1,\ldots,\delta_q$ that do not use every position exactly once.
By \eqref{eq:wm-position-shape-separation}, $w_\delta\notin\mathcal W_q$.
Since $\{w:\operatorname{diam}(w)<1\}$ is an open neighbourhood of
$w_\delta$, there is an open set $U_q$ such that
\[
  w_\delta\in U_q\subset\{w:\operatorname{diam}(w)<1\},
  \qquad
  U_q\cap\mathcal W_q=\varnothing.
\]
Choose $\varphi_q\in C_c(\Sh_q(\mathbb R))$ nonnegative, supported in
$U_q$, and satisfying $\varphi_q(w_\delta)=1$.

Define the finite-coordinate function
\begin{equation}\label{eq:wm-exceptional-function}
  F_q(x)
  :=\prod_{i=1}^q a_i(g_0(x))
  =2^{-q}\sum_{J\subset\{1,\ldots,q\}}
       \chi_{(\sum_{i\in J}e_i,0)}(x_0).
\end{equation}
The summand with $J=\{1,\ldots,q\}$ is $\chi_\alpha(x_0)$, and the
finite-group Fourier projection from
Lemma~\ref{lem:wm-fourier-projections} therefore satisfies
\begin{equation}\label{eq:wm-exceptional-alpha-component}
  P_\alpha F_q=2^{-q}\chi_\alpha(x_0).
\end{equation}
To obtain a lower bound for the variance, we project the $q$-point statistic
onto this $\alpha$-component.  The required symmetry on the suspension is
provided by the subgroup
$H:=\mathbb F_2^q\times\{0\}\leq\mathsf A$.  Since
$\chi_\beta(h)=1$ for $h\in H$, the roof satisfies $r(R_hx)=r(x)$.
Consequently
\[
  (x,u)\longmapsto(R_hx,u),\qquad h\in H,
\]
defines a measure-preserving action on $(X_\vartheta^r,\nu^r)$.  The full
group $\mathsf A$ does not act on the suspension because the roof is not
$R_{(0,1)}$-invariant.  Character orthogonality on
$H\cong\mathbb F_2^q$ shows that
\begin{equation}\label{eq:wm-suspension-alpha-projection}
  \Pi_\alpha G(x,u)
  :=2^{-q}\sum_{h\in\mathbb F_2^q}
       \chi_\alpha(h,0)G(R_{(h,0)}x,u)
\end{equation}
is the orthogonal projection onto the subspace satisfying
$G(R_{(h,0)}x,u)=\chi_\alpha(h,0)G(x,u)$ for every $h\in\mathbb F_2^q$.
For a function independent of the last $\mathbb F_2$-coordinate, summing
first over that coordinate in the full-group formula of
Lemma~\ref{lem:wm-fourier-projections} gives exactly
\eqref{eq:wm-suspension-alpha-projection}.

Define
\begin{equation}\label{eq:wm-selected-index-set}
  \mathcal I_R(x,u)
  :=\{n\in\mathbb Z:\tau_n(x)+\overline\delta-u\in I_R\}.
\end{equation}
Because $T_q$ sums over ordered $q$-tuples with repetitions allowed, the
choice of $\varphi_q$ gives the exact identity
\begin{equation}\label{eq:wm-exact-full-cluster-statistic}
  T_q[\chi_{I_R},\varphi_q]((-u).p_x)
  =q!\sum_{n\in\mathcal I_R(x,u)}F_q(T^nx).
\end{equation}
Indeed, the support of $\varphi_q$ excludes tuples using more than one
symbolic index and excludes every repeated-position tuple at a single
index; the remaining tuple uses all $q$ positions once, in one of its $q!$
orderings.

For $h\in\mathbb F_2^q$, the roof and hence every $\tau_n$ is unchanged by
$R_{(h,0)}$.  Therefore
\begin{equation}\label{eq:wm-selected-index-set-invariance}
  \mathcal I_R(R_{(h,0)}x,u)=\mathcal I_R(x,u).
\end{equation}
By \eqref{eq:wm-selected-index-set-invariance}, the index set in
\eqref{eq:wm-exact-full-cluster-statistic} is unchanged under the
$H$-action.  Since $R_h$ commutes with $T$, applying $\Pi_\alpha$ termwise
and using \eqref{eq:wm-exceptional-alpha-component} gives
\begin{equation}\label{eq:wm-alpha-projection-pattern-statistic}
  \Pi_\alpha T_q[\chi_{I_R},\varphi_q]((-u).p_x)
  =q!\,2^{-q}
    \sum_{n\in\mathcal I_R(x,u)}\chi_\alpha(x_n).
\end{equation}

\subsubsection{Two consecutive copies of a level-$m$ word}

For $a\in\mathsf A$, $m\geq1$, and $0\leq j<Q^m$, define
\begin{equation}\label{eq:wm-address-set-definition}
  C_m(a,j)
  :=\{x\in X_\vartheta:
       x=T^j\vartheta^m(y)\text{ in its unique level-$m$ representation, }
       y_0=a\}.
\end{equation}
The $MQ^m$ sets $C_m(a,j)$ partition $X_\vartheta$.  Set
$E_{m,j}:=\bigcup_aC_m(a,j)$.  Uniqueness of the level-$m$ representation
gives
\[
  T E_{m,j}=E_{m,j+1}\quad(0\leq j<Q^m-1),
  \qquad
  T E_{m,Q^m-1}=E_{m,0},
\]
up to null sets.  Since $\nu$ is $T$-invariant and the $E_{m,j}$ partition
$X_\vartheta$,
\[
  \nu(E_{m,j})=Q^{-m}.
\]
For fixed $j$, translation equivariance of the substitution gives
$R_hC_m(a,j)=C_m(a+h,j)$.  The $R_h$ preserve $\nu$, so the $M$ sets
$C_m(a,j)$, $a\in\mathsf A$, have equal measure inside $E_{m,j}$.  Hence
\begin{equation}\label{eq:wm-address-set-measure}
  \nu(C_m(a,j))=\frac1{MQ^m}.
\end{equation}

The sets below will be used both for the variance lower bound and for weak
mixing.

Because $c_0=c_1=0$, the first two letters of $\vartheta(z)$ are both $z$.
Applying $\vartheta^m$ to this identity shows that, for every
$z\in\mathsf A$,
\begin{equation}\label{eq:wm-repeated-word-prefix}
  \vartheta^{m+1}(z)
  =\vartheta^m(z)\,\vartheta^m(z)\,W_{m,z}
\end{equation}
for a word $W_{m,z}$ of length $(Q-2)Q^m$.

Choose $z_+,z_-\in\mathsf A$ with
$\chi_\beta(z_\sigma)=\sigma$ for $\sigma\in\{+1,-1\}$, and fix
$0<\eta<1/4$.  For $m\geq1$, let
\begin{equation}\label{eq:wm-repeated-word-set}
  D_m^\sigma
  :=\bigcup_{\substack{j\in\mathbb Z\\
          \eta Q^m\leq j\leq(1-\eta)Q^m}}
     C_{m+1}(z_\sigma,j).
\end{equation}
Thus, if $x\in D_m^\sigma$, then in the unique representation
$x=T^j\vartheta^{m+1}(y)$ one has $y_0=z_\sigma$ and the coordinate $x_0$
corresponds to a position $j$ lying inside the first copy of
$\vartheta^m(z_\sigma)$ in \eqref{eq:wm-repeated-word-prefix}, at distance
at least $\eta Q^m$ from both ends of that copy.

\begin{lemma}[Two-copy estimates]
\label{lem:wm-repeated-word-estimates}
For each $\sigma\in\{+1,-1\}$,
\begin{equation}\label{eq:wm-repeated-word-measure}
  \lim_{m\to\infty}\nu(D_m^\sigma)
  =\frac{1-2\eta}{MQ}>0.
\end{equation}
If $x\in D_m^\sigma$, then
\begin{equation}\label{eq:wm-coordinate-agreement}
  x_n=x_{n+Q^m}
  \qquad\text{for every }|n|<\eta Q^m,
\end{equation}
and
\begin{align}
  \sum_{n=0}^{Q^m-1}\chi_\alpha(x_n)
  &=M^{3m}\chi_\alpha(z_\sigma),
  \label{eq:wm-repeated-alpha-sum}\\
  \sum_{n=0}^{Q^m-1}\chi_\beta(x_n)
  &=\sigma M^m.
  \label{eq:wm-repeated-beta-sum}
\end{align}
Consequently
\begin{equation}\label{eq:wm-repeated-return-time}
  \tau_{Q^m}(x)=LQ^m+\sigma\varepsilon M^m.
\end{equation}
\end{lemma}

\begin{proof}
Let
\[
  J_m:=\#\{j\in\mathbb Z:
       \eta Q^m\leq j\leq(1-\eta)Q^m\}.
\]
The sets $C_{m+1}(z_\sigma,j)$ in
\eqref{eq:wm-repeated-word-set} are disjoint, and
\eqref{eq:wm-address-set-measure} gives
\begin{equation}\label{eq:wm-repeated-word-measure-computation}
  \nu(D_m^\sigma)=\frac{J_m}{MQ^{m+1}}.
\end{equation}
Since $J_m/Q^m\to1-2\eta$, this proves
\eqref{eq:wm-repeated-word-measure}.

Take $x\in C_{m+1}(z_\sigma,j)$ with $j$ satisfying the inequalities in
\eqref{eq:wm-repeated-word-set}.  In the word
\eqref{eq:wm-repeated-word-prefix}, positions $j+n$ and
$j+n+Q^m$ lie in the first and second copies of
$\vartheta^m(z_\sigma)$, respectively, whenever
$|n|<\eta Q^m$.  The two letters are equal, which gives
\eqref{eq:wm-coordinate-agreement}.

Write
\[
  \vartheta^m(z_\sigma)=UV,
  \qquad |U|=j,
  \qquad |V|=Q^m-j.
\]
The first two copies in \eqref{eq:wm-repeated-word-prefix} are then
$UVUV$.  Since $x_0$ is at position $j$, we have the exact word identity
\[
  x_{[0,Q^m)}=VU.
\]
Consequently, for every function $b:\mathsf A\to\mathbb C$,
\[
  \sum_{n=0}^{Q^m-1}b(x_n)
  =\sum_{w\text{ a letter of }\vartheta^m(z_\sigma)}b(w),
\]
with multiplicities.  Taking $b=\chi_\alpha$ and $b=\chi_\beta$ and using
\eqref{eq:wm-exact-character-word-sum} gives
\eqref{eq:wm-repeated-alpha-sum} and
\eqref{eq:wm-repeated-beta-sum}.  Finally,
\[
  \tau_{Q^m}(x)
  =\sum_{n=0}^{Q^m-1}\bigl(L+\varepsilon\chi_\beta(x_n)\bigr),
\]
which gives \eqref{eq:wm-repeated-return-time}.
\end{proof}

\subsubsection{The variance lower bound}

Set $R_m:=LQ^m$ and define
\begin{equation}\label{eq:wm-positive-measure-suspension-set}
  E_m^\sigma
  :=\{(x,u):x\in D_m^\sigma,\ 0<u<\overline\delta/2\}.
\end{equation}
Since $\overline\delta/2<1<L-\varepsilon\leq r(x)$,
\eqref{eq:wm-suspension-measure} gives
\[
  \nu^r(E_m^\sigma)
  =\frac{\overline\delta}{2L}\,\nu(D_m^\sigma).
\]
By \eqref{eq:wm-repeated-word-measure}, these measures are bounded below by
a positive constant for all large $m$.

Fix $(x,u)\in E_m^\sigma$.  For $n=0$,
$0<\overline\delta-u<\overline\delta<R_m$, whereas
\[
  \tau_{-1}(x)+\overline\delta-u
  \leq -(L-\varepsilon)+\overline\delta<0.
\]
Since $\tau_n$ is strictly increasing, the smallest integer in
$\mathcal I_{R_m}(x,u)$ is $0$.  Thus
$\mathcal I_{R_m}(x,u)=\{0,\ldots,N_m-1\}$ for some $N_m\geq1$.
By \eqref{eq:wm-repeated-return-time},
\[
  \tau_{Q^m}(x)-R_m=\sigma\varepsilon M^m.
\]
The endpoint defining $\mathcal I_{R_m}(x,u)$ is
$R_m+u-\overline\delta$, whose distance from $R_m$ is less than one.  Hence
\[
  \left|\tau_{Q^m}(x)-\bigl(R_m+u-\overline\delta\bigr)\right|
  \leq \varepsilon M^m+1.
\]
Put
\[
  t_m:=R_m+u-\overline\delta.
\]
Since $\mathcal I_{R_m}(x,u)=\{0,\ldots,N_m-1\}$,
\begin{equation}\label{eq:wm-special-crossing-index}
  \tau_{N_m-1}(x)\leq t_m<\tau_{N_m}(x).
\end{equation}
If $N_m>Q^m$, then
\[
  (N_m-1-Q^m)(L-\varepsilon)
  \leq \tau_{N_m-1}(x)-\tau_{Q^m}(x)
  \leq t_m-\tau_{Q^m}(x).
\]
If $N_m\leq Q^m$, then
\[
  (Q^m-N_m)(L-\varepsilon)
  \leq \tau_{Q^m}(x)-\tau_{N_m}(x)
  <\tau_{Q^m}(x)-t_m.
\]
Together with
$|\tau_{Q^m}(x)-t_m|\leq1+\varepsilon M^m$, these two inequalities give
\begin{equation}\label{eq:wm-index-count-error-at-special-scale}
  |N_m-Q^m|
  \leq 1+\frac{1+\varepsilon M^m}{L-\varepsilon}.
\end{equation}
Both index sets start at zero, so their symmetric difference has exactly
$|N_m-Q^m|$ elements.  Therefore, by
\eqref{eq:wm-repeated-alpha-sum},
\[
\begin{aligned}
  \left|
    \sum_{n\in\mathcal I_{R_m}(x,u)}\chi_\alpha(x_n)
  \right|
  &\geq M^{3m}-|N_m-Q^m| \\
  &\geq M^{3m}-1-\frac{1+\varepsilon M^m}{L-\varepsilon}.
\end{aligned}
\]
Since $M\geq2$, the last expression is at least
$\frac12M^{3m}$ for all sufficiently large $m$.  Thus
\begin{equation}\label{eq:wm-alpha-lower-bound}
  \left|
    \sum_{n\in\mathcal I_{R_m}(x,u)}\chi_\alpha(x_n)
  \right|
  \geq\frac12M^{3m}.
\end{equation}
By the construction preceding
\eqref{eq:wm-suspension-alpha-projection}, $\Pi_\alpha$ is an
orthogonal projection in $L^2(\nu^r)$.  Since $\alpha$ is nontrivial,
$\Pi_\alpha 1=0$.  Hence, for every $G\in L^2(\nu^r)$,
\begin{equation}\label{eq:wm-projection-variance-inequality}
  \|G-\nu^r(G)\|_2^2
  \geq\|\Pi_\alpha(G-\nu^r(G))\|_2^2
  =\|\Pi_\alpha G\|_2^2.
\end{equation}
Apply this to
$G=T_q[\chi_{I_{R_m}},\varphi_q]$.  From
\eqref{eq:wm-alpha-projection-pattern-statistic},
\eqref{eq:wm-alpha-lower-bound}, and the lower bound on
$\nu^r(E_m^\sigma)$, there is $c>0$, independent of $m$, such that
\[
\begin{aligned}
  \|\Pi_\alpha G\|_2^2
  &\geq (q!\,2^{-q})^2
     \int_{E_m^\sigma}
       \left|\sum_{n\in\mathcal I_{R_m}(x,u)}
                    \chi_\alpha(x_n)\right|^2\,d\nu^r(x,u) \\
  &\geq cM^{6m}.
\end{aligned}
\]
Together with \eqref{eq:wm-projection-variance-inequality}, this gives
\begin{equation}\label{eq:wm-exceptional-variance-lower-bound}
  \Var_{\mu_q}\bigl(T_q[\chi_{I_{R_m}},\varphi_q]\bigr)
  \geq cM^{6m}
  =cQ^{3m/2}
  \asymp R_m^{3/2}.
\end{equation}
By \eqref{eq:weak-mixing-interval-ball-bridge}, the same lower bound holds
for the centered interval $[-R_m/2,R_m/2]=B_{R_m/2}$.  Since
$\Vol_1(B_{R_m/2})=R_m$, the normalized variance is bounded below by a
positive multiple of $R_m^{1/2}$ along this sequence.  Hence
$\mu_q\notin\HU_q$.

\subsection{Weak mixing and completion of the proof}
\label{subsec:wm-weak-mixing}

\begin{lemma}\label{lem:wm-base-n}
Let $n\geq2$ be an integer.  If $\theta\in\mathbb R$ satisfies
\begin{equation}\label{eq:wm-base-n-hypothesis}
  \lim_{m\to\infty}
  \|n^m\theta\|_{\mathbb R/\mathbb Z}=0,
\end{equation}
then $\theta\in\mathbb Z[1/n]$.
\end{lemma}

\begin{proof}
Let $a_m\in\mathbb Z$ be a nearest integer to $n^m\theta$.  Then
\begin{equation}\label{eq:wm-base-n-integer-difference}
  |a_{m+1}-na_m|
  \leq
  \|n^{m+1}\theta\|_{\mathbb R/\mathbb Z}
  +n\|n^m\theta\|_{\mathbb R/\mathbb Z}.
\end{equation}
The right-hand side tends to zero and the left-hand side is an integer.
Thus $a_{m+1}=na_m$ for all sufficiently large $m$, so $a_m/n^m$ is
eventually constant.  Since
\[
  \left|\frac{a_m}{n^m}-\theta\right|
  \leq\frac{1}{2n^m},
\]
this constant is $\theta$, and therefore $\theta\in\mathbb Z[1/n]$.
\end{proof}

\begin{proposition}[Weak mixing]
\label{prop:wm-weak-mixing}
The special flow $(X_\vartheta^r,\nu^r)$ is weakly mixing.  Consequently
$\mu_q$ is weakly mixing.
\end{proposition}

\begin{proof}
\smallskip\noindent\emph{Phase constraints from the sets $D_m^\sigma$.}
Suppose that $\xi\in\mathbb R$ is an eigenvalue, and let $F$ be a
corresponding eigenfunction on the suspension.  Ergodicity makes $|F|$
constant almost everywhere, so normalize $|F|=1$.  For almost every $x$ and
almost every $u,v\in[0,r(x))$, the eigenrelation along the fibre gives
\[
  F(x,v)=e^{2\pi i\xi(v-u)}F(x,u).
\]
Hence there is a measurable $\omega:X_\vartheta\to\mathbb T$ such that
\[
  F(x,u)=e^{2\pi i\xi u}\omega(x)
\]
for almost every $(x,u)$.  The identification
$(x,r(x))\sim(Tx,0)$ then yields
\begin{equation}\label{eq:wm-eigenvalue-cohomology}
  \omega(Tx)=e^{2\pi i\xi r(x)}\omega(x)
\end{equation}
for $\nu$-a.e. $x$.  Iterating, with
$r_N(x):=\sum_{j=0}^{N-1}r(T^jx)$, gives
\begin{equation}\label{eq:wm-eigenvalue-iterate}
  \omega(T^Nx)=e^{2\pi i\xi r_N(x)}\omega(x).
\end{equation}

Fix $\sigma\in\{+1,-1\}$.  By
Lemma~\ref{lem:wm-repeated-word-estimates}, the measures of
$D_m^\sigma$ are bounded below for large $m$, and
\begin{equation}\label{eq:wm-eigenvalue-return-time}
  r_{Q^m}(x)=LQ^m+\sigma\varepsilon M^m
  \qquad(x\in D_m^\sigma).
\end{equation}
Moreover, let $d$ be any metric inducing the product topology on
$X_\vartheta$.  From \eqref{eq:wm-coordinate-agreement},
\[
  \sup_{x\in D_m^\sigma}d(T^{Q^m}x,x)\longrightarrow0.
\]
Since $X_\vartheta$ is compact, every continuous
$\psi:X_\vartheta\to\mathbb C$ is uniformly continuous; therefore
\begin{equation}\label{eq:wm-continuous-approximation}
  \sup_{x\in D_m^\sigma}
  |\psi(T^{Q^m}x)-\psi(x)|\longrightarrow0.
\end{equation}

Let $c_\sigma>0$ be a lower bound for $\nu(D_m^\sigma)$ for all large
$m$, and let $X_0\subset X_\vartheta$ be a conull set on which
\eqref{eq:wm-eigenvalue-iterate} holds for every $N\geq1$.  Fix
$\kappa>0$ and choose $\zeta>0$ so that
$2\zeta^2/\kappa^2<c_\sigma/2$.  Then choose a continuous
$\psi$ with $\|\omega-\psi\|_2<\zeta$.  By Chebyshev's inequality and
$T$-invariance, the subset of $D_m^\sigma$ on which either
\begin{equation}\label{eq:wm-two-approximation-errors}
  |\omega(x)-\psi(x)|\geq\kappa
  \qquad\text{or}\qquad
  |\omega(T^{Q^m}x)-\psi(T^{Q^m}x)|\geq\kappa
\end{equation}
has measure at most $2\zeta^2/\kappa^2<c_\sigma/2$.  Thus for every large
$m$ there is $x\in D_m^\sigma\cap X_0$ for which neither inequality in
\eqref{eq:wm-two-approximation-errors} holds.  For such an $x$, \eqref{eq:wm-eigenvalue-iterate} and
\eqref{eq:wm-eigenvalue-return-time} give
\begin{align*}
  \left|e^{2\pi i\xi(LQ^m+\sigma\varepsilon M^m)}-1\right|
  &=|\omega(T^{Q^m}x)-\omega(x)| \\
  &\leq |\omega(T^{Q^m}x)-\psi(T^{Q^m}x)|
       +|\psi(T^{Q^m}x)-\psi(x)|
       +|\psi(x)-\omega(x)|.
\end{align*}
The first and third terms are less than $\kappa$, while the middle term
tends uniformly to zero on $D_m^\sigma$ by
\eqref{eq:wm-continuous-approximation}.  Therefore
\begin{equation}\label{eq:wm-eigenvalue-phase-bound}
  \limsup_{m\to\infty}
  \left|e^{2\pi i\xi(LQ^m+\sigma\varepsilon M^m)}-1\right|
  \leq2\kappa.
\end{equation}
Because $\kappa>0$ is arbitrary, the left-hand side of
\eqref{eq:wm-eigenvalue-phase-bound} tends to zero.  Since
$|e^{2\pi it}-1|\to0$ if and only if
$\|t\|_{\mathbb R/\mathbb Z}\to0$, we obtain
\begin{equation}\label{eq:wm-eigenvalue-phase-limits}
  \lim_{m\to\infty}
  \|\xi(LQ^m+\sigma\varepsilon M^m)\|_{\mathbb R/\mathbb Z}=0
  \qquad(\sigma=\pm1).
\end{equation}
\smallskip\noindent\emph{Arithmetic conclusion.}
For $a,b\in\mathbb R$,
$\|a\pm b\|_{\mathbb R/\mathbb Z}
 \leq\|a\|_{\mathbb R/\mathbb Z}+\|b\|_{\mathbb R/\mathbb Z}$.
Applying this to the two expressions in
\eqref{eq:wm-eigenvalue-phase-limits} gives
\begin{align*}
  \|2\xi LQ^m\|_{\mathbb R/\mathbb Z}
  &\leq
  \|\xi(LQ^m+\varepsilon M^m)\|_{\mathbb R/\mathbb Z}
  +\|\xi(LQ^m-\varepsilon M^m)\|_{\mathbb R/\mathbb Z},\\
  \|2\xi\varepsilon M^m\|_{\mathbb R/\mathbb Z}
  &\leq
  \|\xi(LQ^m+\varepsilon M^m)\|_{\mathbb R/\mathbb Z}
  +\|\xi(LQ^m-\varepsilon M^m)\|_{\mathbb R/\mathbb Z}.
\end{align*}
Hence
\begin{equation}\label{eq:wm-eigenvalue-separated-limits}
  \lim_{m\to\infty}\|2\xi LQ^m\|_{\mathbb R/\mathbb Z}=0,
  \qquad
  \lim_{m\to\infty}\|2\xi\varepsilon M^m\|_{\mathbb R/\mathbb Z}=0.
\end{equation}
Lemma~\ref{lem:wm-base-n}, first with $n=Q$ and then with $n=M$, gives
\begin{equation}\label{eq:wm-eigenvalue-arithmetic}
  2\xi L\in\mathbb Z[1/Q],
  \qquad
  2\xi\varepsilon\in\mathbb Z[1/M].
\end{equation}
If $\xi\neq0$, the quotient of these two nonzero rational numbers is
$L/\varepsilon$, contrary to \eqref{eq:wm-roof-parameters}.  Hence $0$ is
the only eigenvalue.  Since $(X_\vartheta,\nu,T)$ is ergodic and the roof is
positive, the special flow is ergodic; an ergodic flow with no nonzero
eigenvalues is weakly mixing.  Its time reversal is therefore weakly mixing
as well, and $\pi$ is an equivariant factor map for that reversed flow.
Thus weak mixing passes to $\mu_q$.
\end{proof}

\begin{proof}[Proof of Theorem~\ref{thm:weak-mixing-strictness}]
For the process $\mu_q$, Proposition~\ref{prop:wm-lower-order-hu} gives
$\mu_q\in\HU_{\leq q-1}$,
\eqref{eq:wm-exceptional-variance-lower-bound} gives
$\mu_q\notin\HU_q$, and Proposition~\ref{prop:wm-weak-mixing} gives weak
mixing.  Taking $q=k+1$ proves the theorem.
\end{proof}

\section{Cut-and-project point processes}
\label{sec:cut-and-project}

\subsection{Cut-and-project processes and pattern reduction}
\label{subsec:cap-pattern-windows}

Let $H$ be a locally compact second countable abelian group.  A
\emph{cut-and-project scheme} $(V,H,\Gamma)$ consists of a discrete
cocompact subgroup $\Gamma<V\times H$ such that the projection
$\pi_V|_\Gamma$ is injective and $\pi_H(\Gamma)$ is dense in $H$.
We call $V$ the \emph{physical space} and $H$ the \emph{internal space}.
For $\gamma\in\Gamma$ write
\begin{equation}\label{eq:cap-star-notation}
  \gamma=(\gamma_V,\gamma_H),
  \qquad
  \gamma_V:=\pi_V(\gamma),
  \qquad
  \gamma_H:=\pi_H(\gamma).
\end{equation}
The coordinates $\gamma_V$ and $\gamma_H$ are called the physical and
internal components of $\gamma$.  Let $m_H$ be a Haar measure on $H$.  A
compact Borel set $W'\subset H$ will be called a \emph{window}; it is
\emph{regular} if
\begin{equation}\label{eq:cap-regular-window}
  W'=\overline{\operatorname{int}W'},
  \qquad
  \operatorname{int}W'\neq\varnothing,
  \qquad
  m_H(\partial W')=0.
\end{equation}

Put $\Omega:=\Gamma\backslash(V\times H)$ and let $m_\Omega$ be Haar
probability on $\Omega$.  For $\omega=\Gamma+(v,h)\in\Omega$ and a compact
window $W'$, define
\begin{equation}\label{eq:cap-process-definition}
  (p_{W'})_\omega
  :=\sum_{\gamma\in\Gamma}
    \chi_{W'}(\gamma_H+h)\,\delta_{\gamma_V+v}.
\end{equation}
Changing $(v,h)$ by an element of $\Gamma$ only reindexes the sum.  If
$K\subset V$ is compact, only the lattice points in the compact set
$(K-v)\times(W'-h)$ can contribute to $(p_{W'})_\omega(K)$, so the measure is
locally finite.  Injectivity of $\pi_V|_\Gamma$ makes it simple.  For every
$f\in C_c(V)$, the function $\omega\mapsto(p_{W'})_\omega(f)$ is Borel: in
local coordinates it is a finite sum of Borel functions, and the sum is
$\Gamma$-invariant.  Since the Borel structure of the vague topology on
$\cM(V)$ is generated by evaluation against functions in $C_c(V)$, the map
$\omega\mapsto(p_{W'})_\omega$ is Borel.

The group $V$ acts on $\Omega$ by
$t\cdot(\Gamma+(v,h)):=\Gamma+(v+t,h)$.  With the translation action on
$\cM(V)$ from Section~\ref{sec:general-theory},
\begin{equation}\label{eq:cap-equivariance}
  (p_{W'})_{t\cdot\omega}=t.(p_{W'})_\omega.
\end{equation}
Hence
\begin{equation}\label{eq:cap-haar-pushforward}
  \mu_{W'}:=(\omega\mapsto(p_{W'})_\omega)_*m_\Omega
\end{equation}
is a $V$-invariant simple point process.  When $W'$ is regular, we call it
the regular cut-and-project point process associated with $W'$.

\smallskip\noindent\emph{Uniform local moment bound.}
Since $W'$ is compact, put
\[
  D_{W'}:=\{\gamma_V:\gamma\in\Gamma,\ \gamma_H\in W'-W'\}.
\]
This set contains $0$ and has a positive separation constant.  Indeed, if
there were distinct $x_n,y_n\in D_{W'}$ with $\|x_n-y_n\|\to0$, choose
$\delta_n,\delta_n'\in\Gamma$ representing them.  The nonzero lattice points
$\delta_n-\delta_n'$ have internal components in the compact set
$(W'-W')-(W'-W')$ and physical components tending to $0$.  They therefore
lie eventually in a fixed compact subset of $V\times H$, contradicting
discreteness of $\Gamma$; injectivity of $\pi_V|_\Gamma$ ensures that these
nonzero lattice points have nonzero physical components.  Since $0\in D_{W'}$,
every configuration $(p_{W'})_\omega$ has the same positive lower bound on
the distance between distinct points.  Consequently, for every compact
$K\subset V$ there is $N_{K,W'}<\infty$ such that
\[
  (p_{W'})_\omega(K)\leq N_{K,W'}
  \qquad(\omega\in\Omega).
\]
Thus $\mu_{W'}$ is locally $L^r$-integrable for every finite $r$, and the
moment hypotheses required below are automatic.

For $f\in C_c(V)$ set
\begin{equation}\label{eq:cap-window-statistic}
  S_{W'}f(\omega):=(p_{W'})_\omega(f).
\end{equation}
We use the same notation for bounded compactly supported Borel functions.
Whenever a pattern statistic is written as a function in $L^2(\Omega)$, it is
understood to be evaluated at $(p_W)_\omega$.

Let $\widehat H$ be the Pontryagin dual of $H$, written multiplicatively.  We
identify $\xi\in V^*$ with the character $v\mapsto e^{2\pi i\xi(v)}$ and
define the annihilator
\begin{equation}\label{eq:cap-dual-lattice-definition}
  \Gamma^\perp
  :=\bigl\{(\xi,\eta)\in V^*\times\widehat H:
      e^{2\pi i\xi(\gamma_V)}\eta(\gamma_H)=1
      \text{ for every }\gamma\in\Gamma\bigr\}.
\end{equation}

From now on in this section, $H$ is a finite-dimensional real Euclidean
space.  Write
\[
  d:=\dim V,\qquad m:=\dim H.
\]
We call $d$ the \emph{physical dimension} and $m$ the \emph{internal
dimension}; the codimension is $m$.  We take $m_H=\Vol_m$ and equip
$V\times H$ with the orthogonal product Euclidean structure, so its Haar
measure is $\Vol_{d+m}$.  We identify $H^*$ with $\widehat H$ through
$\eta\mapsto(h\mapsto e^{2\pi i\eta(h)})$.  Under this identification,
$\Gamma^\perp$ is a lattice in $V^*\times H^*$, called the \emph{dual
lattice}; for $(\xi,\eta)\in\Gamma^\perp$, $\xi$ and $\eta$ are the physical
and internal frequencies.  For $(\xi,\eta)\in V^*\times H^*$,
\begin{equation}\label{eq:cap-dual-membership-additive}
  (\xi,\eta)\in\Gamma^\perp
  \quad\Longleftrightarrow\quad
  \xi(\gamma_V)+\eta(\gamma_H)\in\mathbb Z
  \quad\text{for every }\gamma\in\Gamma.
\end{equation}

Let $\operatorname{covol}(\Gamma)$ be the $\Vol_{d+m}$-volume of a
fundamental domain.  For a regular window $W'$, write $\sigma_{W'}$ for the
centered first Bartlett spectrum of $\mu_{W'}$.  By
\cite[Theorem~2.9]{BH24},
\begin{equation}\label{eq:cap-diffraction-formula}
  \sigma_{W'}
  =\frac{1}{\operatorname{covol}(\Gamma)^2}
   \sum_{(\xi,\eta)\in\Gamma^\perp\setminus\{0\}}
     |\widehat{\chi}_{W'}(\eta)|^2\,\delta_\xi.
\end{equation}
The density of $\pi_H(\Gamma)$ implies
$\Gamma^\perp\cap(\{0\}\times H^*)=\{0\}$, so the removed atom is exactly
$(0,0)$.

Two elementary facts about windows will be used repeatedly.  First, internal
translation does not change the law:
\begin{equation}\label{eq:cap-window-translation-law}
  (p_{W'+h_0})_{\Gamma+(v,h)}
  =(p_{W'})_{\Gamma+(v,h-h_0)},
  \qquad
  \mu_{W'+h_0}=\mu_{W'}.
\end{equation}
Second, changing a compact window on an $m$-dimensional null set does not
change its law.  Indeed, the quotient integration formula for
$\Gamma\backslash(V\times H)$ shows that the expected number of points in a
fixed bounded physical set selected by the symmetric difference is zero.
Consequently, if a compact window $A$ has null boundary and nonempty
interior, then
\begin{equation}\label{eq:cap-window-regularization}
  A^{\mathrm{reg}}:=\overline{\operatorname{int}A}
  \quad\text{satisfies}\quad
  \mu_A=\mu_{A^{\mathrm{reg}}}.
\end{equation}
In particular, the Bartlett formula above applies to such a window after this
regularization.

Fix from now on one regular window \(W\subset H\).  If
\(\boldsymbol\gamma=(\gamma_1,\ldots,\gamma_j)\in\Gamma^j\) with
\(\gamma_1=0\), set
\begin{equation}\label{eq:cap-derived-window}
  W[\boldsymbol\gamma]
  :=\bigcap_{r=1}^j (W-\gamma_{r,H}).
\end{equation}
For \(\lambda\in\Gamma\), the points associated with
\(\lambda+\gamma_1,\ldots,\lambda+\gamma_j\) all occur in
\((p_W)_{\Gamma+(v,h)}\) precisely when
\begin{equation}\label{eq:cap-acceptance-condition}
  \lambda_H+h\in W[\boldsymbol\gamma].
\end{equation}
\begin{proposition}[Derived-window characterization]
\label{prop:cap-pattern-reduction}
Let $k\geq1$.
\begin{enumerate}[label=\textup{(\roman*)}]
\item For every $\varphi\in C_c(\Sh_k(V))$ there are finitely many windows
$W_1,\ldots,W_N\subset H$, each of the form
\begin{equation}\label{eq:cap-acceptance-window-form}
  W_\ell=W[\boldsymbol\gamma_\ell]
\end{equation}
for a tuple $\boldsymbol\gamma_\ell=(0,\gamma_{\ell,2},\ldots,\gamma_{\ell,j_\ell})$
of at most $k$ distinct lattice vectors, constants $c_1,\ldots,c_N\in\mathbb C$, and vectors
$b_1,\ldots,b_N\in V$ such that, for every bounded compactly supported Borel
function $f$ on $V$,
\begin{equation}\label{eq:cap-pattern-reduction}
  T_k[f,\varphi]
  =\sum_{\ell=1}^N c_\ell\,S_{W_\ell}(\tau_{b_\ell}f)
  \qquad\text{in }L^2(\Omega).
\end{equation}
Windows with empty interior may be omitted.
\item The following are equivalent:
\begin{enumerate}[label=\textup{(\alph*)}]
\item $\mu_W\in\HU_k$;
\item for every $1\leq j\leq k$ and every tuple
$\boldsymbol\gamma=(0,\gamma_2,\ldots,\gamma_j)$ of distinct lattice vectors
for which $W[\boldsymbol\gamma]$ has nonempty interior, the cut-and-project
process $\mu_{W[\boldsymbol\gamma]}$ belongs to $\HU_1$.
\end{enumerate}
\end{enumerate}
\end{proposition}

\begin{proof}
For \textup{(i)}, use the partition decomposition in
Equation~\eqref{eq:full-factorial-partition}.  Let $\pi\in\Pi_k$ have blocks
$B_1,\ldots,B_j$, ordered by their least elements.  An ordered $j$-tuple of
distinct points of $(p_W)_{\Gamma+(v,h)}$ is represented uniquely by distinct
lattice vectors $\lambda_1,\ldots,\lambda_j\in\Gamma$, because
$\pi_V|_\Gamma$ is injective.  Put
\begin{equation}\label{eq:cap-relative-lattice-vectors}
  \gamma_1:=0,
  \qquad
  \gamma_r:=\lambda_r-\lambda_1\quad(2\leq r\leq j).
\end{equation}
Conversely, $\lambda:=\lambda_1$ together with a tuple of distinct relative
vectors $\boldsymbol\gamma=(0,\gamma_2,\ldots,\gamma_j)$ determines the
ordered lattice tuple
\[
  (\lambda,\lambda+\gamma_2,\ldots,\lambda+\gamma_j).
\]
All these points are accepted precisely when
$\lambda_H+h\in W[\boldsymbol\gamma]$.  The barycentre of the repeated
$k$-tuple associated with $\pi$ is
\begin{equation}\label{eq:cap-weighted-partition-barycentre}
  \lambda_V+v+b_{\pi,\boldsymbol\gamma},
  \qquad
  b_{\pi,\boldsymbol\gamma}
  :=\frac1k\sum_{r=1}^j |B_r|\,\gamma_{r,V}.
\end{equation}
Let
\begin{equation}\label{eq:cap-partition-shape}
  w_{\pi,\boldsymbol\gamma}
  :=\left[
     \Delta_\pi(0,\gamma_{2,V},\ldots,\gamma_{j,V})^\circ
    \right]\in\Sh_k(V).
\end{equation}
For fixed $\pi$ and $\boldsymbol\gamma$, the corresponding contribution to
$T_k[f,\varphi]((p_W)_\omega)$ is therefore exactly
\begin{align}\label{eq:cap-one-derived-contribution}
 &\varphi(w_{\pi,\boldsymbol\gamma})
   \sum_{\lambda\in\Gamma}
     \chi_{W[\boldsymbol\gamma]}(\lambda_H+h)
     f(\lambda_V+v+b_{\pi,\boldsymbol\gamma}) \notag\\
 &\hspace{35mm}=
 \varphi(w_{\pi,\boldsymbol\gamma})
 S_{W[\boldsymbol\gamma]}
   (\tau_{b_{\pi,\boldsymbol\gamma}}f)(\omega).
\end{align}

Only finitely many $\boldsymbol\gamma$ contribute.  Choose $R_\varphi<\infty$
so that every representative of a shape in $\operatorname{supp}\varphi$ has
all pairwise differences of norm at most $R_\varphi$.  For every contributing
tuple and every $r,s$,
\[
  \|\gamma_{r,V}-\gamma_{s,V}\|\leq R_\varphi,
  \qquad
  \gamma_{r,H}-\gamma_{s,H}\in W-W.
\]
Thus every difference $\gamma_r-\gamma_s$ belongs to the finite set
\[
  \Gamma\cap\bigl(B_{R_\varphi}^V\times(W-W)\bigr).
\]
Since $\gamma_1=0$, only finitely many relative tuples can occur.  Moreover,
\begin{equation}\label{eq:cap-derived-window-boundary}
  \partial W[\boldsymbol\gamma]
  \subset \bigcup_{r=1}^j(\partial W-\gamma_{r,H}),
\end{equation}
so every derived window has null boundary.  If its interior is empty, then
$W[\boldsymbol\gamma]=\partial W[\boldsymbol\gamma]$ and hence
$\Vol_m(W[\boldsymbol\gamma])=0$.  For every compact $K\subset V$, the quotient
integration formula gives
\[
  \int_\Omega (p_{W[\boldsymbol\gamma]})_\omega(K)\,dm_\Omega(\omega)
  =\frac{\Vol_d(K)\Vol_m(W[\boldsymbol\gamma])}
         {\operatorname{covol}(\Gamma)}=0.
\]
A countable exhaustion of $V$ by compact sets therefore shows that
$(p_{W[\boldsymbol\gamma]})_\omega=0$ for almost every $\omega$.  If the
interior is nonempty, Equation~\eqref{eq:cap-window-regularization} replaces
the window, without changing its law or its $L^2(\Omega)$ statistics, by a
regular window.  Summing \eqref{eq:cap-one-derived-contribution} over the
finitely many partitions and relative tuples, and grouping identical terms,
gives \eqref{eq:cap-pattern-reduction}.  Before grouping, every coefficient is
$\varphi(w_{\pi,\boldsymbol\gamma})$.

For \textup{(ii)(b)} $\Rightarrow$ \textup{(ii)(a)}, fix $b\in V$.
The $V$-invariance of $\mu_{W_\ell}$ gives
\[
  \Var\bigl(S_{W_\ell}(\tau_b\chi_{B_R})\bigr)
  =\Var\bigl(S_{W_\ell}(\chi_{B_R})\bigr).
\]
Minkowski's inequality in $L^2$ therefore yields
\begin{equation}\label{eq:cap-finite-sum-variance}
  \Var\!\left(\sum_{\ell=1}^N c_\ell
      S_{W_\ell}(\tau_{b_\ell}\chi_{B_R})\right)
  \leq
  \left(\sum_{\ell=1}^N |c_\ell|
    \sqrt{\Var\bigl(S_{W_\ell}(\chi_{B_R})\bigr)}\right)^2
  =o(\Vol_d(B_R)).
\end{equation}
Part \textup{(i)} now gives $\mu_W\in\HU_k$.

For the converse, fix $1\leq j\leq k$ and
$\boldsymbol\gamma=(0,\gamma_2,\ldots,\gamma_j)$ as in
\textup{(ii)(b)}.  Choose positive integers $m_1,\ldots,m_j$ with
$m_1+\cdots+m_j=k$ and let $w_{\mathbf m,\boldsymbol\gamma}$ be the centred
unordered shape obtained by repeating
$0,\gamma_{2,V},\ldots,\gamma_{j,V}$ with these multiplicities.  The set
\begin{equation}\label{eq:cap-admissible-differences-uniform-discrete}
  D_W:=\{\delta_V:\delta\in\Gamma,\ \delta_H\in W-W\}
\end{equation}
has a positive separation constant by the argument above.  If $C\subset
\Sh_k(V)$ is compact, there is $R_C<\infty$ such that every shape in $C$ has a
representative whose pairwise differences have norm at most $R_C$.  Since
$D_W\cap B_{R_C}$ is finite, only finitely many admissible repeated shapes lie
in $C$.  The admissible repeated shapes are therefore locally finite in
$\Sh_k(V)$.  Hence there is $\varphi\in C_c(\Sh_k(V))$ which equals one at
$w_{\mathbf m,\boldsymbol\gamma}$ and vanishes on every other admissible
repeated shape.

Suppose two admissible lattice $k$-tuples have the same centred unordered
shape.  After a permutation, their physical coordinates differ by one common
translation $t\in V$.  Matching one occupied site gives a lattice difference
$\delta\in\Gamma$ with $\delta_V=t$.  Every other matched difference has the
same physical component, so injectivity of $\pi_V|_\Gamma$ forces all of them
to equal $\delta$.  Thus the two lattice realizations differ by a common
lattice translate, up to permutation.

The number of orderings of the fixed $j$ sites with multiplicities
$\mathbf m=(m_1,\ldots,m_j)$ is
\begin{equation}\label{eq:cap-multinomial-constant}
  c_{\mathbf m}:=\frac{k!}{m_1!\cdots m_j!}>0.
\end{equation}
Choosing a different site $\gamma_s$ as basepoint changes the derived window
and weighted barycentre by
\[
  W[\boldsymbol\gamma']
    =W[\boldsymbol\gamma]+\gamma_{s,H},
  \qquad
  b_{\mathbf m,\boldsymbol\gamma'}
    =b_{\mathbf m,\boldsymbol\gamma}-\gamma_{s,V},
\]
and the reindexing $\lambda'=\lambda-\gamma_s$ makes the resulting first
statistic identical.  Hence, as functions in $L^2(\Omega)$,
\begin{equation}\label{eq:cap-isolated-derived-window-statistic}
  T_k[f,\varphi]
  =c_{\mathbf m}\,
    S_{W[\boldsymbol\gamma]}
      (\tau_{b_{\mathbf m,\boldsymbol\gamma}}f),
  \qquad
  b_{\mathbf m,\boldsymbol\gamma}
  :=\frac1k\sum_{r=1}^j m_r\gamma_{r,V}.
\end{equation}
If $\mu_W\in\HU_k$, the defining variance condition applied to this fixed
$\varphi$, together with $V$-invariance of
$\mu_{W[\boldsymbol\gamma]}$, gives
$\mu_{W[\boldsymbol\gamma]}\in\HU_1$.  This proves
\textup{(ii)(a)} $\Rightarrow$ \textup{(ii)(b)}.
\end{proof}

\begin{corollary}[Monotonicity for cut-and-project processes]
\label{cor:cap-monotonicity}
For the process associated with a regular window in a Euclidean
cut-and-project scheme,
\begin{equation}\label{eq:cap-monotonicity}
  \mu_W\in\HU_k\quad\Longrightarrow\quad
  \mu_W\in\HU_j\qquad(1\leq j\leq k).
\end{equation}
In particular, $\mu_W\in\HU_k$ implies $\mu_W\in\HU_1$.
\end{corollary}

\begin{proof}
The derived windows involving at most $j$ distinct lattice sites form a
subset of those involving at most $k$ sites.  Apply
Proposition~\ref{prop:cap-pattern-reduction}(ii) at orders $k$ and $j$.
\end{proof}

\subsection{Codimension-one separation in physical dimension three}
\label{subsec:cap-codim-one-separation}

\begin{theorem}[Codimension-one separation]
\label{thm:cap-codim-one-separation}
There exists a Euclidean cut-and-project scheme
\[
  (\mathbb R^3,\mathbb R,\Gamma)
\]
and an interval window $W\subset\mathbb R$ such that
\begin{equation}\label{eq:cap-codim-one-separation}
  \mu_W\in\HU_1\setminus\HU_2.
\end{equation}
\end{theorem}

\begin{proof}
\smallskip\noindent\emph{The lattice and the window cancellation.}
Choose integers $M_n\geq4$ such that
\begin{equation}\label{eq:cap-codim-one-growth}
  M_n\mid M_{n+1},
  \qquad
  M_{n+1}/M_n\in M_n\mathbb Z,
  \qquad
  M_{n+1}\geq M_n^{12}.
\end{equation}
For example, one may take $M_1=4$ and $M_{n+1}=M_n^{12}$.  Set
\begin{equation}\label{eq:cap-codim-one-alpha}
  \alpha_j:=\sum_{n\geq1}M_n^{-j},
  \qquad j=1,2,3,
  \qquad
  \alpha:=(\alpha_1,\alpha_2,\alpha_3)^T.
\end{equation}
Then $0<\alpha_1<1$.  Each $\alpha_j$ is irrational.  Indeed, with
\begin{equation}\label{eq:cap-codim-one-prefixes-intro}
  \alpha_{j,n}:=\sum_{r\leq n}M_r^{-j}
  =\frac{P_{j,n}}{M_n^j},
\end{equation}
one has
\(
  0<\alpha_j-\alpha_{j,n}\leq2M_{n+1}^{-j}=o(M_n^{-j}).
\)
If $\alpha_j=a/b$ were rational, then the distinct rationals $a/b$ and
$P_{j,n}/M_n^j$ would differ by at least $(bM_n^j)^{-1}$, a contradiction for
large $n$.

Let $e_3=(0,0,1)^T$ and put
\begin{equation}\label{eq:cap-codim-one-G}
  G:=
  \begin{pmatrix}
    -I_3+e_3\alpha^T & e_3\\
    \alpha^T & 1
  \end{pmatrix}.
\end{equation}
A direct multiplication gives
\begin{equation}\label{eq:cap-codim-one-Gdual}
  G^{-T}
  =
  \begin{pmatrix}
    -I_3 & \alpha\\
    e_3^T & 1-\alpha_3
  \end{pmatrix}.
\end{equation}
The displayed matrix has determinant $-1$, hence $\det G=-1$ and
$\operatorname{covol}(\Gamma)=1$.  Define
\[
  \Gamma:=G\mathbb Z^4<\mathbb R^3\times\mathbb R,
  \qquad
  \Lambda:=\Gamma^\perp=G^{-T}\mathbb Z^4.
\]
For $m\in\mathbb Z^3$ and $r\in\mathbb Z$,
\[
  G\binom mr
  =\bigl(-m+e_3(\alpha\cdot m+r),\,\alpha\cdot m+r\bigr).
\]
If its physical component vanishes, then $m_1=m_2=0$ and
$r-(1-\alpha_3)m_3=0$.  Since $\alpha_3$ is irrational, $m_3=r=0$.
Thus the physical projection of $\Gamma$ is injective.  Its internal
projection contains $\mathbb Z+\alpha_3\mathbb Z$, hence is dense.  Therefore
$(\mathbb R^3,\mathbb R,\Gamma)$ is a Euclidean cut-and-project scheme.

Writing an element of $\Lambda$ as $(\xi,\eta)$ and the corresponding integer
vector as $(p_1,p_2,p_3,q)$, Equation~\eqref{eq:cap-codim-one-Gdual} gives
\begin{equation}\label{eq:cap-codim-one-dual-coordinates}
  \xi_j=q\alpha_j-p_j,
  \qquad
  \eta=p_3+(1-\alpha_3)q=q-\xi_3.
\end{equation}
The physical projection of $\Lambda$ is injective: if $\xi=0$, then
$q\alpha_1=p_1$, hence $q=0$ and then $p=0$.

Use the lattice vector
\begin{equation}\label{eq:cap-codim-one-primal-vectors}
  \gamma_1=(-e_1+\alpha_1e_3,\alpha_1).
\end{equation}
The remaining dual coordinate satisfies $\eta+\xi_3\in\mathbb Z$, which
produces the Fourier cancellation for the original interval.  Take
\begin{equation}\label{eq:cap-codim-one-window}
  W=[-1/2,1/2].
\end{equation}
Since $\eta+\xi_3=q\in\mathbb Z$, every dual point with $\eta\neq0$
satisfies
\begin{equation}\label{eq:cap-codim-one-window-cancellation}
  |\widehat{\chi}_W(\eta)|
  =\frac{|\sin(\pi\xi_3)|}{\pi|\eta|}
  \leq\frac{|\xi_3|}{|\eta|}.
\end{equation}
If $\eta=0$, irrationality of $\alpha_3$ forces $q=p_3=0$, and every
nonzero such dual point has $\xi\in\mathbb Z^3\setminus\{0\}$.  Hence these
points never occur in $B_\varepsilon^*$ once $\varepsilon<1$.

\smallskip\noindent\emph{Separation estimate for small physical frequencies.}
Put
\begin{equation}\label{eq:cap-codim-one-tails}
  t_{j,n}:=\alpha_j-\alpha_{j,n}.
\end{equation}
The growth assumptions imply
\begin{equation}\label{eq:cap-codim-one-tail-bounds}
  M_{n+1}^{-j}\leq t_{j,n}\leq2M_{n+1}^{-j}.
\end{equation}
Moreover $P_{j,n}$ is coprime to $M_n$.  Indeed, $P_{j,1}=1$, and if
$R_n:=M_{n+1}/M_n$, then
\begin{equation}\label{eq:cap-codim-one-prefix-recurrence}
  P_{j,n+1}=R_n^jP_{j,n}+1.
\end{equation}
Every prime divisor of $M_{n+1}=M_nR_n$ divides $R_n$, because
$M_n\mid R_n$.  Hence $P_{j,n+1}\equiv1$ modulo every prime divisor of
$M_{n+1}$.

Set
\begin{equation}\label{eq:cap-codim-one-qA}
  q_n:=M_n^3,
  \qquad
  A_n:=\frac{M_{n+1}}{8M_n}.
\end{equation}
Write $\operatorname{dist}_\infty(\cdot,\mathbb Z^3)$ for distance in the
sup norm.  If $0<|q|\leq A_n$ and $q$ is not divisible by $q_n=M_n^3$, choose
$r\in\{0,1,2\}$ such that
\[
  M_n^r\mid q,
  \qquad
  M_n^{r+1}\nmid q.
\]
Write $q=M_n^r u$.  Since
$\alpha_{r+1,n}=P_{r+1,n}/M_n^{r+1}$ and
$\gcd(P_{r+1,n},M_n)=1$,
\[
  q\alpha_{r+1,n}=\frac{uP_{r+1,n}}{M_n}.
\]
Because $M_n\nmid u$, this is not an integer.  After cancellation its
denominator is at most $M_n$, so its distance from $\mathbb Z$ is at least
$1/M_n$.  Therefore
\begin{equation}\label{eq:cap-codim-one-prefix-separation}
  \operatorname{dist}_\infty
  \bigl(q(\alpha_{1,n},\alpha_{2,n},\alpha_{3,n}),\mathbb Z^3\bigr)
  \geq\frac1{M_n}.
\end{equation}
For $|q|\leq A_n$, the tail bound gives
\begin{equation}\label{eq:cap-codim-one-tail-perturbation}
  \|q(t_{1,n},t_{2,n},t_{3,n})\|_\infty
  =|q|t_{1,n}
  \leq\frac1{4M_n}.
\end{equation}
For $q\in\mathbb Z$ with
$\operatorname{dist}_\infty(q\alpha,\mathbb Z^3)<1/2$, let $p_j(q)$ be the
unique nearest integer to $q\alpha_j$ and put
\begin{equation}\label{eq:cap-codim-one-return-coordinate-definition}
  \xi_j(q):=q\alpha_j-p_j(q),\qquad j=1,2,3.
\end{equation}
If the hypotheses on the left-hand side of
\eqref{eq:cap-codim-one-return-isolation} hold and $q\notin q_n\mathbb Z$,
then \eqref{eq:cap-codim-one-prefix-separation} and
\eqref{eq:cap-codim-one-tail-perturbation} give
\[
  \operatorname{dist}_\infty(q\alpha,\mathbb Z^3)
  \geq \frac1{M_n}-\frac1{4M_n}
  =\frac3{4M_n},
\]
a contradiction.  Hence
\begin{equation}\label{eq:cap-codim-one-return-isolation}
  0<|q|\leq A_n,
  \qquad
  \operatorname{dist}_\infty(q\alpha,\mathbb Z^3)<\frac1{2M_n}
  \quad\Longrightarrow\quad
  q\in q_n\mathbb Z.
\end{equation}
If $q=kq_n$ and $|q|\leq A_n$, then $q\alpha_{j,n}\in\mathbb Z$ for
$j=1,2,3$, while
\[
  |qt_{j,n}|\leq |q|t_{1,n}<\frac12.
\]
Thus $p_j(q)=q\alpha_{j,n}$ and
\begin{equation}\label{eq:cap-codim-one-special-return}
  \xi_j(q)=qt_{j,n}.
\end{equation}
Put
\begin{equation}\label{eq:cap-codim-one-rn}
  r_n:=\|q_n(t_{1,n},t_{2,n},t_{3,n})\|_\infty
  =q_nt_{1,n}
  \asymp\frac{M_n^3}{M_{n+1}}.
\end{equation}

\smallskip\noindent\emph{Hyperuniformity of the original interval.}
For
$0<\varepsilon<1/4$, define the sup-norm return set
\begin{equation}\label{eq:cap-codim-one-return-set}
  \mathcal R_\varepsilon
  :=\{q\in\mathbb Z\setminus\{0\}:
       \operatorname{dist}_\infty(q\alpha,\mathbb Z^3)\leq\varepsilon\}.
\end{equation}
Let $(\xi,\eta)\in\Lambda\setminus\{0\}$ satisfy
$\|\xi\|_2<\varepsilon$.  Then $q\neq0$: if $q=0$, the identities
$\xi_j=-p_j$ and $\varepsilon<1/4$ force $p=0$.  Hence
$q\in\mathcal R_\varepsilon$.  Moreover
\[
  |\eta|=|q-\xi_3|
  \geq |q|-|\xi_3|
  \geq |q|-\frac14
  \geq\frac{|q|}{2}.
\]
Thus \eqref{eq:cap-diffraction-formula} and
\eqref{eq:cap-codim-one-window-cancellation} give
\begin{equation}\label{eq:cap-codim-one-hu1-sum}
  \sigma_W(B_\varepsilon^*)
  \ll
  \sum_{q\in\mathcal R_\varepsilon}
    \frac{|\xi_3(q)|^2}{q^2}.
\end{equation}
Fix $n$ and assume
\begin{equation}\label{eq:cap-codim-one-epsilon-scale}
  \frac1{4M_{n+1}}\leq\varepsilon<\frac1{4M_n}.
\end{equation}

\smallskip\noindent\emph{Regime I: $\varepsilon\geq r_n/2$.}
Suppose first that $\varepsilon\geq r_n/2$.  For returns with $|q|\leq A_n$,
Equations~\eqref{eq:cap-codim-one-return-isolation} and
\eqref{eq:cap-codim-one-special-return} give $q=kq_n$ and
$\xi_j(q)=kq_nt_{j,n}$.  The condition $|\xi_1(q)|\leq\varepsilon$ implies
\[
  |k|\leq \frac{\varepsilon}{q_nt_{1,n}}.
\]
For every nonzero such $k$,
\[
  \frac{|\xi_3(kq_n)|^2}{(kq_n)^2}=t_{3,n}^2.
\]
Hence the contribution of $|q|\leq A_n$ to
\eqref{eq:cap-codim-one-hu1-sum} is
\begin{equation}\label{eq:cap-codim-one-small-return-contribution}
  \ll
  \left(1+\frac{\varepsilon}{q_nt_{1,n}}\right)t_{3,n}^2.
\end{equation}
Since $\varepsilon\geq r_n/2=q_nt_{1,n}/2$, division by $\varepsilon^3$
and the tail bounds give
\begin{equation}\label{eq:cap-codim-one-small-return-normalized}
  \frac{1}{\varepsilon^3}
  \left(1+\frac{\varepsilon}{q_nt_{1,n}}\right)t_{3,n}^2
  \ll
  \frac{t_{3,n}^2}{q_n^3t_{1,n}^3}
  \ll\frac1{M_n^9M_{n+1}^3}
  \longrightarrow0.
\end{equation}
If $q,q'\in\mathcal R_\varepsilon$, then
\[
  \operatorname{dist}_\infty((q-q')\alpha,\mathbb Z^3)
  \leq 2\varepsilon<\frac1{2M_n}.
\]
If in addition $0<|q-q'|\leq A_n$, then
\eqref{eq:cap-codim-one-return-isolation} gives
$q-q'\in q_n\mathbb Z$, and hence $|q-q'|\geq q_n$.  Thus distinct elements
of $\mathcal R_\varepsilon\cap\{|q|>A_n\}$ are $q_n$-separated.  Splitting the
positive and negative integers into intervals of length $q_n$ gives
\[
  \sum_{\substack{q\in\mathcal R_\varepsilon\\|q|>A_n}}q^{-2}
  \ll \sum_{r\geq0}(A_n+rq_n)^{-2}
  \ll\frac1{A_nq_n}.
\]
Using $|\xi_3(q)|<\varepsilon$, their contribution to
\eqref{eq:cap-codim-one-hu1-sum}, divided by $\varepsilon^3$, is
\begin{equation}\label{eq:cap-codim-one-large-return-normalized}
  \ll\frac1{\varepsilon A_nq_n}
  \ll\frac1{M_n^5},
\end{equation}
where the last estimate uses
$\varepsilon\geq r_n/2\asymp q_n/M_{n+1}$.

\smallskip\noindent\emph{Regime II: $\varepsilon<r_n/2$.}
Suppose next that $\varepsilon<r_n/2$.  Again $2\varepsilon<1/(2M_n)$.  If
$0<|q|\leq A_n$ is a $2\varepsilon$-return, then
\eqref{eq:cap-codim-one-return-isolation} gives $q=kq_n$, and
\eqref{eq:cap-codim-one-special-return} gives
\[
  \operatorname{dist}_\infty(q\alpha,\mathbb Z^3)
  =|k|r_n\geq r_n>2\varepsilon,
\]
a contradiction.  Hence every $\varepsilon$-return has $|q|>A_n$.
If two distinct $\varepsilon$-returns satisfied $|q-q'|\leq A_n$, their
difference would be a nonzero $2\varepsilon$-return in this forbidden range.
Thus $|q-q'|>A_n$, and
\[
  \sum_{q\in\mathcal R_\varepsilon}q^{-2}
  \ll\sum_{r\geq1}(rA_n)^{-2}
  \ll A_n^{-2}.
\]
and \eqref{eq:cap-codim-one-hu1-sum} gives
\begin{equation}\label{eq:cap-codim-one-second-regime}
  \frac{\sigma_W(B_\varepsilon^*)}{\varepsilon^3}
  \ll\frac1{\varepsilon A_n^2}
  \leq C\frac{M_n^2}{M_{n+1}}
  \longrightarrow0,
\end{equation}
because $\varepsilon\geq1/(4M_{n+1})$.  The two regimes cover every
sufficiently small $\varepsilon$.  Hence
\(
  \sigma_W(B_\varepsilon^*)=o(\varepsilon^3),
\)
and Theorem~\ref{thm:geometric-spectral} gives $\mu_W\in\HU_1$.

\smallskip\noindent\emph{Failure for a derived interval.}
The vector $\gamma_1$ from \eqref{eq:cap-codim-one-primal-vectors} has internal
component $\alpha_1\in(0,1)$, so
\begin{equation}\label{eq:cap-codim-one-derived-window}
  W':=W\cap(W-\alpha_1)
\end{equation}
is an interval of length $1-\alpha_1$.  For every $(\xi,\eta)\in\Lambda$,
\[
  (1-\alpha_1)\eta
  \equiv-\xi_1-(1-\alpha_1)\xi_3\pmod{\mathbb Z},
\]
and therefore
\begin{equation}\label{eq:cap-codim-one-derived-fourier}
  |\widehat{\chi}_{W'}(\eta)|
  =\frac{|\sin\pi(\xi_1+(1-\alpha_1)\xi_3)|}{\pi|\eta|}.
\end{equation}
At $q=q_n=M_n^3$, choose
\begin{equation}\label{eq:cap-codim-one-special-dual-p}
  p_{j,n}:=q_n\alpha_{j,n}\in\mathbb Z.
\end{equation}
Then
\[
  \xi_{j,n}=q_nt_{j,n}>0,
  \qquad
  \eta_n=q_n-\xi_{3,n},
\]
with $0<\xi_{3,n}<1$ for all sufficiently large $n$ and
$\|\xi_n\|_2\to0$.  The first coordinate dominates:
\begin{equation}\label{eq:cap-codim-one-first-coordinate-dominates}
  \|\xi_n\|_2\asymp\xi_{1,n}
  \asymp\frac{M_n^3}{M_{n+1}},
  \qquad
  \xi_{3,n}=o(\xi_{1,n}).
\end{equation}
Put
\[
  u_n:=\xi_{1,n}+(1-\alpha_1)\xi_{3,n}.
\]
By \eqref{eq:cap-codim-one-first-coordinate-dominates},
$u_n/\xi_{1,n}\to1$ and $u_n\to0$, while
$\eta_n/q_n\to1$.  Hence, for all large $n$,
$|\sin(\pi u_n)|\geq 2|u_n|$ and $|\eta_n|\leq2q_n$.  Equation
\eqref{eq:cap-codim-one-derived-fourier} therefore gives
\begin{equation}\label{eq:cap-codim-one-derived-atom-size}
  |\widehat{\chi}_{W'}(\eta_n)|^2
  \gg\frac{\xi_{1,n}^2}{q_n^2}.
\end{equation}
Taking $\varepsilon_n:=2\|\xi_n\|_2$ in
\eqref{eq:cap-diffraction-formula} gives
\begin{equation}\label{eq:cap-codim-one-derived-nonhu}
  \frac{\sigma_{W'}(B_{\varepsilon_n}^*)}{\varepsilon_n^3}
  \gg\frac1{q_n^2\xi_{1,n}}
  \asymp\frac{M_{n+1}}{M_n^9}
  \longrightarrow\infty.
\end{equation}
Hence $\mu_{W'}\notin\HU_1$ by
Theorem~\ref{thm:geometric-spectral}.

\smallskip\noindent\emph{From the derived window to failure of $\HU_2$.}
Finally, let
\begin{equation}\label{eq:cap-codim-one-D1}
  D_1:=-e_1+\alpha_1e_3
\end{equation}
be the physical component of $\gamma_1$.  The set
\(
  \{\gamma_V:\gamma\in\Gamma,\ \gamma_H\in W-W\}
\)
has a positive separation constant.  Under the identification
$\Sh_2(\mathbb R^3)\cong\mathbb R^3/\{D\sim-D\}$,
$D\mapsto[(-D/2,D/2)]$, its image is therefore discrete away from the
diagonal shape.  Choose
$\varphi\in C_c(\Sh_2(\mathbb R^3))$ equal to one at the shape of
$(-D_1/2,D_1/2)$ and supported so narrowly that no other nonzero displacement
class contributes; in particular, $\varphi(0)=0$.

For $\omega=\Gamma+(v,h)$, an accepted pair with displacement $D_1$ is
generated by $\gamma$ and $\gamma+\gamma_1$.  Both points occur exactly when
$\gamma_H+h\in W'$, and their barycentre is
$\gamma_V+v+D_1/2$.  The reversed ordering contributes a second ordered pair
with the same barycentre and the same unordered shape.  Therefore, exactly,
\begin{equation}\label{eq:cap-codim-one-pair-reduction}
  T_2[f,\varphi]
  =2S_{W'}(\tau_{D_1/2}f).
\end{equation}
This is the case $k=j=2$ and multiplicities $(1,1)$ of
\eqref{eq:cap-isolated-derived-window-statistic}.
If $\mu_W\in\HU_2$, applying the defining variance condition to this fixed
$\varphi$ and $f=\chi_{B_R}$ would imply
\[
  \Var\bigl(S_{W'}(\tau_{D_1/2}\chi_{B_R})\bigr)
  =o(\Vol_3(B_R)).
\]
Translation invariance of $\mu_{W'}$ removes the fixed translate of the ball,
so $\mu_{W'}\in\HU_1$, contradicting
\eqref{eq:cap-codim-one-derived-nonhu}.  Thus
$\mu_W\notin\HU_2$.
\end{proof}

\subsection{Ball windows in internal dimensions two and three}
\label{subsec:cap-balls-low-dim}

This subsection proves the positive ball-window result.  The physical
dimension $d=\dim V$ is arbitrary, while the internal space is
$H=\mathbb R^m$ with $m\in\{2,3\}$.  If the regular window
$W\subset H$ is a Euclidean ball, then ordinary hyperuniformity implies
$\HU_k$ for every $k$.

\begin{lemma}[Fourier bounds for intersections of balls]
\label{lem:cap-ball-intersections}
Let \(m\in\{2,3\}\) and let \(W\subset\mathbb R^m\) be a Euclidean ball.

\begin{enumerate}
\item If \(W'\) is a finite intersection of translates of \(W\), then
\begin{equation}\label{eq:cap-ball-intersection-decay}
  |\widehat{\chi}_{W'}(\eta)|
  \leq C_{W'}(1+\|\eta\|)^{-(m+1)/2}
  \qquad(\eta\in\mathbb R^m).
\end{equation}

\item There are \(c_W>0\) and \(R_W>0\) such that
\begin{equation}\label{eq:cap-ball-two-harmonic-lower-bound}
  |\widehat{\chi}_W(\eta)|^2
  +|\widehat{\chi}_W(2\eta)|^2
  \geq c_W\|\eta\|^{-(m+1)}
\end{equation}
whenever \(\|\eta\|\geq R_W\).
\end{enumerate}
\end{lemma}

\begin{proof}
\smallskip\noindent\emph{Boundary decomposition.}
Write
\begin{equation}\label{eq:cap-equal-ball-intersection}
  W'=\bigcap_{j=1}^N B(c_j,R_0),
\end{equation}
where $R_0$ is the radius of $W$.  After discarding repetitions, we may and
shall assume that the centers $c_1,\ldots,c_N$ are distinct.  If
$\Vol_m(W')=0$, the first assertion is immediate.  Otherwise $W'$ is a
compact convex body with Lipschitz boundary.  Since distinct spheres of radius
$R_0$ intersect in sets of $(m-1)$-dimensional surface measure zero, up to such
a null set the boundary of $W'$ is the disjoint union of the spherical faces
\begin{equation}\label{eq:cap-spherical-faces}
  F_i:=\partial B(c_i,R_0)\cap\bigcap_{j\neq i}B(c_j,R_0),
  \qquad 1\leq i\leq N.
\end{equation}
On $\partial B(c_i,R_0)$, every remaining ball condition is a linear
half-space condition, since
\begin{equation}\label{eq:cap-ball-halfspace-reduction}
  \|v-c_j\|^2\leq\|v-c_i\|^2
  \quad\Longleftrightarrow\quad
  2\langle c_i-c_j,v\rangle
  \leq \|c_i\|^2-\|c_j\|^2.
\end{equation}

Let $\eta=ru$ with $r>0$ and $\|u\|=1$.  The divergence theorem gives
\begin{equation}\label{eq:cap-divergence-fourier}
  \widehat{\chi}_{W'}(ru)
  =-\frac{1}{2\pi i r}
    \sum_{i=1}^N\int_{F_i}
      \langle u,n_i(v)\rangle
      e^{-2\pi i r\langle u,v\rangle}\,dS(v),
\end{equation}
where $n_i(v)=(v-c_i)/R_0$ on $F_i$.

\smallskip\noindent\emph{Dimension two.}
Suppose first that $m=2$.  Choose a unit vector $e\perp u$ and split each
circle into the two semicircles on which
$\langle u,n_i\rangle$ has fixed sign.  On the positive semicircle write
\begin{equation}\label{eq:cap-circle-direction-coordinates}
  v=c_i+R_0\bigl(se+\sqrt{1-s^2}\,u\bigr),
  \qquad -1<s<1.
\end{equation}
Then
\begin{equation}\label{eq:cap-circle-amplitude-cancellation}
  dS=R_0(1-s^2)^{-1/2}\,ds,
  \qquad
  \langle u,n_i(v)\rangle=\sqrt{1-s^2},
\end{equation}
so their product is $R_0\,ds$.  By
\eqref{eq:cap-ball-halfspace-reduction}, after writing $s=\sin t$ each
boundary equation has the form $A\sin t+B\cos t=C$.  Unless this
equation is identically satisfied (in which case the constraint creates no
endpoint), it has at most two solutions on the semicircle.  The admissible
values of $s$ therefore form a union of at most $2N+1$ intervals.  On every such interval the remaining
oscillatory integral is, up to a constant phase,
\begin{equation}\label{eq:cap-circle-oscillatory-integral}
  R_0\int_a^b e^{-2\pi i rR_0\sqrt{1-s^2}}\,ds.
\end{equation}
The phase $\phi(s)=\sqrt{1-s^2}$ satisfies
$|\phi''(s)|=(1-s^2)^{-3/2}\geq1$.  Van der Corput's second-derivative
estimate therefore bounds \eqref{eq:cap-circle-oscillatory-integral} by
$C_{R_0}r^{-1/2}$, uniformly in $a,b$ and $u$.  If an endpoint is
$\pm1$, apply the estimate first on truncated subintervals and pass to the
limit; the same constant is retained.  The negative semicircle is
identical.  Hence
\begin{equation}\label{eq:cap-two-dimensional-face-bound}
  |\widehat\chi_{W'}(ru)|\leq C_{W'}r^{-3/2}
  \qquad(m=2,\ r\geq1).
\end{equation}

\smallskip\noindent\emph{Dimension three.}
Suppose now that $m=3$.  Parametrize the positive hemisphere of
$\partial B(c_i,R_0)$ by
\begin{equation}\label{eq:cap-sphere-direction-coordinates}
  v=c_i+R_0\bigl(z+\sqrt{1-\|z\|^2}\,u\bigr),
  \qquad z\in u^\perp,\quad \|z\|<1.
\end{equation}
Here
\begin{equation}\label{eq:cap-sphere-amplitude-cancellation}
  dS=R_0^2(1-\|z\|^2)^{-1/2}\,dz,
  \qquad
  \langle u,n_i(v)\rangle=\sqrt{1-\|z\|^2},
\end{equation}
so their product is $R_0^2\,dz$.  Write $z=s\theta$ with
$\theta\in S^1\subset u^\perp$ and $0<s<1$.  For fixed $\theta$, writing $s=\sin t$ again turns every boundary equation
from \eqref{eq:cap-ball-halfspace-reduction} into
$A\sin t+B\cos t=C$.  An identically satisfied equation creates no radial
endpoint; every other one has at most two solutions.  Hence the admissible
radial set is a union of at most $2N+1$ intervals.  On each such interval the radial
integral is, up to a constant phase,
\begin{equation}\label{eq:cap-sphere-radial-integral}
  R_0^2\int_a^b e^{-2\pi i rR_0\sqrt{1-s^2}}s\,ds
  =R_0^2\int_{\sqrt{1-b^2}}^{\sqrt{1-a^2}}
      y e^{-2\pi i rR_0y}\,dy.
\end{equation}
For $c=2\pi rR_0$,
\[
  \int_A^B y e^{-icy}\,dy
  =\left[-\frac{y e^{-icy}}{ic}\right]_A^B
   +\frac1{ic}\int_A^B e^{-icy}\,dy,
\]
so the radial integral is $O_{R_0}(r^{-1})$, uniformly in $a,b$, $\theta$ and
$u$.  Integrating over $\theta$ and treating the negative hemisphere in the
same way gives
\begin{equation}\label{eq:cap-three-dimensional-face-bound}
  |\widehat\chi_{W'}(ru)|\leq C_{W'}r^{-2}
  \qquad(m=3,\ r\geq1).
\end{equation}
Together with the trivial bound
$|\widehat\chi_{W'}|\leq\Vol_m(W')$, this proves
\eqref{eq:cap-ball-intersection-decay}.  The preceding estimates are uniform
in the positions of the centres: sphere intersections have surface measure
zero, and the one-dimensional endpoint count depends only on $N$, while the
oscillatory-integral bounds depend only on $R_0$.  Thus one may take
$C_{W'}=C(N,R_0)$.

\smallskip\noindent\emph{Two-frequency lower bound for the original ball.}
For the second assertion, let $R_0$ again denote the radius of $W$.
Translation of $W$ changes its Fourier transform only by a unimodular phase,
so we may assume that $W$ is centred at the origin.  The Bessel formula and
its large-argument asymptotic then give, with $r=\|\eta\|$,
\begin{equation}\label{eq:cap-ball-bessel-asymptotic}
  \widehat{\chi}_W(\eta)
  =c_m r^{-(m+1)/2}
   \left(
     \cos\left(2\pi R_0r-\frac{\pi(m+1)}4\right)+O(r^{-1})
   \right),
\end{equation}
where $c_m\neq0$.  If
$t=2\pi R_0r-\pi(m+1)/4$, the two leading oscillatory factors at
$\eta$ and $2\eta$ are
\begin{equation}\label{eq:cap-ball-two-phases}
  \cos t,
  \qquad
  \cos\left(2t+\frac{\pi(m+1)}4\right).
\end{equation}
For $m=2,3$ they have no common zero.  Indeed, if
$\cos t=0$, then $t\equiv\pi/2\pmod\pi$; the second factor equals
$\cos(7\pi/4)$ for $m=2$ and $\cos(2\pi)$ for $m=3$, up to its period, and is
nonzero.  After multiplying by $r^{m+1}$, the squared sum in
\eqref{eq:cap-ball-two-harmonic-lower-bound} has leading term
\[
  |c_m|^2\left(\cos^2 t
    +2^{-(m+1)}\cos^2\left(2t+\frac{\pi(m+1)}4\right)\right),
\]
whose minimum on $\mathbb R/\pi\mathbb Z$ is positive.  The $O(r^{-1})$
remainder in \eqref{eq:cap-ball-bessel-asymptotic} is uniform in direction,
so \eqref{eq:cap-ball-two-harmonic-lower-bound} follows for sufficiently
large $r$.  The same argument fails for $m=1$: when $\cos t=0$, the second
factor in \eqref{eq:cap-ball-two-phases} also vanishes because the phase shift
is $\pi/2$.  This is consistent with the interval-window separation in
Theorem~\ref{thm:cap-codim-one-separation}.
\end{proof}

\begin{theorem}[Ball windows in internal dimensions two and three]
\label{thm:cap-low-dimensional-balls}
Let \((V,\mathbb R^m,\Gamma)\) be a Euclidean cut-and-project scheme with
\(m\in\{2,3\}\), and let \(W\subset\mathbb R^m\) be a Euclidean ball.
Then, for every \(k\geq1\),
\begin{equation}\label{eq:cap-low-dimensional-collapse}
  \mu_W\in\HU_k
  \quad\Longleftrightarrow\quad
  \mu_W\in\HU_1.
\end{equation}
\end{theorem}

\begin{proof}
Assume \(\mu_W\in\HU_1\), and let \(W'\) be a finite intersection of translates of \(W\) with nonempty interior arising in
Proposition~\ref{prop:cap-pattern-reduction}.  By
Lemma~\ref{lem:cap-ball-intersections}, for all sufficiently large
\(\|\eta\|\),
\begin{equation}\label{eq:cap-derived-ball-fourier-comparison}
  |\widehat{\chi}_{W'}(\eta)|^2
  \leq C_{W'}\left(
    |\widehat{\chi}_W(\eta)|^2
    +|\widehat{\chi}_W(2\eta)|^2
  \right).
\end{equation}
There is \(\varepsilon_0>0\) such that every nonzero
\((\xi,\eta)\in\Gamma^\perp\) with \(\|\xi\|<\varepsilon_0\) satisfies
\(\|\eta\|\geq R_W\).  Otherwise a sequence of nonzero lattice points
would have \(\xi\to0\) and bounded \(\eta\), contradicting discreteness of
\(\Gamma^\perp\) and
\(\Gamma^\perp\cap(\{0\}\times H^*)=\{0\}\).

The projection \((\xi,\eta)\mapsto\xi\) is injective on \(\Gamma^\perp\):
if two dual points have the same physical component, their difference lies in
\(\Gamma^\perp\cap(\{0\}\times H^*)=\{0\}\).  Thus, for
\(0<\varepsilon<\varepsilon_0\),
\begin{align*}
 \sigma_{W'}(B_\varepsilon^*)
 &\leq \frac{C_{W'}}{\operatorname{covol}(\Gamma)^2}
   \sum_{\substack{(\xi,\eta)\in\Gamma^\perp\setminus\{0\}\\
                    \|\xi\|<\varepsilon}}
   \left(|\widehat\chi_W(\eta)|^2
        +|\widehat\chi_W(2\eta)|^2\right)\\
 &\leq C_{W'}\left(
      \sigma_W(B_\varepsilon^*)+\sigma_W(B_{2\varepsilon}^*)
    \right).
\end{align*}
Indeed, the first sum is exactly the contribution to
\(\sigma_W(B_\varepsilon^*)\), while the map
\((\xi,\eta)\mapsto(2\xi,2\eta)\) sends the second sum injectively into the
terms defining \(\sigma_W(B_{2\varepsilon}^*)\).  Hence
\begin{equation}\label{eq:cap-derived-ball-spectrum-comparison}
  \sigma_{W'}(B_\varepsilon^*)
  \leq C_{W'}\bigl(
    \sigma_W(B_\varepsilon^*)+\sigma_W(B_{2\varepsilon}^*)
  \bigr).
\end{equation}
Spectral hyperuniformity of \(\mu_W\) gives
\(\sigma_W(B_\varepsilon^*)=o(\varepsilon^d)\), and hence the same estimate
for \(\sigma_{W'}\).  Proposition~\ref{prop:cap-pattern-reduction}(ii),
\textup{(b)} $\Rightarrow$ \textup{(a)}, gives $\mu_W\in\HU_k$.  The converse is
Corollary~\ref{cor:cap-monotonicity}.
\end{proof}

\subsection{Ball windows in internal dimensions four and higher}
\label{subsec:cap-critical-dimension}

This subsection proves the complementary negative result.  Here both the
physical and internal spaces are $\mathbb R^m$.  For every $m\geq4$ we
construct a process with the unit ball $W=B_1^m$ as internal window that
belongs to $\HU_1\setminus\HU_2$.  The failure of $\HU_2$ is detected by
an intersection of two translated copies of $W$.  For $0<s<2$ set
\begin{equation}\label{eq:cap-lens-definition}
  W_{m,s}:=W\cap(W-se_1).
\end{equation}
Along the $e_1$-axis, the lens has decay $r^{-2}$, whereas the ball has
decay $r^{-(m+1)/2}$; the lens is therefore slower exactly when $m\geq4$.

\begin{lemma}[Intersection of two translated balls]
\label{lem:cap-lens-fourier}
For \(m\geq4\) and \(0<s<2\), there is \(c_{m,s}>0\) such that
\begin{equation}\label{eq:cap-lens-axial-asymptotic}
  \lim_{r\to\infty} r^2
  |\widehat{\chi}_{W_{m,s}}(re_1)|=c_{m,s}.
\end{equation}
In particular,
\begin{equation}\label{eq:cap-lens-intensity-asymptotic}
  |\widehat{\chi}_{W_{m,s}}(re_1)|^2\asymp r^{-4}.
\end{equation}
\end{lemma}

\begin{proof}
Set \(\widetilde W_{m,s}:=W_{m,s}+se_1/2\).  Translation does not change
\(|\widehat{\chi}_{W_{m,s}}|\).  For \(x\in\mathbb R\), define
\begin{equation}\label{eq:cap-lens-section-function}
  a_{m,s}(x)
  :=\Vol_{m-1}\bigl(
      \{y\in e_1^\perp:xe_1+y\in\widetilde W_{m,s}\}
    \bigr)
  =\kappa_{m-1}
    \bigl[1-(|x|+s/2)^2\bigr]_+^{(m-1)/2},
\end{equation}
where \(\kappa_{m-1}:=\Vol_{m-1}(B_1^{m-1})\).  Fubini's theorem gives
\begin{equation}\label{eq:cap-lens-axial-slice-fourier}
  \widehat{\chi}_{\widetilde W_{m,s}}(re_1)=\widehat a_{m,s}(r).
\end{equation}
The one-sided derivatives at the origin satisfy
\begin{equation}\label{eq:cap-lens-derivative-jump}
  a_{m,s}'(0+)-a_{m,s}'(0-)
  =-\kappa_{m-1}(m-1)s
    \left(1-\frac{s^2}{4}\right)^{(m-3)/2}
  =:J_{m,s},
\end{equation}
where \(J_{m,s}\neq0\).  Let $x_0:=1-s/2$ be the positive endpoint of the support.  As
$x\uparrow x_0$,
\[
  a_{m,s}(x)\asymp(x_0-x)^{(m-1)/2},
  \qquad
  a_{m,s}''(x)\asymp(x_0-x)^{(m-5)/2}.
\]
Thus $a_{m,s}''$ is integrable at the endpoints and $a_{m,s}'$ vanishes there
exactly for $m\geq4$.  For $m=3$ the first derivative has nonzero endpoint
limits, so the distributional second derivative has additional endpoint
atoms.  For \(m\geq4\), in the sense of distributions,
\begin{equation}\label{eq:cap-lens-second-derivative}
  D^2a_{m,s}=g_{m,s}+J_{m,s}\delta_0
  \qquad\text{with }g_{m,s}\in L^1(\mathbb R).
\end{equation}
Taking Fourier transforms gives
\begin{equation}\label{eq:cap-lens-fourier-second-derivative}
  -4\pi^2r^2\widehat a_{m,s}(r)
  =\widehat g_{m,s}(r)+J_{m,s}.
\end{equation}
The Riemann--Lebesgue lemma gives
$\widehat g_{m,s}(r)\to0$, and therefore
\[
  r^2\widehat a_{m,s}(r)
  \longrightarrow -\frac{J_{m,s}}{4\pi^2}.
\]
Since $J_{m,s}\neq0$ and translation does not change the modulus of the
Fourier transform, \eqref{eq:cap-lens-axial-asymptotic} follows with
$c_{m,s}=|J_{m,s}|/(4\pi^2)$.
\end{proof}

Let
\begin{equation}\label{eq:cap-golden-ratios}
  \omega:=\frac{1+\sqrt5}{2},
  \qquad
  \omega':=\frac{1-\sqrt5}{2}.
\end{equation}
For \(m\geq1\), define
\(\Gamma_m^\perp<\mathbb R^m\times\mathbb R^m\) by
\begin{equation}\label{eq:cap-fibonacci-dual-lattice}
  \Gamma_m^\perp
  :=\left\{(\xi,\eta):
    \begin{array}{l}
      \xi_j=p_j+\omega q_j,\\
      \eta_j=p_j+\omega' q_j
    \end{array}
    \text{ for }p,q\in\mathbb Z^m
  \right\}.
\end{equation}
Define
\begin{equation}\label{eq:cap-fibonacci-primal-lattice}
  \Gamma_m
  =\left\{(v,h):
    \begin{array}{l}
      v_j=(-\omega' r_j+s_j)/\sqrt5,\\
      h_j=(\omega r_j-s_j)/\sqrt5
    \end{array}
    \text{ for }r,s\in\mathbb Z^m
  \right\}.
\end{equation}
For one coordinate, the pairing of
$(p+\omega q,p+\omega' q)$ with
$((-\omega' r+s)/\sqrt5,(\omega r-s)/\sqrt5)$ is
$p r+q s\in\mathbb Z$.  Hence the lattice in
\eqref{eq:cap-fibonacci-dual-lattice} is contained in the annihilator of
\eqref{eq:cap-fibonacci-primal-lattice}.  Their one-dimensional covolumes are
$\sqrt5$ and $1/\sqrt5$, respectively; in $m$ coordinates the product of the
two covolumes is one.  Since the annihilator of a lattice has reciprocal
covolume, the inclusion has index one.  Thus
\eqref{eq:cap-fibonacci-dual-lattice} is exactly the annihilator of
\eqref{eq:cap-fibonacci-primal-lattice}.  The physical projection of
$\Gamma_m$ is injective because $\omega'$ is irrational, and the internal
projection is dense because $\mathbb Z+\omega\mathbb Z$ is dense in
$\mathbb R$.  Thus
$(\mathbb R^m,\mathbb R^m,\Gamma_m)$ is a Euclidean cut-and-project scheme.

\begin{theorem}[Ball windows in internal dimensions four and higher]
\label{thm:cap-four-dimensional-obstruction}
For every \(m\geq4\), let \(V=H=\mathbb R^m\), let \(\Gamma_m\) be the
lattice in \eqref{eq:cap-fibonacci-primal-lattice}, and let \(W=B_1^m\).
Then
\begin{equation}\label{eq:cap-four-dimensional-separation}
  \mu_W\in\HU_1\setminus\HU_2.
\end{equation}
\end{theorem}

\begin{proof}
For every coordinate in \eqref{eq:cap-fibonacci-dual-lattice},
\begin{equation}\label{eq:cap-quadratic-norm-identity}
  \xi_j\eta_j
  =p_j^2+p_jq_j-q_j^2\in\mathbb Z.
\end{equation}
If \((\xi,\eta)\neq0\) and \(\|\xi\|<\varepsilon<1\), then some coordinate
pair \((\xi_j,\eta_j)\) is nonzero.  Irrationality of \(\omega\) shows that
\(\xi_j=0\) or \(\eta_j=0\) can occur only when \(p_j=q_j=0\); hence this
nonzero coordinate pair has \(\xi_j\eta_j\in\mathbb Z\setminus\{0\}\).
Equation~\eqref{eq:cap-quadratic-norm-identity} therefore gives
\(|\eta_j|\geq|\xi_j|^{-1}\geq\varepsilon^{-1}\), and thus
\(\|\eta\|\geq\varepsilon^{-1}\).  Equivalently, every nonzero
\((\xi,\eta)\in\Gamma_m^\perp\) with \(\|\xi\|<1\) satisfies
\(\|\eta\|\geq\|\xi\|^{-1}\).  By \cite[Remark~3.5]{BH24}, the Euclidean ball is Fourier smooth with
exponent $\vartheta=1$.  The same coordinate argument gives the repellence
condition in the sup norm used in \cite[Definition~1.8]{BH24}: if
$\|\xi\|_\infty<\varepsilon<1$, choose a nonzero coordinate pair
$(\xi_j,\eta_j)$.  Then $|\xi_j|<\varepsilon$ and
\eqref{eq:cap-quadratic-norm-identity} gives
$|\eta_j|\geq|\xi_j|^{-1}>\varepsilon^{-1}$, so
$\|\eta\|_\infty>\varepsilon^{-1}$.  Thus $\Gamma_m^\perp$ is
$1$-repellent on the right.  Therefore \cite[Theorem~5.1]{BH24}, with
physical and internal dimensions both equal to $m$, gives for all sufficiently
small $\varepsilon$
\begin{equation}\label{eq:cap-higher-ball-hu1-bound}
  \sigma_W(B_\varepsilon^*)\ll\varepsilon^{m+1}.
\end{equation}
Because \(\Vol_m(B_\varepsilon^*)\asymp\varepsilon^m\), this proves
\(\mu_W\in\HU_1\).

Set \(s:=1/\sqrt5\).  Taking \(r_1=0\), \(s_1=1\), and all remaining
coordinates zero in \eqref{eq:cap-fibonacci-primal-lattice} gives
\begin{equation}\label{eq:cap-lens-primal-displacement}
  \gamma_0=(se_1,-se_1)\in\Gamma_m.
\end{equation}
The set
\begin{equation}\label{eq:cap-relevant-difference-set}
  \{\gamma_V:\gamma\in\Gamma_m,\ \gamma_H\in W-W\}
\end{equation}
has a positive separation constant by the estimate at the beginning of
Section~\ref{subsec:cap-pattern-windows}.  Choose
$\varphi\in C_c(\Sh_2(\mathbb R^m))$ equal to one at the shape of
$(0,se_1)$, with $\varphi(0)=0$, and with support containing no other nonzero
admissible displacement class.  A pair with displacement $se_1$ is generated
by $\gamma$ and $\gamma+\gamma_0$.  Since $(\gamma_0)_H=-se_1$, both points
are accepted exactly when
\[
  \gamma_H+h\in W\cap(W+se_1)=W_{m,s}+se_1.
\]
Their barycentre is $\gamma_V+v+se_1/2$, and the reversed ordering gives the
same term.  Therefore
\begin{equation}\label{eq:cap-lens-pair-statistic}
  T_2[f,\varphi]
  =2S_{W_{m,s}+se_1}(\tau_{se_1/2}f).
\end{equation}
By \eqref{eq:cap-window-translation-law},
$\mu_{W_{m,s}+se_1}=\mu_{W_{m,s}}$.  Moreover
$\widehat{\tau_{se_1/2}f}(\xi)
=e^{\pi i\langle\xi,se_1\rangle}\widehat f(\xi)$, so
\eqref{eq:cap-lens-pair-statistic} gives
\[
  \Var(T_2[f,\varphi])
  =4\int_{(\mathbb R^m)^*}|\widehat f(\xi)|^2\,d\sigma_{W_{m,s}}(\xi).
\]
Hence
\begin{equation}\label{eq:cap-lens-pair-spectrum}
  \sigma^{(2)}_{\varphi,\varphi}=4\sigma_{W_{m,s}}.
\end{equation}

Let \((F_n)_{n\geq0}\) be the Fibonacci sequence, \(F_0=0\), \(F_1=1\),
and in the first coordinate of
\eqref{eq:cap-fibonacci-dual-lattice} take
\begin{equation}\label{eq:cap-fibonacci-approximants}
  p_n:=-F_{n+1},
  \qquad
  q_n:=F_n.
\end{equation}
The Fibonacci identities give
\begin{equation}\label{eq:cap-fibonacci-dual-sequence}
  \xi_n=(-1)^{n+1}\omega^{-n}e_1,
  \qquad
  \eta_n=-\omega^n e_1,
  \qquad
  \|\xi_n\|\,\|\eta_n\|=1.
\end{equation}
Thus this sequence attains the reciprocal-scale relation used above and is
aligned with the axial direction in Lemma~\ref{lem:cap-lens-fourier}.
Lemma~\ref{lem:cap-lens-fourier} gives
\begin{equation}\label{eq:cap-lens-atom-asymptotic}
  |\widehat{\chi}_{W_{m,s}}(\eta_n)|^2
  \asymp\omega^{-4n}
  =\|\xi_n\|^4.
\end{equation}
Put \(\varepsilon_n:=2\|\xi_n\|\).  From
\eqref{eq:cap-diffraction-formula},
\begin{equation}\label{eq:cap-higher-normalized-divergence}
  \frac{\sigma^{(2)}_{\varphi,\varphi}(B_{\varepsilon_n}^*)}
       {\Vol_m(B_{\varepsilon_n}^*)}
  \gg \|\xi_n\|^{4-m}.
\end{equation}
For \(m=4\) the right-hand side is bounded below by a positive constant;
for \(m>4\) it tends to infinity.  Hence \(\mu_W\notin\HU_2\) by
Theorem~\ref{thm:geometric-spectral}.
\end{proof}

\section{First-order stealthiness without \texorpdfstring{$\HU_2$}{HU2}}
\label{sec:stealthy-nonhu2}

\subsection{A diffuse model}
\label{subsec:simple-stealthy-example}

\begin{proposition}
\label{prop:beat-frequency-example}
There exists an ergodic $\mathbb R$-invariant probability measure $\mu$ on
$\cM(\mathbb R)$ such that every $p$ in its support is absolutely continuous
with a bounded density bounded away from zero and
\begin{equation}\label{eq:beat-main-conclusion}
  \mu\text{ is $1$-stealthy},
  \qquad
  \mu\notin\HU_2.
\end{equation}
\end{proposition}

\begin{proof}
Choose a real even $\psi\in C_c^\infty(\mathbb R)$ with
$\widehat\psi(3/2)\neq0$.  There are an open interval $I$ with
$\overline I\subset(1,2)$ and $\eta>0$ such that
\begin{equation}\label{eq:beat-test-lower-bound}
  |\widehat\psi(c)|\geq\eta
  \qquad(c\in I).
\end{equation}
For $n\geq1$ put $b_n:=2^{-n-3}$.  Choose $c_n\in I$ and
$0<\delta_n<b_n^8$ inductively so that
\begin{equation}\label{eq:beat-frequency-pairs}
  \lambda_{2n-1}:=c_n+\frac{\delta_n}{2},
  \qquad
  \lambda_{2n}:=c_n-\frac{\delta_n}{2}
\end{equation}
lie in $(1,2)$ and the family $(\lambda_j)_{j\geq1}$ is rationally
independent.  At each stage the forbidden choices $(c,\delta)$ lie in a
countable union of affine lines.  Set
\begin{equation}\label{eq:beat-amplitudes}
  a_{2n-1}=a_{2n}:=b_n.
\end{equation}
Then $\sum_j a_j=1/4$.

Fix $0<\kappa\leq1$, let $\theta=(\theta_j)$ be Haar-uniform on
$\Theta:=(\mathbb R/2\pi\mathbb Z)^{\mathbb N}$, and set
\begin{equation}\label{eq:beat-random-density-ansatz}
  \rho_\theta(x)
  :=1+\kappa\sum_{j\geq1}a_j
       \cos(2\pi\lambda_jx+\theta_j).
\end{equation}
The series converges uniformly, $\rho_\theta$ is bounded above and away
from zero, and
\begin{equation}\label{eq:beat-random-measure}
  p_\theta(f):=\int_{\mathbb R}f(x)\rho_\theta(x)\,d\Vol_1(x)
\end{equation}
defines a diffuse random measure.

Let $\mathbb R$ act on $\Theta$ by
\begin{equation}\label{eq:beat-torus-action}
  (t.\theta)_j:=\theta_j+2\pi\lambda_jt\pmod{2\pi}.
\end{equation}
From \eqref{eq:beat-random-density-ansatz},
\[
  \rho_{t.\theta}(x)=\rho_\theta(x+t),
  \qquad
  p_{t.\theta}=(-t).p_\theta.
\]
Hence $\mu:=(\theta\mapsto p_\theta)_*m_\Theta$ is $\mathbb R$-invariant.  A
character of $\Theta$ has the form
$\theta\mapsto e^{i\sum_j n_j\theta_j}$ with finitely many nonzero
$n_j\in\mathbb Z$; invariance under \eqref{eq:beat-torus-action} requires
$\sum_jn_j\lambda_j=0$.  Rational independence therefore leaves only the
constant invariant character, so the action and its factor $\mu$ are
ergodic.
Moreover $3/4\leq\rho_\theta\leq5/4$ for every $\theta$.  The map
$\theta\mapsto p_\theta$ is continuous for the vague topology: a finite
number of modes is continuous in $\theta$, while the remaining tail is
uniformly small.  Since Haar measure has full support on $\Theta$, the
support of $\mu$ is the compact image $\{p_\theta:\theta\in\Theta\}$.
Thus every measure in the support has a density between $3/4$ and $5/4$.
In particular, $\mu$ is locally $L^r$-integrable for every finite $r$.

For $f\in C_c^\infty(\mathbb R)$,
\begin{align*}
  S_f(p_\theta)-\int f
  =\frac{\kappa}{2}\sum_{j\geq1}a_j
   \left(
     e^{i\theta_j}\widehat f(-\lambda_j)
     +e^{-i\theta_j}\widehat f(\lambda_j)
   \right).
\end{align*}
The characters $e^{\pm i\theta_j}$ are orthogonal in $L^2(\Theta)$, so
\[
  \Var_\mu(S_f)
  =\frac{\kappa^2}{4}\sum_{j\geq1}a_j^2
    \left(|\widehat f(\lambda_j)|^2+|\widehat f(-\lambda_j)|^2\right).
\]
Comparison with the Bartlett covariance identity yields
\begin{equation}\label{eq:beat-first-bartlett}
  \boldsymbol\sigma^{(1)}
  =\frac{\kappa^2}{4}
   \sum_{j\geq1}a_j^2
      (\delta_{\lambda_j}+\delta_{-\lambda_j}).
\end{equation}
All $\lambda_j$ lie in $(1,2)$, so this measure vanishes on $(-1,1)$.
Thus $\mu$ is $1$-stealthy.

For the second-order statistic, $\psi$ is real and even, so
$\widehat\psi$ is real and even, and
$\varphi_\psi\in C_c(\Sh_2(\mathbb R))$ is well defined by
\begin{equation}\label{eq:beat-pair-shape-function}
  \varphi_\psi\bigl([(-r/2,r/2)]\bigr):=\psi(r).
\end{equation}
Since $p_\theta$ is diffuse, the derived pair measure has, in
barycentre--difference coordinates, density
\begin{equation}\label{eq:beat-pair-density}
  H_\psi(\theta,c)
  :=\int_{\mathbb R}\psi(r)
     \rho_\theta(c-r/2)\rho_\theta(c+r/2)\,d\Vol_1(r).
\end{equation}
For distinct $j,l$, the two ordered cross terms in
$\rho_\theta(c-r/2)\rho_\theta(c+r/2)$ are
\begin{align*}
 &\kappa^2a_ja_l
 \cos\bigl(2\pi\lambda_j(c-r/2)+\theta_j\bigr)
 \cos\bigl(2\pi\lambda_l(c+r/2)+\theta_l\bigr),\\
 &\kappa^2a_ja_l
 \cos\bigl(2\pi\lambda_l(c-r/2)+\theta_l\bigr)
 \cos\bigl(2\pi\lambda_j(c+r/2)+\theta_j\bigr).
\end{align*}
Using $\cos A\cos B=\frac12(\cos(A-B)+\cos(A+B))$, adding these terms and
integrating against the even function $\psi(r)$ gives the difference-frequency
term
\begin{equation}\label{eq:beat-difference-frequency}
  \kappa^2a_ja_l\,
  \widehat\psi\!\left(\frac{\lambda_j+\lambda_l}{2}\right)
  \cos\bigl(2\pi(\lambda_j-\lambda_l)c+\theta_j-\theta_l\bigr)
\end{equation}
in $H_\psi$.  Thus the pair $j=2n-1$, $l=2n$ produces the translation
frequency
$\lambda_{2n-1}-\lambda_{2n}=\delta_n$.  For $j=2n-1$ and $l=2n$, the coefficient of the torus character
$e^{i(\theta_j-\theta_l)}$ at translation frequency $\delta_n$ is
\[
  \frac{\kappa^2}{2}b_n^2\widehat\psi(c_n).
\]
Orthogonality of the torus characters therefore contributes an atom at
$\delta_n$ of mass
\begin{equation}\label{eq:beat-pair-atom}
  \frac{\kappa^4}{4}b_n^4|\widehat\psi(c_n)|^2
  \geq \frac{\kappa^4\eta^2}{4}b_n^4.
\end{equation}
With $\varepsilon_n:=2\delta_n$ we therefore obtain
\begin{equation}\label{eq:beat-normalized-pair-mass}
  \frac{\sigma^{(2)}_{\varphi_\psi,\varphi_\psi}
    ((-\varepsilon_n,\varepsilon_n))}
       {\Vol_1((-\varepsilon_n,\varepsilon_n))}
  \geq
  \frac{\kappa^4\eta^2}{16}\frac{b_n^4}{\delta_n}
  >\frac{\kappa^4\eta^2}{16}b_n^{-4}
  \longrightarrow\infty.
\end{equation}
Hence spectral $2$-hyperuniformity fails, and
Theorem~\ref{thm:geometric-spectral} gives $\mu\notin\HU_2$.
\end{proof}

\subsection{The Kurasov--Sarnak return-time process}
\label{subsec:kurasov-sarnak}

Fix an irrational $\alpha>0$ and consider the return-time process for a
cross-section of the Kronecker flow
\begin{equation}\label{eq:ks-kronecker-flow}
  s.\theta:=\theta+s(\alpha,1)
  \qquad(\theta\in\mathbb T^2,\ s\in\mathbb R),
\end{equation}
where $\mathbb T:=\mathbb R/\mathbb Z$.  The cross-section is the torus zero
set of the Kurasov--Sarnak polynomial
\begin{equation}\label{eq:ks-polynomial}
  P(z_1,z_2)
  :=1-\frac13z_1+\frac13z_2^2-z_1z_2^2,
\end{equation}
namely
\begin{equation}\label{eq:ks-zero-curve}
  \Sigma
  :=\bigl\{(x,y)\in\mathbb T^2:
      P(e^{2\pi i x},e^{2\pi i y})=0\bigr\}.
\end{equation}
On the universal cover, $\Sigma$ is represented by a smooth lift
$F:\mathbb R\to\mathbb R$, normalized by $F(0)=0$, and satisfying
\begin{equation}\label{eq:ks-curve-lift}
  \tan(\pi F(y))=-\frac12\tan(2\pi y),
  \qquad
  F\!\left(y+\frac12\right)=F(y)-1,
\end{equation}
and
\begin{equation}\label{eq:ks-curve-derivative}
  F'(y)=-\frac{4}{1+3\cos^2(2\pi y)},
  \qquad -4\leq F'(y)\leq-1.
\end{equation}
In particular the flow is transverse to $\Sigma$.

For $\theta\in\mathbb T^2$ define
\begin{equation}\label{eq:ks-point-process-factor}
  \Lambda_\theta
  :=\bigl\{t\in\mathbb R:
      \theta+t(\alpha,1)\in\Sigma\bigr\},
  \qquad
  p_\theta:=\sum_{t\in\Lambda_\theta}\delta_t,
\end{equation}
and let
\begin{equation}\label{eq:ks-invariant-law}
  \mu_\alpha:=(\theta\mapsto p_\theta)_*m_{\mathbb T^2}.
\end{equation}
Since
$p_{\theta+s(\alpha,1)}=(-s).p_\theta$, this law is $\mathbb R$-invariant;
it is ergodic because the Kronecker flow has no nonconstant invariant torus
character.  Choose lifts of the coordinates of $\theta$ and put
\[
  h_\theta(t):=\theta_1+\alpha t-F(\theta_2+t).
\]
Then $\Lambda_\theta=h_\theta^{-1}(\mathbb Z)$; changing the chosen lifts
only adds an integer to $h_\theta$.  Moreover
\[
  1+\alpha\leq h_\theta'(t)=\alpha-F'(\theta_2+t)\leq4+\alpha.
\]
Since $h_\theta$ is strictly increasing from $-\infty$ to $+\infty$, every
integer level is crossed exactly once.  If $t_n<t_{n+1}$ are consecutive
points of $\Lambda_\theta$, then
\[
  1=h_\theta(t_{n+1})-h_\theta(t_n)
    =h_\theta'(\zeta_n)(t_{n+1}-t_n)
\]
for some $\zeta_n\in(t_n,t_{n+1})$.  Hence
\[
  \frac1{4+\alpha}\leq t_{n+1}-t_n\leq\frac1{1+\alpha}.
\]
Thus $\Lambda_\theta$ is simple and uniformly discrete, and $\mu_\alpha$ has
local moments of every finite order.

\begin{theorem}[Kurasov--Sarnak point process]
\label{thm:ks-stealthy-not-hu2}
For every irrational $\alpha>0$,
\begin{equation}\label{eq:ks-main-conclusion}
  \mu_\alpha\text{ is $1$-stealthy},
  \qquad
  \mu_\alpha\notin\HU_2.
\end{equation}
\end{theorem}

\subsubsection{The first-order gap}

\begin{lemma}[The first Bartlett spectrum]
\label{lem:ks-first-bartlett}
There are coefficients $a_k\in\mathbb C$, indexed by
$k=(k_1,k_2)\in\mathbb Z_{\geq0}^2\setminus\{0\}$, such that
\begin{equation}\label{eq:ks-first-bartlett-explicit}
  \boldsymbol\sigma^{(1)}_{\mu_\alpha}
  =\sum_{k\in\mathbb Z_{\geq0}^2\setminus\{0\}}
     |a_k|^2
     \bigl(\delta_{k_1\alpha+k_2}
            +\delta_{-(k_1\alpha+k_2)}\bigr).
\end{equation}
In particular,
\begin{equation}\label{eq:ks-first-spectral-gap}
  \boldsymbol\sigma^{(1)}_{\mu_\alpha}
  \bigl((-\varepsilon_0,\varepsilon_0)\bigr)=0,
  \qquad
  \varepsilon_0:=\min\{\alpha,1\}.
\end{equation}
\end{lemma}

\begin{proof}
For $\theta\in\mathbb T^2$ put
\begin{equation}\label{eq:ks-phased-polynomial}
  P_\theta(z_1,z_2)
  :=P(e^{2\pi i\theta_1}z_1,e^{2\pi i\theta_2}z_2).
\end{equation}
The polynomial $P$ has no zero in the open bidisc.  If $P(z,w)=0$ and
$1+3w^2\neq0$, then
\[
  z=\frac{3+w^2}{1+3w^2},
\]
and for $|w|<1$,
\begin{equation}\label{eq:ks-bidisc-zero-free}
  |3+w^2|^2-|1+3w^2|^2=8(1-|w|^4)>0,
\end{equation}
so $|z|>1$; the exceptional case $1+3w^2=0$ does not give a zero either.
Since $P(0,0)=1$, $\log P$ is therefore holomorphic in the bidisc.  Write
\begin{equation}\label{eq:ks-log-coefficients}
  \log P(z_1,z_2)
  =\sum_{k\in\mathbb Z_{\geq0}^2\setminus\{0\}}
     c_P(k)z_1^{k_1}z_2^{k_2}.
\end{equation}
Replacing $P$ by $P_\theta$ multiplies $c_P(k)$ by
$e^{2\pi i(k_1\theta_1+k_2\theta_2)}$.  The D-stability and self-duality
verified in \cite[Section~3]{KS20} pass to the phased polynomials: $P_\theta$
and $P_{-\theta}$ are D-stable and
\[
  P_{-\theta}(z)
  =-e^{-2\pi i(\theta_1+2\theta_2)}z_1z_2^2
    P_\theta(z_1^{-1},z_2^{-1}),
\]
so they form a stable pair with multi-degree $(1,2)$.  In the
Kurasov--Sarnak summation formula, the corresponding positive-frequency
coefficient is, with our Fourier normalization,
\begin{equation}\label{eq:ks-coefficient-link}
  a_k:=-(k_1\alpha+k_2)c_P(k).
\end{equation}
Their general summation formula \cite[Equation~(29) and Theorem~1]{KS20}
therefore applies and gives
\begin{equation}\label{eq:ks-deterministic-fourier-expansion}
  \widehat{p_\theta}
  =\rho\,\delta_0
   +\sum_{k\in\mathbb Z_{\geq0}^2\setminus\{0\}}
      a_k e^{2\pi i(k_1\theta_1+k_2\theta_2)}
      \delta_{k_1\alpha+k_2}
   +\sum_{k\in\mathbb Z_{\geq0}^2\setminus\{0\}}
      \overline{a_k}e^{-2\pi i(k_1\theta_1+k_2\theta_2)}
      \delta_{-(k_1\alpha+k_2)}.
\end{equation}
For $f\in C_c^\infty(\mathbb R)$, pairing the preceding formula with
$\check f$ is absolutely convergent: the coefficients $c_P(k)$ are uniformly
bounded by \cite[Equation~(35)]{KS20}, while $\widehat f$ is rapidly
decreasing.  It gives
\begin{align*}
 S_f(p_\theta)-\mathbb E_{\mu_\alpha}S_f
 &=\sum_{k\neq0}
   a_ke^{2\pi i(k_1\theta_1+k_2\theta_2)}
     \widehat f(k_1\alpha+k_2)\\
 &\quad+
   \sum_{k\neq0}
   \overline{a_k}e^{-2\pi i(k_1\theta_1+k_2\theta_2)}
     \widehat f(-(k_1\alpha+k_2)).
\end{align*}
The nonconstant torus characters in the two sums are pairwise orthogonal.
Parseval therefore gives
\begin{equation}\label{eq:ks-first-variance-parseval}
  \Var_{\mu_\alpha}(S_f)
  =\sum_{k\neq0}|a_k|^2
    \left(
      |\widehat f(k_1\alpha+k_2)|^2
      +|\widehat f(-(k_1\alpha+k_2))|^2
    \right),
\end{equation}
which is exactly \eqref{eq:ks-first-bartlett-explicit}.  Since
$k_1\alpha+k_2\geq\min\{\alpha,1\}$ for every nonzero
$k\in\mathbb Z_{\geq0}^2$, the spectral gap
\eqref{eq:ks-first-spectral-gap} follows.
\end{proof}

\subsubsection{A pair statistic on consecutive returns}

Let $\tau(y)$ be the gap from the return at $(F(y),y)$ to the next return.
For the lift with initial point $(F(y),y)$, crossing the next integer level is
equivalent to
\[
  \alpha t-F(y+t)+F(y)=1.
\]
Thus $\tau(y)$ is characterized by
\begin{equation}\label{eq:ks-return-equation}
  F(y+\tau(y))-F(y)-\alpha\tau(y)=-1.
\end{equation}
This is the return-time description used in \cite[Lemma~6.6]{AV24}.  The
partial derivative of the left-hand side with respect to $t$ is
$F'(y+t)-\alpha\leq-1-\alpha<0$, so the solution is unique and the
implicit-function theorem makes $\tau$ smooth.  Using
$F(y+1/2)=F(y)-1$ gives
\begin{equation}\label{eq:ks-return-range}
  0<\tau(y)<\frac12,
  \qquad
  \tau\!\left(y+\frac12\right)=\tau(y).
\end{equation}
\begin{lemma}[The return time near its minimum]
\label{lem:ks-return-geometry}
The function $\tau$ has a non-degenerate minimum $\tau_*>0$.  There is a
bounded open interval $J$ with
\begin{equation}\label{eq:ks-gap-test-interval}
  \overline J\subset(\tau_*,2\tau_*)
\end{equation}
such that $\tau^{-1}(J)\cap(0,1/2)$ is the union of two intervals and $\tau$
is a diffeomorphism from each of them onto $J$.  Write
$y_-,y_+:J\to(0,1/2)$ for the inverse branches.  If
\begin{equation}\label{eq:ks-phase-functions}
  \Phi_r(y):=F(y)-ry,
  \qquad
  \Delta_r(t):=\Phi_r(y_+(t))-\Phi_r(y_-(t)),
\end{equation}
then there is $c_J>0$ such that
\begin{equation}\label{eq:ks-branch-phase-derivative}
  |\Delta_r'(t)|\geq c_J
  \qquad(t\in J,\ r>0).
\end{equation}
\end{lemma}

\begin{proof}
Differentiating \eqref{eq:ks-return-equation} gives
\begin{equation}\label{eq:ks-return-derivative}
  \tau'(y)
  =\frac{F'(y)-F'(y+\tau(y))}
         {F'(y+\tau(y))-\alpha}.
\end{equation}
At a critical point,
$F'(y)=F'(y+\tau(y))$.  From \eqref{eq:ks-curve-derivative}, this is equivalent
to
\[
  y+\tau(y)\equiv \pm y\pmod{1/2}.
\]
The plus sign would give $\tau(y)\in\frac12\mathbb Z$, contrary to
$0<\tau(y)<1/2$.  If $y\equiv0\pmod{1/2}$, the minus sign gives the same
contradiction.  Thus we may take $0<y<1/2$, and the minus sign gives exactly
\begin{equation}\label{eq:ks-critical-point-relations}
  2y+\tau(y)=\frac12
  \qquad\text{or}\qquad
  2y+\tau(y)=1.
\end{equation}
The normalization $F(0)=0$ and the tangent equation give $F(-y)=-F(y)$.
If the first relation holds, then
$F(y+\tau)=F(1/2-y)=-F(y)-1$, and
\eqref{eq:ks-return-equation} becomes
\begin{equation}\label{eq:ks-minimum-critical-equation}
  \Phi_\alpha(y)=-\frac{\alpha}{4},
  \qquad 0<y<\frac14.
\end{equation}
Since $\Phi_\alpha'=F'-\alpha<0$, there is at most one solution.  Moreover
$\Phi_\alpha(0)=0$, while $F'\leq-1$ gives
$\Phi_\alpha(1/4)\leq-(1+\alpha)/4<-\alpha/4$, so there is exactly one;
denote it by $y_0$.

For the second relation, $y+\tau=1-y$ and
$F(1-y)=F(-y+1)=-F(y)-2$.  Equation
\eqref{eq:ks-return-equation} then becomes
\begin{equation}\label{eq:ks-maximum-critical-equation}
  \Phi_\alpha(y)=-\frac{1+\alpha}{2},
  \qquad \frac14<y<\frac12.
\end{equation}
Again there is at most one solution.  The tangent relation and continuity of
our lift give $F(1/4)=-1/2$ and $F(1/2)=-1$, so
\[
  \Phi_\alpha(1/4)=-\frac12-\frac\alpha4
  >-\frac{1+\alpha}{2}
  >-1-\frac\alpha2=\Phi_\alpha(1/2).
\]
Hence \eqref{eq:ks-maximum-critical-equation} has exactly one solution.

One has
\begin{equation}\label{eq:ks-second-derivative-F}
  F''(y)
  =-\frac{48\pi\cos(2\pi y)\sin(2\pi y)}
          {(1+3\cos^2(2\pi y))^2}.
\end{equation}
At $y_0\in(0,1/4)$, $F''(y_0)<0$ and
$F''(y_0+\tau(y_0))=-F''(y_0)$.  Differentiating
\eqref{eq:ks-return-equation} once more at $y_0$ gives
\begin{equation}\label{eq:ks-return-second-derivative}
  \tau''(y_0)
  =-\frac{2F''(y_0)}{\alpha-F'(y_0+\tau(y_0))}>0.
\end{equation}
At the unique critical point in $(1/4,1/2)$, the signs of the two second
derivatives are reversed and the same calculation gives $\tau''<0$.  Thus
$\tau$ has exactly one non-degenerate minimum and one non-degenerate maximum
on $\mathbb R/(\frac12\mathbb Z)$.  Put
\[
  \tau_*:=\tau(y_0)>0,
  \qquad
  \tau^*:=\max\tau.
\]
The function $\tau$ is therefore strictly monotone on the two arcs between
these critical points.  We may choose $J$ so close to $\tau_*$ that
\[
  \overline J\subset(\tau_*,\min\{2\tau_*,\tau^*\}).
\]
Then $\tau^{-1}(J)\cap(0,1/2)$ consists of exactly two intervals, giving the
inverse branches $y_-$ and $y_+$, with
$\tau'(y_-(t))<0<\tau'(y_+(t))$.

Differentiating \eqref{eq:ks-phase-functions} along these branches gives
\begin{equation}\label{eq:ks-phase-derivative-direct}
  \Delta_r'(t)
  =\frac{F'(y_+(t))-r}{\tau'(y_+(t))}
   -\frac{F'(y_-(t))-r}{\tau'(y_-(t))}.
\end{equation}
For $r>0$ the first quotient is negative and the second is positive, so the
two terms on the right-hand side of
\eqref{eq:ks-phase-derivative-direct} have the same negative sign.  Put
\[
  M_J:=\max_{\sigma\in\{-,+\}}
       \sup_{t\in J}|\tau'(y_\sigma(t))|<\infty.
\]
Since $r-F'(y_\sigma(t))\geq1$, each quotient has absolute value at least
$M_J^{-1}$.  Therefore
\[
  |\Delta_r'(t)|\geq \frac2{M_J}
  \qquad(t\in J,r>0),
\]
which proves \eqref{eq:ks-branch-phase-derivative}.
\end{proof}

Let $\psi\in C_0(J)$, extend it by zero outside $J$ and evenly to a compactly
supported continuous function $\widetilde\psi$ on $\mathbb R$, and define
$\varphi_\psi\in C_c(\Sh_2(\mathbb R))$ by
\begin{equation}\label{eq:ks-pair-shape-test}
  \varphi_\psi\bigl([(-t/2,t/2)]\bigr)=\widetilde\psi(t).
\end{equation}
Every gap is at least $\tau_*$.  If $\lambda_m-\lambda_n$ with $m\geq n+2$,
then
\[
  \lambda_m-\lambda_n
  =\sum_{r=n}^{m-1}(\lambda_{r+1}-\lambda_r)
  \geq2\tau_*.
\]
Since $J\subset(\tau_*,2\tau_*)$, a pair whose distance lies in $J$ is
therefore consecutive.  If
\begin{equation}\label{eq:ks-ordered-points-gaps}
  \cdots<\lambda_{-1}<\lambda_0<\lambda_1<\cdots,
  \qquad
  g_n:=\lambda_{n+1}-\lambda_n,
\end{equation}
then
\begin{equation}\label{eq:ks-nearest-neighbour-pair-statistic}
  T_2[f,\varphi_\psi](p_\theta)
  =2\sum_{n\in\mathbb Z}
       \psi(g_n)f\!\left(\lambda_n+\frac{g_n}{2}\right).
\end{equation}

\begin{lemma}[Transversal integration formula]
\label{lem:ks-flow-box}
For every integrable function $G$ on $\mathbb T^2\times\mathbb R$ whose
second variable has compact support,
\begin{equation}\label{eq:ks-flow-box-formula}
  \int_{\mathbb T^2}
    \sum_{t:\,\theta+t(\alpha,1)\in\Sigma}
      G(\theta,t)\,d\theta
  =\int_0^1\int_{\mathbb R}
      G\bigl((F(y),y)-t(\alpha,1),t\bigr)
      \bigl(\alpha-F'(y)\bigr)
      \,d\Vol_1(t)\,d\Vol_1(y).
\end{equation}
\end{lemma}

\begin{proof}
Consider
\[
  \Psi:[0,1)\times\mathbb R\longrightarrow\mathbb T^2,
  \qquad
  \Psi(y,t):=(F(y)-\alpha t,\,y-t)\pmod{\mathbb Z^2}.
\]
The derivative matrix of $\Psi$ is
\[
  \begin{pmatrix}F'(y)&-\alpha\\1&-1\end{pmatrix},
  \qquad
  |\det D\Psi|=\alpha-F'(y)>0.
\]
Since $F(y+1)=F(y)-2$, the map
$y\mapsto(F(y),y)\pmod{\mathbb Z^2}$ is a bijection from $[0,1)$ onto
$\Sigma$.  Thus
\[
  \Psi(y,t)=\theta
  \quad\Longleftrightarrow\quad
  (F(y),y)=\theta+t(\alpha,1)\in\Sigma,
\]
so the fibre $\Psi^{-1}(\theta)$ is indexed by the return times in the
left-hand side of \eqref{eq:ks-flow-box-formula}.  Applying the area formula
to $\Psi$ and summing over these fibres gives
\eqref{eq:ks-flow-box-formula}.
\end{proof}

For $(p,q)\in\mathbb Z^2$ put
\begin{equation}\label{eq:ks-physical-character-frequency}
  \xi_{p,q}:=q-p\alpha.
\end{equation}
For the torus character
$\chi_{p,q}(x,y):=e^{2\pi i(px-qy)}$ one has
\[
  \chi_{p,q}(\theta+s(\alpha,1))
  =e^{-2\pi i\xi_{p,q}s}\chi_{p,q}(\theta).
\]

\begin{lemma}[Pair coefficients]
\label{lem:ks-pair-coefficients}
For $\psi\in C_0(J)$ define
\begin{equation}\label{eq:ks-pair-functional}
  \mathcal L_{p,q}(\psi)
  :=2\int_0^1
       \psi(\tau(y))e^{\pi i\xi_{p,q}\tau(y)}
       e^{-2\pi i(pF(y)-qy)}
       \bigl(\alpha-F'(y)\bigr)\,d\Vol_1(y).
\end{equation}
Then
\begin{equation}\label{eq:ks-pair-bartlett-spectrum}
  \sigma^{(2)}_{\varphi_\psi,\varphi_\psi}
  =\sum_{(p,q)\in\mathbb Z^2\setminus\{(0,0)\}}
      |\mathcal L_{p,q}(\psi)|^2\,\delta_{\xi_{p,q}}.
\end{equation}
\end{lemma}

\begin{proof}
Apply Lemma~\ref{lem:ks-flow-box} to
\eqref{eq:ks-nearest-neighbour-pair-statistic} and take its
$\chi_{p,q}$-Fourier coefficient.  Since
\begin{equation}\label{eq:ks-character-along-flow}
  \overline{\chi_{p,q}((F(y),y)-t(\alpha,1))}
  =e^{-2\pi i(pF(y)-qy)}e^{-2\pi i\xi_{p,q}t},
\end{equation}
the coefficient equals
\begin{align*}
 2\int_0^1\int_{\mathbb R}
 &\psi(\tau(y))f\!\left(t+\frac{\tau(y)}2\right)
 e^{-2\pi i(pF(y)-qy)}e^{-2\pi i\xi_{p,q}t}\\
 &\hspace{35mm}\times(\alpha-F'(y))\,dt\,dy.
\end{align*}
With $s=t+\tau(y)/2$,
\[
  \int_{\mathbb R}f\!\left(t+\frac{\tau(y)}2\right)
       e^{-2\pi i\xi_{p,q}t}\,dt
  =e^{\pi i\xi_{p,q}\tau(y)}\widehat f(\xi_{p,q}).
\]
Substitution gives
\begin{equation}\label{eq:ks-pair-fourier-coefficient}
  \int_{\mathbb T^2}
     T_2[f,\varphi_\psi](p_\theta)
     \overline{\chi_{p,q}(\theta)}\,d\theta
  =\widehat f(\xi_{p,q})\,\mathcal L_{p,q}(\psi).
\end{equation}
Parseval on $\mathbb T^2$, after removing the constant character, gives
\begin{equation}\label{eq:ks-pair-covariance-parseval}
  \Var_{\mu_\alpha}\bigl(T_2[f,\varphi_\psi]\bigr)
  =\sum_{(p,q)\neq(0,0)}
    |\widehat f(\xi_{p,q})|^2
    |\mathcal L_{p,q}(\psi)|^2.
\end{equation}
Since $\alpha$ is irrational, the frequencies $\xi_{p,q}$ are distinct.
Comparison with the Bartlett covariance identity proves
\eqref{eq:ks-pair-bartlett-spectrum}.
\end{proof}

\begin{proof}[Proof of Theorem~\ref{thm:ks-stealthy-not-hu2}]
Lemma~\ref{lem:ks-first-bartlett} proves $1$-stealthiness.  For the second
assertion, Lemma~\ref{lem:ks-pair-coefficients} reduces the problem to
finding one $\psi\in C_0(J)$ and a sequence $(p_j,q_j)$ for which
\begin{equation}\label{eq:ks-target-small-frequency-atoms}
  \delta_j:=|q_j-p_j\alpha|\longrightarrow0,
  \qquad
  \limsup_{j\to\infty}
  \frac{|\mathcal L_{p_j,q_j}(\psi)|^2}{\delta_j}=\infty.
\end{equation}
Such a sequence gives atoms approaching the origin whose normalized mass
violates spectral $2$-hyperuniformity.

Split the integral in \eqref{eq:ks-pair-functional} at $1/2$ and substitute
$y\mapsto y+1/2$ in the second half.  The functions $\tau$ and $F'$ are
$1/2$-periodic, while
\[
 e^{-2\pi i(pF(y+1/2)-q(y+1/2))}
 =(-1)^q e^{-2\pi i(pF(y)-qy)}.
\]
Therefore
\[
  \mathcal L_{p,q}(\psi)
  =2(1+(-1)^q)\int_0^{1/2}
       \psi(\tau(y))e^{\pi i\xi_{p,q}\tau(y)}
       e^{-2\pi i(pF(y)-qy)}
       (\alpha-F'(y))\,dy.
\]
In particular,
\begin{equation}\label{eq:ks-parity}
  \mathcal L_{p,q}=0\quad\text{if $q$ is odd}.
\end{equation}
For even $q$ and $p\neq0$, this becomes
\begin{equation}\label{eq:ks-half-period-functional}
  \mathcal L_{p,q}(\psi)
  =4\int_0^{1/2}
       \psi(\tau(y))e^{\pi i\xi_{p,q}\tau(y)}
       e^{-2\pi ip\Phi_{q/p}(y)}
       \bigl(\alpha-F'(y)\bigr)\,d\Vol_1(y).
\end{equation}
Because $\psi$ is supported in $J$, only the two branches $y_\pm$ contribute.
Changing variables $t=\tau(y)$ gives
\begin{equation}\label{eq:ks-functional-kernel}
  \mathcal L_{p,q}(\psi)
  =4\int_J \psi(t)K_{p,q}(t)\,d\Vol_1(t),
\end{equation}
where
\begin{align}
  K_{p,q}(t)
  &:=e^{\pi i\xi_{p,q}t}
      \sum_{\sigma\in\{-,+\}}
        A_\sigma(t)
        e^{-2\pi ip\Phi_{q/p}(y_\sigma(t))},
  \label{eq:ks-kernel-definition}\\
  A_\sigma(t)
  &:=\frac{\alpha-F'(y_\sigma(t))}
             {|\tau'(y_\sigma(t))|}.
  \label{eq:ks-kernel-amplitudes}
\end{align}
Since $\overline J\subset(\tau_*,\tau^*)$, the branches $y_\pm$ extend
smoothly to $\overline J$ and
\[
  \inf_{t\in\overline J}|\tau'(y_\pm(t))|>0.
\]
Hence $A_\pm$ and their first derivatives are bounded on $\overline J$.
Consequently
\begin{equation}\label{eq:ks-kernel-square}
  |K_{p,q}(t)|^2
  =A_-(t)^2+A_+(t)^2
   +2A_-(t)A_+(t)
      \cos\bigl(2\pi p\Delta_{q/p}(t)\bigr).
\end{equation}

Let $q_j'/p_j'$ be continued-fraction convergents to $\alpha$ and set
\begin{equation}\label{eq:ks-even-convergents}
  p_j:=2p_j',
  \qquad
  q_j:=2q_j'.
\end{equation}
Then $q_j$ is even, $p_j\to\infty$, $q_j/p_j\to\alpha$, and
\begin{equation}\label{eq:ks-linear-form-convergence}
  \delta_j=|q_j-p_j\alpha|
  =2|q_j'-p_j'\alpha|\longrightarrow0.
\end{equation}
For all large $j$, Lemma~\ref{lem:ks-return-geometry} gives
$|\Delta_{q_j/p_j}'|\geq c_J$ on $J$.  Put $B=A_-A_+$.  On
$\overline J$, $B$ is $C^1$, while $\Delta_r'$ and $\Delta_r''$ are uniformly
bounded for $r$ near $\alpha$.  Integration by parts gives
\begin{align*}
 \int_J B(t)e^{2\pi ip_j\Delta_{q_j/p_j}(t)}\,dt
 &=\left[
   \frac{B(t)e^{2\pi ip_j\Delta_{q_j/p_j}(t)}}
        {2\pi ip_j\Delta_{q_j/p_j}'(t)}
   \right]_{\partial J}\\
 &\quad-\frac1{2\pi ip_j}
   \int_J
   \left(\frac{B}{\Delta_{q_j/p_j}'}\right)'(t)
   e^{2\pi ip_j\Delta_{q_j/p_j}(t)}\,dt.
\end{align*}
Both terms are $O(p_j^{-1})$, uniformly in $j$.  Hence
\begin{equation}\label{eq:ks-oscillatory-cross-term}
  \int_J
    A_-(t)A_+(t)
      e^{2\pi ip_j\Delta_{q_j/p_j}(t)}\,d\Vol_1(t)
  =O(p_j^{-1})\longrightarrow0.
\end{equation}
It follows from \eqref{eq:ks-kernel-square} that
\begin{equation}\label{eq:ks-kernel-l2-lower-bound}
  \liminf_{j\to\infty}
  \int_J|K_{p_j,q_j}(t)|^2\,d\Vol_1(t)
  \geq
  \int_J(A_-^2+A_+^2)\,d\Vol_1>0.
\end{equation}
The amplitudes $A_\pm$ are bounded on $\overline J$, hence
$\|K_{p_j,q_j}\|_\infty\leq C_J$ uniformly in $j$.  Choose $c_1>0$ so that
\[
  \int_J|K_{p_j,q_j}|^2\geq c_1
\]
for all large $j$, as permitted by
\eqref{eq:ks-kernel-l2-lower-bound}.  Then
\[
  c_1
  \leq \int_J|K_{p_j,q_j}|^2
  \leq C_J\int_J|K_{p_j,q_j}|.
\]
By \eqref{eq:ks-functional-kernel},
\begin{equation}\label{eq:ks-functional-norm-lower-bound}
  \|\mathcal L_{p_j,q_j}\|_{C_0(J)^*}
  =4\int_J|K_{p_j,q_j}(t)|\,d\Vol_1(t)
  \geq c_0>0,
  \qquad c_0:=\frac{4c_1}{C_J},
\end{equation}
for all large $j$.

Define
\begin{equation}\label{eq:ks-rescaled-functionals}
  \mathcal T_j
  :=\delta_j^{-1/2}\mathcal L_{p_j,q_j}
  \in C_0(J)^*.
\end{equation}
Then $\|\mathcal T_j\|\to\infty$.  By the Uniform Boundedness Principle
there is a single $\psi\in C_0(J)$ such that
\begin{equation}\label{eq:ks-uniform-boundedness-output}
  \limsup_{j\to\infty}
  \frac{|\mathcal L_{p_j,q_j}(\psi)|^2}{\delta_j}=\infty.
\end{equation}
Write $\psi=\psi_1+i\psi_2$ with real $\psi_1,\psi_2$.  Since
\[
  |\mathcal T_j(\psi)|
  \leq |\mathcal T_j(\psi_1)|+|\mathcal T_j(\psi_2)|,
\]
at least one of the two real functions has unbounded values along a
subsequence.  Replacing $\psi$ by that function and passing to the subsequence,
we may take $\psi$ real.

For this fixed $\psi$, Lemma~\ref{lem:ks-pair-coefficients} gives an atom of
mass $|\mathcal L_{p_j,q_j}(\psi)|^2$ at the frequency
$\xi_{p_j,q_j}$, whose absolute value is $\delta_j$.  With
$\varepsilon_j:=2\delta_j$,
\begin{equation}\label{eq:ks-normalized-pair-spectrum}
  \frac{\sigma^{(2)}_{\varphi_\psi,\varphi_\psi}
       ((-\varepsilon_j,\varepsilon_j))}
       {\Vol_1((-\varepsilon_j,\varepsilon_j))}
  \geq
  \frac{|\mathcal L_{p_j,q_j}(\psi)|^2}{4\delta_j},
\end{equation}
and the right-hand side is unbounded by
\eqref{eq:ks-uniform-boundedness-output}.  Hence spectral
$2$-hyperuniformity fails.  Theorem~\ref{thm:geometric-spectral} therefore
gives $\mu_\alpha\notin\HU_2$.
\end{proof}

\section{All-order stealthiness and lattice sphere packing}
\label{sec:all-order-stealthiness}

A locally finite set $C\subset\mathbb R^d$ is $r$-\emph{uniformly discrete}
if $|x-y|\geq r$ for all distinct $x,y\in C$.  It is \emph{periodic} if its
full translation group
\[
  \Gamma_C:=\{t\in\mathbb R^d:C+t=C\}
\]
is a full-rank lattice in $\mathbb R^d$.  A Radon measure $p$ on
$\mathbb R^d$ is \emph{nonperiodic} if
$\{t\in\mathbb R^d:t.p=p\}=\{0\}$.

Whenever the Bartlett spectra are defined at every finite order, set
\begin{equation}\label{eq:all-order-stealth-radius}
  \operatorname{st}_\infty(\mu)
  :=\sup\bigl\{\varepsilon>0:
  \boldsymbol\sigma^{(k)}(B_\varepsilon^*)=0
  \text{ for every }k\geq1\bigr\}.
\end{equation}
Set $\operatorname{st}_\infty(\mu)=0$ if the set in
\eqref{eq:all-order-stealth-radius} is empty.  The measure $\mu$ has a
\emph{common all-order stealth gap} if there exists
$\varepsilon_0>0$ such that
\begin{equation}\label{eq:common-all-order-gap}
  \boldsymbol\sigma^{(k)}(B_{\varepsilon_0}^*)=0
  \qquad\text{for every }k\geq1.
\end{equation}
For a periodic configuration $C$, let $\mu_C$ denote Haar probability
measure on the translation orbit of $\delta_C$, and write
$\operatorname{st}_\infty(C):=\operatorname{st}_\infty(\mu_C)$.

\begin{proposition}[A common all-order gap forces periodicity]
\label{prop:common-gap-periodic}
Let $\mu$ be an ergodic $\mathbb R^d$-invariant simple point process of positive
intensity.  Suppose that there is $r>0$ such that $\mu$ is supported on
$r$-uniformly discrete configurations, and that \eqref{eq:common-all-order-gap}
holds.  Then there is a periodic locally finite set $C\subset\mathbb R^d$ such
that $\mu$ is supported on the translation orbit of $\delta_C$ and coincides
there with Haar probability measure.
\end{proposition}

\begin{proof}
Let $\mathscr X_r$ be the space of $r$-uniformly discrete counting measures on
$\mathbb R^d$.  The balls of radius $r/2$ centred at the points of a
configuration in $B_R$ are disjoint and contained in $B_{R+r/2}$, so
\begin{equation}\label{eq:common-gap-packing-bound}
  p(B_R)\leq
  \frac{\Vol_d(B_{R+r/2})}{\Vol_d(B_{r/2})}
  =\left(1+\frac{2R}{r}\right)^d
  \qquad(p\in\mathscr X_r).
\end{equation}
The family $\mathscr X_r$ is vaguely closed.  Suppose that
$p_j\to p$ vaguely.  If the support of $p$ contained distinct points $x,y$
with $|x-y|<r$, one could choose disjoint neighbourhoods of $x$ and $y$ such
that every pair of points, one in each neighbourhood, has distance less than
$r$; both neighbourhoods would contain points of $p_j$ for all large $j$,
a contradiction.  To verify that $p$ is a simple counting measure, fix an
open ball $B$ of radius
$\rho<r/2$ with compact closure.  Each restriction $p_j|_B$ is either zero
or a single Dirac mass.  If $p(B)>0$, choose $0\leq f\in C_c(B)$ with
$p(f)>0$.  Then, for all large $j$, the unique point $x_j$ of $p_j$ meeting
$\operatorname{supp}f$ exists; after passing to a subsequence,
$x_j\to x\in\operatorname{supp}f\subset B$.  For every $g\in C_c(B)$ we then have
$p(g)=g(x)$, so $p|_B$ is a single Dirac mass (and if $p(B)=0$ it is zero).
A countable cover by such balls shows that $p$ is a simple counting measure,
and the preceding separation argument shows that its support is
$r$-separated.  Thus $p\in\mathscr X_r$.

Equation~\eqref{eq:common-gap-packing-bound} gives
$\sup_{p\in\mathscr X_r}p(K)<\infty$ for every compact
$K\subset\mathbb R^d$.  By the standard vague compactness criterion for
Radon measures, $\mathscr X_r$ is therefore relatively compact; since it is
vaguely closed, it is compact.  The action
$\mathbb R^d\times\mathscr X_r\to\mathscr X_r$, $(t,p)\mapsto t.p$, is
jointly continuous: if $t_j\to t$ and $p_j\to p$ vaguely, then for
$f\in C_c(\mathbb R^d)$ all the translates involved have support in one
fixed compact set $K$, and
\[
  |p_j(\tau_{t_j}f)-p_j(\tau_t f)|
  \leq \|\tau_{t_j}f-\tau_t f\|_\infty\,p_j(K)\longrightarrow0,
\]
while $p_j(\tau_t f)\to p(\tau_t f)$.
Hence $X:=\operatorname{supp}\mu\subset\mathscr X_r$ is compact metrizable,
and the same bound gives local moments of every order.

Let $U_t$ denote the Koopman representation on $L^2(X,\mu)$, let $E$ be its
spectral resolution on $(\mathbb R^d)^*$, and put
\[
  \mathcal H_{\varepsilon_0}
  :=E\bigl((\mathbb R^d)^*\setminus B_{\varepsilon_0}^*\bigr)L^2(X,\mu).
\]
For $f\in C_c^\infty(\mathbb R^d;\mathbb C)$ and
$\varphi\in\mathcal C_k$, translation covariance and
\eqref{eq:bartlett-defining-identity} give
\begin{equation}\label{eq:common-gap-koopman-spectrum}
  \operatorname{Cov}_\mu\!\left(
    U_tT_k[f,\varphi],T_k[f,\varphi]
  \right)
  =\int_{(\mathbb R^d)^*}e^{2\pi i\xi(t)}|\widehat f(\xi)|^2
    \,d\sigma^{(k)}_{\varphi,\varphi}(\xi).
\end{equation}
Thus the Koopman spectral measure of the centred statistic
$T_k[f,\varphi]-\int_XT_k[f,\varphi] \,d\mu$ is
$|\widehat f|^2\sigma^{(k)}_{\varphi,\varphi}$.  Therefore
\begin{equation}\label{eq:common-gap-pattern-subspace}
  T_k[f,\varphi]-\int_XT_k[f,\varphi] \,d\mu
  \in\mathcal H_{\varepsilon_0}
  \qquad(k\geq1).
\end{equation}

Let $f_1,\ldots,f_k\in C_c^\infty(\mathbb R^d)$ and set
\[
  F(p):=p(f_1)\cdots p(f_k).
\]
Put
\[
  \widetilde H(x_1,\ldots,x_k)
  :=\frac1{k!}\sum_{\pi\in S_k}\prod_{i=1}^k f_i(x_{\pi(i)}),
  \qquad
  \bar x:=\frac{x_1+\cdots+x_k}{k}.
\]
The function $\widetilde H$ is symmetric, continuous and compactly supported,
so there is a unique
$H\in C_c(\mathbb R^d\times\Sh_k(\mathbb R^d))$ such that
\[
  \widetilde H(x_1,\ldots,x_k)
  =H\!\left(\bar x,
    \left[(x_1-\bar x,\ldots,x_k-\bar x)\right]\right).
\]
Since $p^{\otimes k}$ is permutation invariant,
$\int\widetilde H\,dp^{\otimes k}=F(p)$.  Hence
\begin{equation}\label{eq:common-gap-polynomial-as-pattern}
  F(p)=\int_{(\mathbb R^d)^k}
  H\!\left(\bar x,
    \left[(x_1-\bar x,\ldots,x_k-\bar x)\right]\right)
  \,dp(x_1)\cdots dp(x_k).
\end{equation}
Choose compact sets $K_c,K_w$ whose interiors contain the projections of
$\operatorname{supp}H$.  Stone--Weierstrass on $K_c\times K_w$, followed by
fixed cutoffs and uniform approximation of the $K_c$-factors by functions in
$C_c^\infty(\mathbb R^d)$, gives
\[
  H_N(c,w)=\sum_{j=1}^{m_N}g_{N,j}(c)\varphi_{N,j}(w),
  \qquad
  g_{N,j}\in C_c^\infty(\mathbb R^d),\quad
  \varphi_{N,j}\in\mathcal C_k,
\]
with all supports contained in one fixed compact set $K$ and
$\|H_N-H\|_\infty\to0$.

For a centred representative $w=[(y_1,\ldots,y_k)]$, the function
$w\mapsto\max_i|y_i|$ is well defined and continuous on shape space.  Hence
there is $R_K<\infty$ such that every tuple whose barycentre--shape coordinate
lies in $K$ is contained in $B_{R_K}^k$.  Put
\[
  N_K:=\left(1+\frac{2R_K}{r}\right)^d.
\]
By \eqref{eq:common-gap-packing-bound}, every $p\in X$ has at most $N_K$
points in $B_{R_K}$ and therefore at most $N_K^k$ ordered $k$-tuples
contributing to \eqref{eq:common-gap-polynomial-as-pattern}.  Consequently
\begin{equation}\label{eq:common-gap-polynomial-approximation}
  \sup_{p\in X}
  \left|
    F(p)-\sum_{j=1}^{m_N}T_k[g_{N,j},\varphi_{N,j}](p)
  \right|
  \leq N_K^k\|H-H_N\|_\infty
  \longrightarrow0.
\end{equation}
Let
$F_N:=\sum_{j=1}^{m_N}T_k[g_{N,j},\varphi_{N,j}]$.  Then
\[
  \left\|(F_N-\textstyle\int_XF_N\,d\mu)
       -(F-\textstyle\int_XF\,d\mu)\right\|_2
  \leq2\|F_N-F\|_\infty\longrightarrow0.
\]
By \eqref{eq:common-gap-pattern-subspace} and closedness of
$\mathcal H_{\varepsilon_0}$,
\begin{equation}\label{eq:common-gap-polynomial-subspace}
  F-\int_XF\,d\mu\in\mathcal H_{\varepsilon_0}.
\end{equation}

The unital self-adjoint algebra generated by the functions
$p\mapsto p(f)$, $f\in C_c^\infty(\mathbb R^d)$, separates points of $X$,
because $C_c^\infty(\mathbb R^d)$ determines a locally finite measure.
Stone--Weierstrass therefore makes this algebra uniformly dense in $C(X)$,
and hence dense in $L^2(X,\mu)$.  Equation
\eqref{eq:common-gap-polynomial-subspace} gives
\begin{equation}\label{eq:common-gap-full-koopman-gap}
  L_0^2(X,\mu)
  \subset
  \mathcal H_{\varepsilon_0}.
\end{equation}

Let $\beta_s$ be the Gaussian probability measure on $\mathbb R^d$ with
density
\[
  (4\pi s)^{-d/2}e^{-|t|^2/(4s)},
\]
so that
$\widehat\beta_s(\xi)=e^{-4\pi^2s|\xi|^2}$.  The averaging operator
\[
  P_s:=\int_{\mathbb R^d}U_t\,d\beta_s(t)
\]
has spectral multiplier $\widehat\beta_s$.  Hence
\begin{equation}\label{eq:common-gap-gaussian-contraction}
  \|P_s|_{L_0^2(X,\mu)}\|
  \leq e^{-4\pi^2s\varepsilon_0^2}<1
\end{equation}
by \eqref{eq:common-gap-full-koopman-gap}.  Thus the action has a spectral gap
at the trivial representation.  If the action were properly ergodic, meaning
ergodic but not essentially transitive, this would contradict
\cite[Remark~11.2(2)]{Nev06}, since $\mathbb R^d$ is amenable.  The action is
therefore essentially transitive: there is a configuration $p$ whose
translation orbit has full $\mu$-measure.

Let
\[
  \Gamma:=\{t\in\mathbb R^d:t.p=p\}
\]
be its stabilizer.  Positive intensity excludes the empty configuration, so
$p$ is nonempty.  Uniform discreteness gives
$\Gamma\cap\{t:|t|<r\}=\{0\}$: if $0\neq t\in\Gamma$ and $|t|<r$, then
for any $x\in\operatorname{supp}p$ the distinct points $x$ and $x+t$ both
belong to $\operatorname{supp}p$.  Hence $\Gamma$ is discrete.  Since the translation action on $X$ is
continuous, $\Gamma$ is also closed, and the orbit map induces a continuous
injection
\[
  \iota:\mathbb R^d/\Gamma\longrightarrow X.
\]
The quotient $\mathbb R^d/\Gamma$ is Polish and $X$ is compact metrizable,
so the Lusin--Souslin theorem implies that
$\iota(\mathbb R^d/\Gamma)=\mathbb R^d.p$ is Borel in $X$ and that
$\iota$ is a Borel isomorphism onto its image.  Transporting $\mu$ through
this isomorphism gives a translation-invariant Borel probability measure on
the locally compact group $\mathbb R^d/\Gamma$.  This measure is Radon and,
by uniqueness of Haar measure, is its normalized Haar measure; in particular
$\mathbb R^d/\Gamma$ is compact.  Thus $\Gamma$ is cocompact, and a
discrete cocompact subgroup of $\mathbb R^d$ is a full-rank lattice.  Hence
$p=\delta_C$ for a periodic set $C$, and $\mu$ is Haar probability measure
on its translation orbit.
\end{proof}

\begin{proposition}[Higher-order spectra recover the period lattice]
\label{prop:periodic-full-reciprocal-spectrum}
Let $C\subset\mathbb R^d$ be a nonempty uniformly discrete periodic set, let
\[
  \Gamma_C:=\{t\in\mathbb R^d:C+t=C\}
\]
be its full translation lattice, and let $\mu_C$ be Haar probability measure
on its translation orbit.  Then
\begin{equation}\label{eq:periodic-all-order-support}
  \bigcup_{k\geq1}\ \bigcup_{\varphi\in\mathcal C_k}
  \operatorname{supp}\sigma^{(k)}_{\varphi,\varphi}
  =\Gamma_C^*\setminus\{0\}.
\end{equation}
Consequently,
\begin{equation}\label{eq:all-order-stealth-period-lattice}
  \operatorname{st}_\infty(C)=\lambda_1(\Gamma_C^*)
  :=\min_{0\neq\xi\in\Gamma_C^*}|\xi|.
\end{equation}
\end{proposition}

\begin{proof}
The orbit map identifies the translation orbit with
$\mathbb R^d/\Gamma_C$; write $X_C:=\operatorname{supp}\mu_C$ for this compact orbit.  Hence every centred pattern statistic has Koopman
spectrum in $\Gamma_C^*\setminus\{0\}$.  Since $C$ is uniformly discrete and its translation orbit is compact, the
packing and local-polynomial approximation estimates from the proof of
Proposition~\ref{prop:common-gap-periodic} apply on this orbit.  By
\eqref{eq:common-gap-koopman-spectrum}, varying $f$ gives
\[
  \operatorname{supp}\sigma^{(k)}_{\varphi,\varphi}
  \subset\Gamma_C^*\setminus\{0\}
  \qquad(k\geq1,\ \varphi\in\mathcal C_k).
\]

Fix $0\neq\xi\in\Gamma_C^*$ and suppose that
$\sigma^{(k)}_{\varphi,\varphi}(\{\xi\})=0$ for every $k$ and every
$\varphi\in\mathcal C_k$.  Equation
\eqref{eq:common-gap-koopman-spectrum} then shows that every pattern statistic
$T_k[f,\varphi]$ has zero $\xi$-Fourier coefficient on
$\mathbb R^d/\Gamma_C$.  The approximation
\eqref{eq:common-gap-polynomial-approximation}, applied to this compact orbit,
shows the same for every local polynomial
\[
  p\longmapsto p(f_1)\cdots p(f_k).
\]
Since $\Gamma_C$ is the full stabilizer, the orbit map
$\mathbb R^d/\Gamma_C\to X_C$ is injective.  Local linear statistics separate
configurations, so the algebra $\mathcal A$ of local polynomials is uniformly
dense in $C(X_C)$ by Stone--Weierstrass.  Under the assumption above,
\[
  \langle F,\chi_\xi\rangle_{L^2(\mu_C)}=0
  \qquad(F\in\mathcal A),
  \qquad
  \chi_\xi(u):=e^{2\pi i\xi(u)}.
\]
Density then gives
$\|\chi_\xi\|_2^2=\langle\chi_\xi,\chi_\xi\rangle=0$, a contradiction.
Thus for some $k$ and $\varphi$,
$\sigma^{(k)}_{\varphi,\varphi}(\{\xi\})>0$.  Since the first inclusion has
already placed every coefficient measure on the discrete set
$\Gamma_C^*\setminus\{0\}$, its support there is precisely the set of its
positive atoms, and \eqref{eq:periodic-all-order-support} follows.

For every $\varepsilon<\lambda_1(\Gamma_C^*)$, the left-hand side of
\eqref{eq:periodic-all-order-support} misses $B_\varepsilon^*$, whereas the
closed ball $B_{\lambda_1(\Gamma_C^*)}^*$ contains a shortest nonzero vector
of $\Gamma_C^*$ at which some finite-order spectrum has a positive atom.
Thus the admissible radii in \eqref{eq:all-order-stealth-radius} are precisely
those below $\lambda_1(\Gamma_C^*)$, and taking the supremum gives
\eqref{eq:all-order-stealth-period-lattice}.
\end{proof}

\begin{corollary}[All-order stealth and lattice sphere packing]
\label{cor:periodic-all-order-hermite}
Let $C\subset\mathbb R^d$ be a periodic configuration of intensity one, and
let $m$ be the number of points of $C$ in a fundamental domain of its full
translation lattice $\Gamma_C$.  Let
\[
  \gamma_d:=\sup_{L}
  \frac{\lambda_1(L)^2}{\operatorname{covol}(L)^{2/d}}
\]
be the Hermite constant, where the supremum is over full-rank lattices
$L\subset\mathbb R^d$.  Then
\begin{equation}\label{eq:all-order-hermite-bound}
  \operatorname{st}_\infty(C)
  \leq \sqrt{\gamma_d}\,m^{-1/d}
  \leq \sqrt{\gamma_d}.
\end{equation}
Equality in the overall bound
$\operatorname{st}_\infty(C)\leq\sqrt{\gamma_d}$ holds if and only if $m=1$
and $\Gamma_C^*$ is Hermite-extremal; in particular, $C$ is then a translate
of a lattice.
\end{corollary}

\begin{proof}
Intensity one gives
$\operatorname{covol}(\Gamma_C)=m$ and therefore
$\operatorname{covol}(\Gamma_C^*)=m^{-1}$.  Proposition~\ref{prop:periodic-full-reciprocal-spectrum} and the definition of $\gamma_d$
give
\[
  \operatorname{st}_\infty(C)
  =\lambda_1(\Gamma_C^*)
  \leq\sqrt{\gamma_d}\,
      \operatorname{covol}(\Gamma_C^*)^{1/d}
  =\sqrt{\gamma_d}\,m^{-1/d}.
\]
Since $m\geq1$, \eqref{eq:all-order-hermite-bound} follows.  If
$\operatorname{st}_\infty(C)=\sqrt{\gamma_d}$, equality holds in both
inequalities in \eqref{eq:all-order-hermite-bound}; hence $m=1$ and
$\Gamma_C^*$ is Hermite-extremal.  Conversely, these two conditions make both
inequalities equalities.

For a full-rank lattice $L$, its lattice sphere-packing density is
\begin{equation}\label{eq:lattice-packing-hermite}
  \Delta(L)
  =\frac{\Vol_d(B_1)}{2^d}
   \left(
     \frac{\lambda_1(L)^2}
          {\operatorname{covol}(L)^{2/d}}
   \right)^{d/2}.
\end{equation}
Equation~\eqref{eq:lattice-packing-hermite} identifies maximization of the
Hermite quotient with lattice sphere-packing density.  When $m=1$,
$\Gamma_C^*$ has covolume one, and every covolume-one lattice $L$ occurs as
$\Gamma_C^*$ by taking $C$ to be a translate of $L^*$.  When $m>1$, the factor
$m^{-1/d}<1$ is strict.  Hence the maximizers are precisely the translates of
$L^*$ with $L$ Hermite-extremal.
\end{proof}

\begin{corollary}[Sphere-packing bound for a common all-order gap]
\label{cor:common-gap-sphere-packing}
Let $\mu$ be an ergodic $\mathbb R^d$-invariant simple point process of
intensity one, supported on $r$-uniformly discrete configurations for some
$r>0$, and suppose that $\mu$ has a common all-order stealth gap.  Then
$\mu=\mu_C$ for a periodic configuration $C$, and if $m$ is the number of
points of $C$ in a fundamental domain of its full translation lattice, then
\[
  \operatorname{st}_\infty(\mu)
  =\operatorname{st}_\infty(C)
  \leq\sqrt{\gamma_d}\,m^{-1/d}
  \leq\sqrt{\gamma_d}.
\]
Equality in the overall bound holds exactly when $\mu$ is the
random-translate process of a lattice whose dual lattice is Hermite-extremal.
\end{corollary}

\begin{proof}
Proposition~\ref{prop:common-gap-periodic} gives $\mu=\mu_C$ for a periodic
$C$, and Corollary~\ref{cor:periodic-all-order-hermite} gives the asserted
bound and equality statement.  Conversely, every unit-intensity periodic
process $\mu_C$ has a common all-order gap: by
Proposition~\ref{prop:periodic-full-reciprocal-spectrum}, any
$\varepsilon<\lambda_1(\Gamma_C^*)$ is admissible.  Hence the class in the corollary coincides with the unit-intensity periodic
class.
\end{proof}

Cohen, Cohn and Viazovska proved that the isometry-invariant $E_8$ and
Leech processes uniquely maximize the first-order stealth radius among
isometry-invariant locally square-integrable point processes of intensity one
in dimensions $8$ and $24$ \cite[Theorem~1.1]{CCV26}.  A nontrivial periodic
basis may cancel the shortest reciprocal-lattice Bragg peaks.  By
Proposition~\ref{prop:periodic-full-reciprocal-spectrum}, such cancellations do
not persist through all finite orders.  In dimension $20$, the formal dual of
Vardy's packing in \cite[Section~5]{CCV26} is a union of sixteen translates of
$\Lambda_{20}$ whose unit-density first-order stealth radius exceeds that of
any known twenty-dimensional lattice.

A locally finite set $\Lambda\subset\mathbb R$ has \emph{finite local
complexity} in the sense of \cite{KL16} if it admits an increasing
bi-infinite enumeration $(\lambda_n)_{n\in\mathbb Z}$ for which the set of
successive gaps $\{\lambda_{n+1}-\lambda_n:n\in\mathbb Z\}$ is finite.
Equivalently, $\Lambda$ is Delone and
$(\Lambda-\Lambda)\cap[-R,R]$ is finite for every $R>0$.

\begin{proposition}[Finite-local-complexity rigidity]
\label{prop:flc-first-order-periodic}
Let $\mu$ be an ergodic $1$-stealthy point process of positive
intensity on $\mathbb R$, supported on $r$-uniformly discrete
configurations for some $r>0$.  If $\mu$-almost every configuration has
finite local complexity, then $\mu$ is supported on the translation orbit
of a periodic set.
\end{proposition}

\begin{proof}
Let $(-\varepsilon_0,\varepsilon_0)$ be a first-order spectral gap.  Uniform
discreteness makes $f\mapsto p(f)$ continuous on $\mathcal S(\mathbb R)$,
uniformly over the support of $\mu$, while
Lemma~\ref{lem:bartlett-translation-bounded} makes the first Bartlett measure
tempered.  Hence both sides of \eqref{eq:bartlett-defining-identity} extend
continuously from $C_c^\infty(\mathbb R)$ to $\mathcal S(\mathbb R)$.
Choose a nonempty open
interval $I\Subset(-\varepsilon_0,\varepsilon_0)\setminus\{0\}$ and
$f\in\mathcal S(\mathbb R)$ such that
\[
  \operatorname{supp}\widehat f
  \subset(-\varepsilon_0,\varepsilon_0)\setminus\{0\},
  \qquad \widehat f(-\xi)\neq0\quad(\xi\in I).
\]
Since $0\notin\operatorname{supp}\widehat f$, we have
$\widehat f(0)=0$; if $\rho$ denotes the intensity, stationarity therefore gives
$\int p(\tau_t f)\,d\mu(p)=\rho\widehat f(0)=0$.  The spectral gap then gives,
for every $t\in\mathbb R$,
\[
  \int |p(\tau_t f)|^2\,d\mu(p)=0.
\]
Taking first $t\in\mathbb Q$ and then
using continuity in $t$, we obtain for $\mu$-almost every
$p=\delta_\Lambda$ that $p(\tau_t f)=0$ for every $t\in\mathbb R$.
Since, with $\check f(x):=f(-x)$,
\[
  p(\tau_t f)=(\delta_\Lambda*\check f)(-t),
\]
this is equivalent to
\[
  (\delta_\Lambda*\check f)(t)=0\qquad(t\in\mathbb R).
\]
Uniform discreteness makes $\delta_\Lambda$ translation bounded and hence a
tempered distribution.  Fourier transformation in $\mathcal S'(\mathbb R)$
yields
$\widehat{\delta_\Lambda}(\xi)\widehat f(-\xi)=0$.  Since
$\widehat f(-\cdot)$ is smooth and nonvanishing on the open interval $I$,
division by this multiplier gives $\widehat{\delta_\Lambda}|_I=0$.  By
\cite[Theorem~5.1]{KL16}, finite local complexity then forces $\Lambda$ to be
periodic.

Thus almost every configuration is periodic.  Let
$X:=\operatorname{supp}\mu\subset\mathscr X_r$, which is compact by the
proof of Proposition~\ref{prop:common-gap-periodic}.  For $p\in X$, put
$\Gamma_p:=\{t\in\mathbb R:t.p=p\}$.  If $p\neq0$, uniform discreteness gives
$\Gamma_p\cap(-r,r)=\{0\}$, so $\Gamma_p$ is either trivial or of the form
$a\mathbb Z$ with $a\ge r$.  Define $\ell(p)$ to be this least positive period
when it exists, set $\ell(p)=\infty$ for nonzero aperiodic $p$, and set
$\ell(0)=0$.  For $0<a<r$ one has $\{\ell\le a\}=X\cap\{0\}$, while for
$a\ge r$ the set $\{\ell\le a\}$ is the projection of the compact set
$\{(s,p)\in[r,a]\times X:s.p=p\}$.  Thus $\ell$ is Borel and translation
invariant.  Ergodicity makes $\ell$ almost surely constant; positive intensity
excludes the value $0$, and the periodicity just proved excludes $\infty$.
Hence $\ell=a\in[r,\infty)$ almost surely.  The action then factors on the
full-measure closed set $X_a:=\{p:a.p=p\}$ through the compact group
$\mathbb R/a\mathbb Z$.  The push-forward of $\mu$ to the compact metrizable
orbit space is $0$--$1$ valued by ergodicity and hence is a point mass.  Thus
$\mu$ is supported on one periodic translation orbit.
\end{proof}

Remark~7.8 of \cite{Bjo26a} applies the same finite-local-complexity
obstruction to the specific nonperiodic stealthy realizations constructed
there.

\begin{proposition}[All-order stealthiness without periodicity]
\label{prop:all-order-diffuse-stealthy}
There exists an ergodic $\mathbb R$-invariant probability measure
$\mu$ on $\cM(\mathbb R)$, supported on nonperiodic absolutely continuous
measures with bounded densities bounded away from zero, such that $\mu$ is
$k$-stealthy for every $k\geq1$.  No positive spectral gap works simultaneously
for all $k$.
\end{proposition}

\begin{proof}
Fix an irrational $\alpha>0$ and put
\[
  h(\theta_1,\theta_2)
  :=1+\kappa\bigl(\cos(2\pi\theta_1)+\cos(2\pi\theta_2)\bigr),
  \qquad 0<\kappa<\frac12.
\]
For $\theta=(\theta_1,\theta_2)\in\mathbb T^2$, define
\begin{equation}\label{eq:all-order-diffuse-density}
  p_\theta(dx)
  :=h(\theta_1+x,\theta_2+\alpha x)\,d\Vol_1(x),
\end{equation}
and let
\[
  \mu:=(\theta\mapsto p_\theta)_*m_{\mathbb T^2},
\]
where $m_{\mathbb T^2}$ is Haar probability measure.  For
$t\in\mathbb R$,
\[
  p_{\theta+t(1,\alpha)}=(-t).p_\theta.
\]
Since $n_1+\alpha n_2\neq0$ for every
$n\in\mathbb Z^2\setminus\{0\}$, the Kronecker flow
$\theta\mapsto\theta+t(1,\alpha)$ is ergodic.  The map
$\theta\mapsto p_\theta$ intertwines this flow with translations, so $\mu$ is
ergodic.  Moreover
\[
  1-2\kappa\leq h\leq1+2\kappa,
\]
so all local moments are finite.

If $t\neq0$ were a period of $p_\theta$, then
\[
  h(\theta_1+x+t,\theta_2+\alpha(x+t))
  =h(\theta_1+x,\theta_2+\alpha x)
  \qquad(x\in\mathbb R).
\]
Taking the Bohr coefficients at the distinct frequencies $1$ and $\alpha$
gives
\[
  \frac{\kappa}{2}e^{2\pi i\theta_1}(e^{2\pi i t}-1)=0,
  \qquad
  \frac{\kappa}{2}e^{2\pi i\theta_2}(e^{2\pi i\alpha t}-1)=0.
\]
Thus $t,\alpha t\in\mathbb Z$, which is impossible for irrational $\alpha$.
Hence every $p_\theta$ is nonperiodic.

Fix $k\geq1$ and $\varphi\in\mathcal C_k$.  Write
$y_k=-\sum_{i=1}^{k-1}y_i$.  The Jacobian matrix of
\[
  (c,y_1,\ldots,y_{k-1})\longmapsto(x_1,\ldots,x_k),
  \qquad x_i=c+y_i,
\]
has first column $(1,\ldots,1)^\mathsf T$ and remaining columns
$e_i-e_k$; its determinant is $(-1)^{k-1}k$.  Hence
\[
  dx_1\cdots dx_k=k\,dc\,dy_1\cdots dy_{k-1}.
\]
Since the permutation group is finite, the quotient map from the centred
hyperplane $\{\sum_i y_i=0\}$ to $\Sh_k(\mathbb R)$ is proper; hence the
inverse image of $\operatorname{supp}\varphi$ is compact.  Therefore
\begin{equation}\label{eq:all-order-derived-density}
  \mathcal P_{k,\varphi}(p_\theta)(dc)
  =G_{k,\varphi}(\theta_1+c,\theta_2+\alpha c)\,d\Vol_1(c),
\end{equation}
where
\begin{align}\label{eq:all-order-torus-density}
  G_{k,\varphi}(\eta_1,\eta_2)
  :=k\int_{\mathbb R^{k-1}}
   &\varphi([y_1,\ldots,y_k])\nonumber\\
   &\times\prod_{i=1}^k
      h(\eta_1+y_i,\eta_2+\alpha y_i)
      \,dy_1\cdots dy_{k-1}.
\end{align}
The torus Fourier convention is
\[
  \widehat G(n):=\int_{\mathbb T^2}G(\eta)e^{2\pi i n\cdot\eta}
  \,dm_{\mathbb T^2}(\eta),
  \qquad
  G(\eta)=\sum_n\widehat G(n)e^{-2\pi i n\cdot\eta}.
\]
The Fourier support of $h$ is $\{0,\pm e_1,\pm e_2\}$.  Hence
$G_{k,\varphi}$ is a trigonometric polynomial and
\begin{equation}\label{eq:all-order-fourier-support}
  \operatorname{supp}\widehat G_{k,\varphi}
  \subset
  \{n=(n_1,n_2)\in\mathbb Z^2:\ |n_1|+|n_2|\leq k\}.
\end{equation}
The product in \eqref{eq:all-order-torus-density} has Fourier support in the
$k$-fold sumset of $\{0,\pm e_1,\pm e_2\}$, and integration in the shape
variables introduces no new torus modes.  Thus, for
$f\in C_c^\infty(\mathbb R;\mathbb C)$,
\[
  T_k[f,\varphi](p_\theta)-\int_{\mathbb T^2}T_k[f,\varphi](p_\eta)
  \,dm_{\mathbb T^2}(\eta)
  =\sum_{n\neq0}\widehat G_{k,\varphi}(n)
    e^{-2\pi i n\cdot\theta}\,
    \widehat f(n_1+\alpha n_2).
\]
Orthogonality of the torus characters gives the covariance identity in
\eqref{eq:bartlett-defining-identity}; uniqueness in
Proposition~\ref{prop:operator-bartlett} therefore gives
\begin{equation}\label{eq:all-order-diffuse-bartlett}
  \sigma^{(k)}_{\varphi,\varphi}
  =\sum_{n\in\mathbb Z^2\setminus\{0\}}
      |\widehat G_{k,\varphi}(n)|^2
      \delta_{n_1+\alpha n_2}.
\end{equation}
By \eqref{eq:all-order-fourier-support}, every nonzero frequency in
\eqref{eq:all-order-diffuse-bartlett} belongs to the finite set
\begin{equation}\label{eq:all-order-frequency-set}
  \{n_1+\alpha n_2:
    0<|n_1|+|n_2|\leq k\}.
\end{equation}
Since $\alpha$ is irrational, this set does not contain zero.  Hence
\begin{equation}\label{eq:all-order-gap-k}
  \varepsilon_k
  :=\frac12\min_{0<|n_1|+|n_2|\leq k}|n_1+\alpha n_2|>0
\end{equation}
is a spectral gap for every diagonal coefficient at order $k$.
Positivity and the Cauchy--Schwarz inequality for the coefficient measures
\eqref{eq:bartlett-measure-cauchy-schwarz} then give
$\boldsymbol\sigma^{(k)}(B_{\varepsilon_k}^*)=0$, so $\mu$ is
$k$-stealthy.

To exclude a gap uniform in $k$, fix $n=(n_1,n_2)\neq0$ and take
$k\geq|n_1|+|n_2|$.  For
$y=(y_1,\ldots,y_k)$ with $\sum_i y_i=0$, let $C_{k,n}(y)$ be the $n$-th
torus Fourier coefficient, with the convention above, of
\[
  \eta\longmapsto
  \prod_{i=1}^k h(\eta_1+y_i,\eta_2+\alpha y_i).
\]
At $y=0$, this is the $n$-th Fourier coefficient of $h^k$.  The Fourier
coefficients of $h$ at $0,\pm e_1,\pm e_2$ are respectively
$1,\kappa/2,\kappa/2,\kappa/2,\kappa/2$, all positive.  Writing $n$ as a sum
of $|n_1|$ copies of $\operatorname{sgn}(n_1)e_1$ and $|n_2|$ copies of
$\operatorname{sgn}(n_2)e_2$, with the remaining summands equal to zero,
produces a strictly positive term in the convolution formula for the
$n$-th coefficient of $h^k$.  All other terms are nonnegative, so
\[
  C_{k,n}(0)>0.
\]
The function $C_{k,n}$ is continuous and symmetric in the variables $y_i$.
Thus
\[
  \{y:\operatorname{Re}C_{k,n}(y)>0\}
\]
descends to an open neighbourhood of the zero shape.  Choose a nonzero
$0\leq\varphi\in C_c(\Sh_k(\mathbb R))$ supported there.  Taking the $n$-th
coefficient in \eqref{eq:all-order-torus-density} gives
\[
  \operatorname{Re}\widehat G_{k,\varphi}(n)
  =k\int_{\mathbb R^{k-1}}
     \varphi([y])\operatorname{Re}C_{k,n}(y)\,dy>0.
\]
Hence \eqref{eq:all-order-diffuse-bartlett} has an atom at
$n_1+\alpha n_2$.  Given $\varepsilon>0$, choose $n\neq0$ with
$0<|n_1+\alpha n_2|<\varepsilon$, which is possible because
$\mathbb Z+\alpha\mathbb Z$ is dense in $\mathbb R$ for irrational $\alpha$,
and put $k=|n_1|+|n_2|$.  For the shape
test above,
\[
  \sigma^{(k)}_{\varphi,\varphi}(B_\varepsilon^*)>0.
\]
No positive spectral gap therefore works at every order.
\end{proof}

\begin{remark}[An open all-order rigidity problem]
\label{rem:all-order-open-question}
Proposition~\ref{prop:common-gap-periodic} and
Proposition~\ref{prop:all-order-diffuse-stealthy} leave the following question.
If an ergodic uniformly discrete point process of positive intensity on
$\mathbb R$ is $k$-stealthy for every finite $k$, with the gap allowed
to depend on \(k\), must it be periodic?  Proposition~\ref{prop:flc-first-order-periodic} answers this affirmatively when the
configurations have finite local complexity, already under $1$-stealthiness.  Proposition~\ref{prop:all-order-diffuse-stealthy} shows that
the corresponding statement fails without the point-process assumption.
\end{remark}

\section*{Statements and Declarations}

\paragraph{\textbf{Funding.}}
This work was supported by the Swedish Research Council under grant
VR 11253322.

\paragraph{\textbf{Competing interests.}}
The author has no relevant financial or non-financial interests to disclose.

\paragraph{\textbf{Data availability.}}
No datasets were generated or analysed during the current study.

\paragraph{\textbf{AI-assisted preparation.}}
During preparation of the manuscript, the author used ChatGPT (OpenAI) for
editorial assistance, consistency checking, and mathematical cross-checking,
and Aristotle (Harmonic) for automated manuscript and proof auditing.  All
AI-assisted output was reviewed and verified by the author, who takes full
responsibility for the mathematical content, references, and final text.


\end{document}